\documentclass[12pt,a4paper]{article}

\usepackage[dvips]{color}
\usepackage{color}
\usepackage{xcolor}
\usepackage{mathrsfs}
\usepackage{amsfonts}
\usepackage{mathtools}
\usepackage{stmaryrd}
\usepackage{amsmath}
\usepackage{amsthm}
\usepackage{amssymb}
\usepackage{bbm}
\usepackage{graphicx}
\usepackage{enumerate}

\makeatletter
\def\tank#1{\protected@xdef\@thanks{\@thanks
        \protect\footnotetext[0]{#1}}}
\def\bigfoot{

    \@footnotetext}
\makeatother

\newcommand{\DEQS}{\begin{eqnarray*}}
\newcommand{\EEQS}{\end{eqnarray*}}
\newcommand{\DEQZ}{\begin{eqnarray}}
\newcommand{\EEQZ}{\end{eqnarray}}

\theoremstyle{plain}

\newtheorem{thm}{Theorem}[section]
\newtheorem{lem}{Lemma}[section]
\newtheorem{prop}{Proposition}[section]

\newtheoremstyle{boldremark}
    {\dimexpr\topsep/2\relax} 
    {\dimexpr\topsep/2\relax} 
    {}          
    {}          
    {\bfseries} 
    {.}         
    {.5em}      
    {}          

\theoremstyle{boldremark}
\newtheorem{definition}{Definition}[section]
\newtheorem{example}{Example}[section]

\newtheorem{ass}{Assumption}[section]
\newtheorem{rmk}{Remark}[section]

\numberwithin{equation}{section}
\renewenvironment{proof}{{\bfseries Proof}}{\qed}

\def\RR{\mathbb{R}}
\def\PP{\mathbb{P}}

\def\EE{\mathbb{E}}
\def\NN{\mathbb{N}}

\def\cE{{\mathcal E}}

\def\de{{\delta}}

\def\si{{\sigma}}

\def\et{{\eta}}

\def\al{{\alpha}}

\def\Ga{{\Gamma}}
\def\ga{{\gamma}}
\def\de{{\delta}}

\def\si{{\sigma}}

\def\eps{{\vare}}
\def\vare{{\varepsilon}}

\def\EE{\mathbb{ E}}

\def\si{{\sigma}}
\def\al{{\alpha}}

\allowdisplaybreaks

\def\dt{\, dt}

\begin{document}

\baselineskip 15.35pt
\numberwithin{equation}{section}
\title
{{Uniform large deviations and metastability of random dynamical systems }
}

\date{}

\author{{Jifa Jiang}$^1$\footnote{E-mail:jiangjf@shnu.edu.cn}~~~{Jian Wang}$^2$\footnote{E-mail:wg1995@mail.ustc.edu.cn}~~~ {Jianliang Zhai}$^3$\footnote{E-mail:zhaijl@ustc.edu.cn}~~~ {Tusheng Zhang}$^{4}$\footnote{E-mail:Tusheng.Zhang@manchester.ac.uk}
\\
\small 1. Mathematics and Science College, Shanghai Normal University, Shanghai 200234, China \\
\small 2. School of Mathematics, Hangzhou Normal University, Hangzhou 311121,  China \\
 \small  3. School of Mathematical Sciences,
 \small  University of Science and Technology of China,\\
 \small  Hefei 230026, China.\\
  \small 4. Department of Mathematics, University of Manchester, Oxford Road, Manchester, M13 9PL, UK
}

\maketitle	

\newcommand\blfootnote[1]{%
\begingroup
\renewcommand\thefootnote{}\footnote{#1}%
\addtocounter{footnote}{-1}%
\endgroup
}

\begin{center}
\begin{minipage}{130mm}
{\bf Abstract:}
In this paper, we first provide a criterion on uniform large deviation principles (ULDP) of stochastic differential equations under Lyapunov conditions on the coefficients, which can be applied to stochastic systems with coefficients of polynomial growth and possible degenerate driving noises. In the second part, using the ULDP criterion we preclude the concentration of limiting measures of invariant measures of stochastic dynamical systems on repellers and acyclic saddle chains and extend Freidlin and Wentzell's asymptotics theorem to stochastic systems with unbounded coefficients. Of particular interest, we determine the limiting measures of the invariant measures of the famous stochastic van der Pol equation and  van der Pol Duffing equation whose noises are naturally degenerate. We also construct two examples to match the global phase portraits of Freidlin and Wentzell's unperturbed systems and to explicitly  compute their transition difficulty matrices. Other applications include stochastic May-Leonard system and random systems with infinitely many  equivalent classes.

\vspace{3mm} {\bf Keywords:}
Uniform large deviations; Lyapunov conditions; invariant measure; asymptotic measure; stochastic van der Pol equation

\vspace{3mm} {\bf AMS Subject Classification:}
60B10; 60F10; 60H10; 37A50, 37C70.

\end{minipage}
\end{center}

\newpage

\renewcommand\baselinestretch{1.2}
\setlength{\baselineskip}{0.28in}

\section{Introduction}
Let $b$ be a measurable vector field on $\RR^d$. Consider the dynamical system
\begin{equation}\label{0.1}
dX(t)=b(X(t))\dt,\quad t\ge 0, \quad X(0)=x \in \RR^d,
\end{equation}
and its random perturbation
\begin{equation}\label{11}
dX^\vare(t)=b(X^\vare(t))\dt+\sqrt{\vare}\sigma(X^\vare(t))dB(t),\quad t\ge 0, \quad X^\vare(0)=x \in \RR^d,
\end{equation}
where $\{B(t)\}_{t \geq 0}$ is an \emph{m}-dimensional Brownian motion on a complete probability space $(\Omega,\mathcal{F},\{{\mathcal{F}}_t\}_{t\geq 0},\mathbb{P})$   with a filtration $\{{\mathcal{F}}_t\}_{t\geq 0}$ satisfying the usual conditions, $\sigma(\cdot): \RR^d\rightarrow \RR^d\otimes \RR^m$ is a measurable mapping, and $\vare$ is a strictly positive constant.

The deterministic system (\ref{0.1}) usually has several invariant measures. It is well-known that the noise often stabilizes the deterministic system, which means that the perturbed system (\ref{11}) has a unique invariant measure, denoted by $\mu^{\varepsilon}$. In order to understand how the random perturbations  influence the  behavior of the deterministic dynamical system, one of the fundamental problems is to identify  which invariant measure of (\ref{0.1}) $\mu^{\varepsilon}$ converges to as $\varepsilon$ tends to 0, which is the famous asymptotic measure problem proposed by Kolmogorov, see \cite{Sinai 1989}.

Except for potential systems (e.g., \cite{Hwang1}) and monostable systems (e.g., \cite{Khasminskii, FW, XCJ}),  it is extremely
hard (if not impossible) to describe precisely the limit measure $\mu$ of the family $\mu^{\varepsilon}$. Instead, people try  to identify where the support of the limiting measure $\mu$ is. This is equally non-trivial.
 The existing results \cite{FW1, FW, Hwang1,XCJ} state that the limiting measures are concentrated on  stable limit sets of (\ref{0.1}) under the assumptions that coefficients are bounded, the driving noise is non-degenerate
 (see, e.g., (${ \bf A}_1$) or (${ \bf A}_2$), and (${ \bf A}_4$) below). However, many important models from physics are nonlinear systems with polynomial coefficients, e.g., the Lorentz system \cite{Lorenz}; the R\"{o}ssler system \cite{Rossler};  the Chua circuit \cite{Chua} and the FitzHugh-Nagumo system \cite{FN} et al, and also all randomly forced damping Hamiltonian systems originating from Langevin dynamics naturally have degenerate noise, which include the well-known stochastic  van der Pol equation and  van der Pol Duffing equation. The purpose of this paper is to provide a framework and general results on the support of limiting measures under mild conditions, which in particular can be applied to the above interesting models.

We now give some details about the existing results which are relevant to the current paper.
Freidlin and Wentzell \cite{FW1, FW} introduced the notion of a uniform large deviation principle (ULDP) and revealed  the close connections to the long-time influence of small random perturbations of deterministic dynamical systems, especially to the support of the limiting measures $\mu$. Let us recall the conditions introduced in \cite{FW}.\\
(${\bf A}_1$). The coefficients  $b$ and $\sigma$ are bounded, locally Lipschitz continuous and uniformly continuous on $\mathbb{R}^d$  and there exists a positive constant $\lambda>0$ such that
$$\beta^*\sigma(x)\sigma^{*}(x)\beta\ge \lambda |\beta|^2\ {\rm for\ all}\ \beta, x\in \mathbb{R}^d.$$
(${\bf A}_2$). $d=m$, $\sigma$ is an identity matrix and the drift $b$ is globally Lipschitz continuous.
\vskip 0.3cm
Under the assumption (${\bf A}_1$) or (${\bf A}_2$),  Freidlin and Wentzell proved that the family of the laws of the solution processes  $\{X^{\varepsilon,x}\}$ of stochastic differential equations (SDEs) (\ref{11})
satisfies a ULDP with respect to the initial value $x\in \mathbb{R}^d$; see \cite [Theorem 3.1, p.135] {FW} etc.. Since then, when people apply ULDP to study  the metastable behavior of stochastic dynamical systems, they usually impose the condition  (${\bf A}_1$) or (${\bf A}_2$) (see \cite{FW2,E, Vanden, FW3, FW4, Olivieri}).
\vskip 0.3cm
To describe the previous results on the support of the limiting measures of the family of the invariant measures of  SDEs (\ref{11}), we introduce two further conditions.\\
(${\bf A}_3$). There are $\vare_0, \gamma>0$ and a nonnegative $C^2$ Lyapunov function $V$ with $\lim_{|x|\rightarrow \infty}V(x)= \infty$  such that
$$ \frac{\vare}{2}{\rm Trace}\left(\sigma^{*}(x)\nabla^2V(x)\sigma(x)\right)+ \langle b(x),\nabla V(x)\rangle\le -\gamma\ {\rm for}\ \ |x|\gg1, \ 0<\vare \le \vare_0.$$
(${\bf A}_4$). The deterministic  system (\ref{0.1}) admits a finite number of compact maximal equivalent  classes (defined in terms of the quasi-potential associated with the rate function of ULDP, see Section {3})  $K_1, K_2, \cdots, K_l$ such that  the set of limit points of any trajectory of the dynamical system (\ref{0.1}) (the so called $\omega-$limit set) is contained in one of the $K_i'$s.
\vskip 0.3cm
Under the assumptions (${\bf A}_1$), (${\bf A}_3$), (${\bf A}_4$),
Freidlin and Wentzell set up a procedure to determine on which stable equivalent classes the limiting measures of the invariant measures $\mu^\vare$ of SDEs (\ref{11})  will be supported  by estimating the escape
time of the solutions $X^\vare$ from a domain and  the so called transitive difficulty matrix $(\mathbb{V}(K_i,K_j)_{l\times l}$ (see the definition in Section {3}, and also  \cite [Theorems 4.1 and 4.2]  {FW} of Chapter 6). These results are now called the Freidlin and Wentzell asymptotics theorem. They imply that the support of the limiting  measures of the invariant measures $\mu^\vare$ stays away from the so called repelling equivalent classes $K_i'$s and acyclic saddle or trap chains $K_i'$s.  Xu et al. \cite{XCJ} proved that the above implied results still hold under (${\bf A}_1$) and (${\bf A}_3$) but without the finiteness  of the number of  compact equivalent classes.
Assuming that the diffusion matrix is uniformly elliptic, imposing the conditions (${\bf A}_3$) and (${\bf A}_4$) and using Freidlin and Wentzell's ULDP technique,  Hwang and Sheu \cite{Hwang} studied the long time behavior and the exponential rate of convergence to the invariant measures of the system (\ref{11}).
\vskip 0.3cm

Now we highlight the main contributions of the current work.

$\bullet$ ULDP under mild conditions and improvement of the Freidlin and Wentzell asymptotics theorem.
The ULDPs are essential to the study of the metastable problems of small random perturbation of the dynamical system (\ref{0.1}). The current conditions (${\bf A}_1$) or (${\bf A}_2$)  for ULDP to hold  are rather restrictive, which excludes many interesting models where the  coefficients of the systems are unbounded and the diffusion coefficients could be degenerate. In the first part of the paper, we obtain a general result on ULDPs of SDEs under Lyapunov conditions on the coefficients, which can be applied to stochastic systems with unbounded coefficients and degenerate diffusions, and significantly extends results in literature. This part of the work is of independent interest and is also essential to the study of the support of the limiting measures. As a result, we improve the
Freidlin and Wentzell asymptotics theorem et al. so that they apply to systems with unbounded coefficients.

$\bullet$ Computation of the transitive difficulty matrix $\left(\mathbb{V}(K_i,K_j)\right)_{4\times 4}$ of two quasipotential planar polynomial systems. By the transitive difficulty matrices and our improved Freidlin and Wentzell's asymptotics theorem,  we prove that the invariant measure $\mu^{\varepsilon}$ of the random  planar dynamical systems converges weakly to the arc-length measure of stable limit cycles as $\varepsilon\rightarrow 0$. In the literature, when considering limiting measures and their concentrations, people always first draw  a  picture of global phase portraits and then design  a transitive difficulty matrix, instead of first studying the global phase portraits of the given system (\ref{0.1}) and then calculating  the transitive difficulty matrix $\left(\mathbb{V}(K_i,K_j)\right)_{l\times l}$ of the given system (\ref{11}), see e.g., {\cite[p.151]{FW}} and \cite[p.557]{FW2}.
    It seems to us this is the first time to compute the transitive difficulty matrix explicitly from  original given SDEs (\ref{11}).

$\bullet$ Characterization of the limit measures of the invariant measures of the van der Pol equation and the van der Pol Duffing equation perturbed by unbounded random noise. We  prove that the invariant measures of the random van der Pol equation converge weakly to the arc-length measure corresponding to the unique limit cycle, and that the invariant measures of the random van der Pol Duffing equation converge weakly to either the convex combination of Delta measures of the stable equilibria, or  the convex combination of Delta measures of the stable equilibria and the arc-length measure of stable limit cycle. We stress here that, due to the degenerate driving noise, the Freidlin and Wentzell asymptotics theorem is not applicable to them.

     Both models are damping Hamiltonian systems or Langevin dynamics. The van der Pol equation describes a self-oscillating triode circuit and the van der Pol Duffing equation  describes single diode circuit (see \cite{Moser,Smale1972}). The study of these models goes back as early as 1927s when van der Pol \cite{van} discovered  an ``irregular noise" in a diode subject to periodic forcing and also  the coexistence of periodic orbits of different periods, which is the first experimental observation of deterministic chaos. Through the years, the study of the dynamical behavior of the van der Pol equation gave rise the birth of the well-known Smale horseshoe, even differential dynamical systems, see \cite {Smale1998, Lev, CL1945, CL1951}.  Because  the random  van der Pol equation has highly degenerate driving noise and polynomial coefficients, to characterize the limit measures of the invariant measures, new difficulties occur and new ideas are needed.

    To obtain the existence of invariant measures for the random van der Pol equations, {because the standard energy functionals are not applicable,} we use a carefully designed Lyapunov function (see (\ref{LyaF})), which not only
     allows polynomial driving noise, but also help us obtain the ULDP  and prove that  both  stochastically and periodically forced van der Pol equation admits a periodic stationary distribution, see Remark \ref{rmk4.1}. The other challenge we encounter is the continuity of its quasipotential in these systems with degenerate driving noise. Unlike the non-degenerate case, what we have only been able to show is that  the quasipotential is upper-continuous.
      Even this is not easy, it involves carefully constructing sample orbits connecting  any pair of states, which are required to satisfy certain restrictions in order to design good controls of rate functions of the corresponding random systems; see (\ref{eq lemma 3.1 1004}), (\ref{eq h 1.1}), (\ref{eq control 000}), (\ref{eq h 2.1}), and (\ref{eq control 001}).
      The upper-continuity of the  quasipotential implies that it is continuous at any equilibrium, see Lemma \ref{PropQ 0808}. Combined with the estimate of rare probability of a small neighborhood of the unstable equilibrium via ULDP, this turns out to be sufficient to characterize the limiting measures. It is somehow surprising.
    Due to the degenerate driving noise and polynomial growth of the coefficients, some other technical difficulties have also to be overcome, see Lemma \ref{PropQ 0808} and Subsection 4.2 for the details.
    Finally, we point out that our method used in this part is valid for any single degree of freedom
stochastic damping Hamiltonian systems because Lemma \ref{PropQ 0808} is effective to all second-order stochastic equation.

\vskip 0.3cm
The paper is organized as follows. Section 2 establishes a general result for  Freidlin and Wentzell's ULDP, which is essential for later sections.  In Section 3, we present new results on  the support of the limiting measures of the invariant measures of the system  (\ref{11}), including the  generalization of the results in \cite{FW}, \cite{XCJ} and \cite{Hwang} to the case of unbounded coefficients. Section 4 is devoted to applications. We give several examples of random perturbations of dynamical  systems with unbounded coefficients, including the realization of Freidlin's two phase portraits and the computation of their transitive difficulty matrices; stochastic van der Pol equation and stochastic van der Pol Duffing equation; stochastic May-Leonard system and an example of systems with unbounded  coefficients and infinitely many  equivalent classes.

\section{Uniform large deviation principles}
In this section, we will establish a ULDP for the laws of the solutions of the stochastic differential equations (\ref{11}) under some mild conditions. Let us now  recall the notion of ULDP from \cite{FW}. Let $(\cE, \rho)$ be a Polish space and  let $\cE_0$ be some topological space that will be used for indexing. We say  a function $I : \cE \rightarrow [0, +\infty]$ is a rate function if it has a compact level sets; i.e., for all $M <\infty$, $\{\varphi : I(\varphi) \leq M\}$ is compact. For any  $\varphi \in \cE$ and $K \subset \cE$, we denote
\[
{\rm dist}(\varphi, K) \doteq \inf_{\varphi \in K} \rho(\varphi, \varphi).
\]

For $x \in \cE_0$, we write $I_x(\cdot)$ to indicate the dependence of the rate function on the value $x$.  Let $\mathscr{K}$ be a collection of all compact subsets of $\cE_0$ and $\Phi_x(s) \doteq \{\varphi : I_x(\varphi) \leq s\}$.

\begin{definition}
A family of rate functions $\{I_x\}_{x \in \cE_0}$ on $\cE$ has compact level sets on compact sets  of $\cE_0$ if, for any $K \in \mathscr{K}$ and for every $s \in [0, \infty)$, $\cup_{x \in K} \Phi_x(s)$ is a compact subset of $\cE$.

\end{definition}

\begin{definition}[Freidlin-Wentzell ULDP]
  A family of $\cE$-valued random variables $\{X^{\vare,x}\}_{\vare>0}$ indexed by $x \in \cE_0$ is said to satisfy a  ULDP with respect to the rate function $I_x,\  x \in \cE_0$, uniformly over $\mathscr{K}$, if

(i) LDP lower bound:  For any $K \in \mathscr{K}, \de>0, \ga>0$ and $s_0>0$, there exists $\vare_0>0$ such that
\begin{equation}
\PP\big(\rho(X^{\vare,x}, \varphi) < \de\big) \geq {\rm exp}\big(-\vare^{-1}(I_x(\varphi) + \ga)\big),
\end{equation}
for all $\vare \in (0, \vare_0]$, $x \in K$ and $\varphi \in \Phi_x(s_0)$;

(ii) LDP upper bound:  For any $K \in \mathscr{K}, \de>0, \ga>0$ and $s_0>0$, there exists $\vare_0>0$ such that
\begin{equation}
\PP\big({\rm dist}(X^{\vare,x}, \Phi_x(s)) \geq \de\big) \leq {\rm exp}\big(-\vare^{-1}(s -\ga)\big),
\end{equation}
for all $\vare \in (0, \vare_0]$, $s \leq s_0$ and $x \in K$.

\end{definition}

The uniform Laplace principle was introduced in  Definition 1.11 in \cite{dupuis-2}.

\begin{definition}
  A family of $\cE$-valued random variables $\{X^{\vare,x}\}_{\vare>0}$ indexed by $x \in \cE_0$ is said to satisfy a uniform Laplace principle  with respect to the rate function $I_x,\  x \in \cE_0$, uniformly over $\mathscr{K}$, if for any $K \in \mathscr{K}$ and any bounded, continuous $h : \cE \rightarrow \RR$,
\begin{equation}
\lim_{\vare \rightarrow 0 }\sup_{x \in K} \Big| \vare {\rm log} \EE{\rm exp }(-\vare^{-1}h(X^{\vare,x})) + \inf_{\varphi \in \cE}\{h(\varphi) + I_x(\varphi)\}\Big| = 0.
\end{equation}

\end{definition}

\subsection{The main result on ULDP}

Throughout,  we will use the following notation. Let $(\mathbb{R}^d, \langle \cdot, \cdot \rangle, |\cdot|)$ be the \emph{d}-dimensional Euclidean space with the inner product $\langle \cdot, \cdot \rangle$ which induces the norm $|\cdot|$. The norm $\| \cdot \|$ stands for the Hilbert-Schmidt norm $\| \sigma \|^2 :=  \sum_{i=1}^{d} {\sum_{j=1}^{m}{\sigma_{ij}^2}}$ for any $d \times m$-matrix $\sigma = (\sigma_{ij}) \in \mathbb{R}^d \otimes \mathbb{R}^m$. $\sigma^*$ stands for the transpose of the matrix $\sigma$.

Let $(\Omega,\mathcal{F},\{{\mathcal{F}}_t\}_{t\geq 0},\mathbb{P})$  be a complete probability space with a filtration $\{{\mathcal{F}}_t\}_{t\geq 0}$ satisfying the usual conditions and $\{B(t)\}_{t \geq 0}$  an \emph{m}-dimensional Brownian motion on this probability space. Fix $T\in(0,\infty)$. Consider the following stochastic differential equations:
\begin{equation}\label{12}
dX^\vare(t)=b(X^\vare(t))\dt+\sqrt{\vare}\sigma(X^\vare(t))dB(t),\ t\in[0,T], \quad x(0)=x \in \RR^d,
\end{equation}
where $\sigma : \mathbb{R}^d \ni x \mapsto \sigma (x) \in \mathbb{R}^d\otimes \mathbb{R}^m$ and $ b:  \mathbb{R}^d\ni x \mapsto b(x) \in \mathbb{R}^d$ are continuous. In the following, we use the notation $X^{\vare,x}$ to indicate the solution of (\ref{12}) starting from $x$.

Let us now introduce the following assumptions.

\begin{ass}\label{ass1}
Let $\vare_0\in(0,1)$. For arbitrary $R>0$, if $|x|\vee|y|\leq R$, there exists $L_R>0$ such that the following locally  monotonicity condition
\begin{equation}\label{21}
2\langle x-y,b(x)-b(y)\rangle+\|\sigma(x)-\sigma(y)\|^2\leq L_R|x-y|^2
\end{equation}
holds for $|x-y|\leq \vare_0$.
\end{ass}
Note that Assumption \ref{ass1} holds if $b$ and $\sigma$ are locally Lipschitz continuous.
\begin{ass}\label{ass2}
There exist a Lyapunov function $V\in C^2(\mathbb R^d;\mathbb R_+)$  and $\theta>0,~\eta>0$ such that
\begin{equation}\label{As b V}
\lim_{|x|\rightarrow+\infty}V(x)=+\infty,
\end{equation}
\begin{equation}\label{22}
J(x):=\langle b(x),\nabla V(x)\rangle+\frac{\theta}{2}\text{Trace}\big(\sigma^{*}(x)\nabla^2 V(x)\sigma(x)\big)+\frac{|\sigma^{*}(x) \cdot \nabla V (x)|^2}{\eta V(x)}\leq C\big(1+V(x)\big),
\end{equation}
and
\begin{equation}\label{23}
\text{Trace}\big(\sigma^{*}(x)\nabla^2 V(x)\sigma(x)\big)\geq -M -CV(x).
\end{equation}
\end{ass}
\noindent Here $\nabla V$ and $\nabla^2 V$ stand for the gradient vector and Hessian matrix of the function
 $V$, respectively, $C,\ M > 0$ are some fixed constants.

The next result gives the existence and uniqueness of the solution of SDE \eqref{12}. Its proof is classical. The existence and uniqueness of a local solution can be established using the Assumption \ref{ass1} and the continuity of $b$ and $\si$.  Furthermore, one can show that the solution is global using the Lyapunov function $V$, \eqref{As b V} and \eqref{22} (cf. \cite [Theorem 3.5, p.75]{KhasminskiiB}).
\begin{prop}\label{prop1}
For any $0<\vare<1$, under Assumptions    {\rm \ref{ass1}} and {\rm \ref{ass2}}, there exists a unique  solution to Eq. \eqref{12} defined on $[0, +\infty)$.
\end{prop}
For each $h\in L^2([0,T],\mathbb R^m)$, consider the so called skeleton equation:
\begin{equation}\label{30}
dX_x^h(t)=b(X_x^h(t))dt+\sigma(X_x^h(t))h(t)dt,
\end{equation}
with the initial value $X_x^h(0)=x$.
We have the following result:
\begin{prop}\label{prop2}
Under Assumptions  {\rm\ref{ass1}} and {\rm\ref{ass2}}, there exists a unique solution to Eq. \eqref{30}.
\end{prop}
The proof of this proposition is similar to that of Proposition \ref{prop1}, so we omit it here.

\vskip 0.3cm
For any $x \in \RR^d$ and $f \in C([0,T],\mathbb{R}^d)$, we define
\begin{equation}\label{rate-0}
I_x(f)=\inf_{\left\{h \in L^2([0,T]; \mathbb{R}^m):\  f=X_x^h \right\}}\left\{\frac{1}{2}
\int_0^T|h(s)|^2ds\right\},
\end{equation}with the convention $\inf\{\emptyset\}=\infty$,  here $X_x^h\in C([0,T],\mathbb{R}^d)$ solves Eq. (\ref{30}).

We now formulate the main result on ULDP.

\begin{thm}\label{thmULDP}
For $\vare>0$, let $X^{\vare,x}$ be the solution to Eq. \eqref{12}. Suppose Assumptions {\rm \ref{ass1}} and {\rm\ref{ass2}}  are satisfied, then $I_x$ defined by \eqref{rate-0} is a rate function  on $C([0,T],\mathbb{R}^d)$ and the  family $\{I_x, x \in \RR^d\}$ of rate functions has compact
level sets on compacts. Furthermore, $\{X^{\vare,x}\}_{\vare>0}$ satisfies a ULDP on the space $C([0,T],\mathbb{R}^d)$ with the rate function $I_x$, uniformly over the initial value $x$ in bounded subsets of $\RR^d$.
\end{thm}
\begin{rmk}\label{rmk2.1}
Note that  Freidlin and Wentzell's criteria on ULDP  valid for bounded coefficients (${\bf A}_1$) (see \cite [Theorem 3.1, p.135] {FW}) or linear growth drift (${\bf A}_2$) (see \cite [Theorem 1.1, p.86] {FW} ) are corollaries of Theorem \ref {thmULDP} with the Lyapunov function $V(x)=|x|^2+1$.

For the study of the asymptotics of the invariant measures of the diffusion process $\{X^{\varepsilon,x}\}$,  Freidlin  assumes the dissipative condition (see \cite[p.556]{FW2} and \cite[p.110]{FW}):
\begin{equation}\label{FDissiP}
\langle b(x),x\rangle \le -\gamma |x|,\ {\rm for}\ \gamma>0,\ |x|\gg 1.
\end{equation}
In this case, the above Lyapunov function is mostly used in Freidlin--Wentzell's proof of  ULDP in $\RR^d$ (see also subsection 5.2 in \cite{XCJ} and (2.3) in \cite{Hwang}).

\end{rmk}
The proof of Theorem \ref{thmULDP} will be given in subsection 2.3 below.

In the sequel, the symbol $C$ will denote a positive generic constant whose value may change from place to place.

\subsection{A Sufficient Condition for ULDP}
In this section we recall the criteria obtained in \cite{dupuis-1} for proving the uniform Laplace principle. and we will provide a sufficient condition to verify  the criteria.

Let $\cE = C([0,T],\mathbb{R}^d)$. $\rho(\cdot,\cdot)$ stands for the uniform metric in the space $\cE$ and $\cE_0 =\RR^d$. 
Recall that $\mathscr{K}$ is a collection of all compact subsets of $\RR^d$.

The following result shows that a uniform Laplace principle implies the ULDP. Its proof can be found in \cite[Proposition 14]{dupuis-2} and \cite{SA, SBD}.
\begin{prop}\label{eqprop}
Let $I_x$ be a family of rate functions
on $\cE$ parameterized by $x$ in $\cE_0$ and assume that this family has compact level sets on compacts. Suppose that the family of $\cE$-valued random variables $\{X^{\vare, x}\}_{\vare>0}$ satisfies a uniform Laplace principle  with  rate function $I_x$ uniformly over any element in $\mathscr{K}$. Then  $\{X^{\vare, x}\}_{\vare>0}$ satisfies a uniform  LDP with rate function $I_x$ uniformly over any element in $\mathscr{K}$.
\end{prop}

Let
\begin{equation*}
S^N :=\{h\in L^2([0,T],\mathbb R^m):|h|_{L^2([0,T],\mathbb R^m)}^2\leq N\},
\end{equation*}
and
\begin{equation*}
\tilde{S}^N :=\{\phi : \phi \ \text{is} \ \mathbb{R}^m\text{-valued} \  {\mathcal{F}}_t\text{-predictable\  process\  such\ that}\  \phi(\omega) \in S^N,\  \mathbb{P}\text{-}a.s.\}.
\end{equation*}
$S^N$ will be  endowed with the weak topology on $L^2([0,T],\mathbb R^m)$, under which $S^N$ is a compact Polish space.

For any $\vare >0$, let $\Ga^{\vare}: \cE_0 \times C([0,T],\mathbb{R}^m) \rightarrow \cE$ be a measurable mapping. Set $X^{\vare, x} := \Ga^{\vare}(x, B(\cdot))$.
The following result was proved in \cite{dupuis-1}.
\begin{thm}[A Criteria of Budhiraja-Dupuis] \label{BD-criteria}
Suppose that there exists a measurable map $\Ga^{0} : \cE_0 \times C([0,T],\mathbb{R}^m) \rightarrow \cE$ and let
\begin{equation}\label{rate}
I_x(f)=\inf_{\left\{h \in L^2([0,T]; \mathbb{R}^m):\  f=\Ga^0(x, \int_0^{\cdot} h(s)ds) \right\}}\big\{\frac{1}{2}
\int_0^T |h(s)|^2ds\big\}.
\end{equation}
Suppose that for all $f \in \cE$, $x \mapsto I_x(f)$ is a lower semi-continuous (l.s.c.) map
from $\cE_0$ to $[0,\infty]$ and the following conditions hold:

{\rm (a)} for every $N<\infty$ and $K \in \mathscr{K}$, the set
\[
\Lambda_{N,K} \doteq \big\{ \Ga^0(x, \int_0^{\cdot} h(s)ds): h \in S^N, x \in K\big\}
\]
is a compact subset of $\cE$;

{\rm (b)} for every $N<\infty$ and any families $\{h^{\vare}\} \subset \tilde{S}^N$ and $\{x^{\vare}\} \subset \cE_0$ satisfying that $x^{\vare} \rightarrow x$ and  $h^{\vare}$ converges in law to some element $h$ as $\vare \rightarrow 0$, $\Ga^{\vare}\big(x^{\vare}, B(\cdot) + \frac{1}{\sqrt{\vare}}\int_0^{\cdot}{h^{\vare}(s)ds}\big)$ converges in law to $\Ga^0(x, \int_0^{\cdot} h(s)ds)$ as $\vare \rightarrow 0$.

Then for all $x \in \cE_0$, $I_x$ is a rate function on $\cE$, the family $\{I_x, x \in \cE_0\}$ of rate functions has compact level sets on compacts and $\{X^{\vare, x}\}_{\vare>0}$ satisfies a uniform Laplace principle  with respect to rate function $I_x$ uniformly over $\mathscr{K}$.

\end{thm}

Next we present a sufficient condition for verifying the assumptions in Theorem \ref{BD-criteria}.  It is a modification of Theorem 3.2 in \cite{MSZ}.
\begin{thm}\label{suffthm}
Suppose that there exists a measurable map $\Ga^{0} : \cE_0 \times C([0,T],\mathbb{R}^m) \rightarrow \cE$ such that

{\rm (i)} for every $N < +\infty$, $x_n \rightarrow x$ and any family $\{h_n, n\in\mathbb{N}\} \subset S^N$ converging  weakly to some element $h$ as $n \rightarrow \infty$, $\Ga^0 \big( x_n, \int_0^{\cdot}{h_n}(s)ds\big)$ converges to $\Ga^0 \big(x, \int_0^{\cdot}{h(s)ds}\big)$ in the space $C([0,T],\mathbb{R}^d)$;

{\rm (ii)} for every $N < +\infty$, $\{x^{\vare}, \vare>0\} \subset \{x:  |x| \leq N \}$ and any family $\{h^{\vare}, \vare > 0\} \subset \tilde{S}^N$ and any $\delta > 0$,
\[
\lim_{\vare \rightarrow 0} \mathbb{P}\big(\rho(Y^{\vare,x^\vare},Z^{\vare,x^\vare})> \delta \big) = 0,
\]
where $Y^{\vare,x^\vare}= \Ga^{\vare}\big( x^\vare, B(\cdot) + \frac{1}{\sqrt{\vare}}\int_0^{\cdot}{h^{\vare}(s)ds}\big)$
and $Z^{\vare, x^\vare}=\Ga^0 \big(x^{\vare}, \int_0^{\cdot}{h^{\vare}(s)ds}\big)$.

Then for all $x \in \cE_0$, $I_x$ defined by \eqref{rate} is a rate function on $\cE$, the family $\{I_x, x \in \cE_0\}$ of rate functions has compact level sets on compacts and $\{X^{\vare, x}\}_{\vare>0}$ satisfies a uniform Laplace principle  with the rate function $I_x$ uniformly over $\mathscr{K}$.
\end{thm}

\begin{proof}
We will show that the conditions in Theorem \ref{BD-criteria} are fulfilled. Condition (a) in Theorem \ref{BD-criteria}  follows from condition (i) because $S^N$ and $K$ are both compact sets.
Fix $f \in \cE$ and $x_n \rightarrow x$ in $\cE_0$, we prove that $I_x(f) \leq \varliminf_{x_n \rightarrow x} I_{x_n}(f)$. Without loss of generality, we assume that $\lim_{x_n \rightarrow x}  I_{x_n}(f) =\varliminf_{x_n \rightarrow x} I_{x_n}(f) < M-1 <\infty$ for some $M >0$. By the definition of $I_{x_n}$, there exists
$h_n \in S^{2M}$ satisfies $f= \Ga^0(x_n, \int_0^{\cdot} h_n(s)ds)$  such that
\[
I_{x_n}(f) \geq \frac{1}{2}\int_0^T |h_n(s)|^2ds -\frac{1}{n}.
\]
By the compactness of $S^{2M}$, there exist $h \in S^{2M}$ such that $ h_n$ (take a subsequence if necessary) converging weakly to $h$. Thus, condition (i) implies  that  $f = \Ga^0(x, \int_0^{\cdot} h(s)ds)$, then
\[
\lim_{x_n \rightarrow x} I_{x_n}(f) \geq \frac{1}{2}\int_0^T |h(s)|^2ds  \geq I_x(f).
\]
Theorefore, for all $f \in \cE$, $x \mapsto I_x(f)$ is a lower semi-continuous map
from $\cE_0$ to $[0,\infty]$.

Condition (ii) implies that for any bounded, uniformly continuous function $F(\cdot)$ on $\cE$,
\begin{equation}\label{suff-1}
\lim_{\vare \rightarrow 0} \EE[|F(Y^{\vare,x^\vare})-F(Z^{\vare,x^\vare})|] =0.
\end{equation}
Because the mapping $\Ga^0$ is continuous by condition (i) and $(x^\vare, h^\vare)$ converge in law to $(x, h)$, then $Z^{\vare,x^\vare}$ converges in law to $\Ga^0(x, \int_0^{\cdot} h(s)ds)$ in the space $\cE$. Combined with \eqref{suff-1} we see that   condition (b) holds. The proof is complete.
\end{proof}

\subsection{Proof of Theorem \ref{thmULDP}}
Let $S^N$ and $\tilde{S}^N$ be defined as in Section 3. According to  Proposition \ref{prop2}, there exists  a measurable mapping $\Ga^0$ from $\RR^d \times C([0,T], \mathbb{R}^m)$ to $C([0,T], \mathbb{R}^d)$ such that $X_x^{h} = \Ga^0 \big(x, \int_0^{\cdot}{h(s)ds}\big)$ for  $x \in \RR^d$ and $h\in L^2([0,T],\mathbb R^m)$.

By the Yamada-Watanabe theorem, the existence of a unique strong solution of Eq. \eqref{12} and Assumption \ref{ass1} implies that for
every $\vare>0$, there exists a measurable mapping $\Ga^{\vare}:\RR^d \times C([0,T], \mathbb{R}^m)\rightarrow C([0,T], \mathbb{R}^d)$ such that
\begin{equation*}
X^{\vare,x}=\Ga^{\vare}( x, B(\cdot) ),
\end{equation*}
and applying the Girsanov theorem, for any $N>0$ and $h^{\vare}\in \tilde{S}^N$,
\begin{equation}\label{52}
Y^{\vare,x} := \Ga^{\vare}\big(x,  B(\cdot) + \frac{1}{\sqrt{\vare}}\int_0^{\cdot}{h^{\vare}(s)ds}\big)
\end{equation}
is the solution of the following SDE
\begin{equation}\label{32}
Y^{\vare,x}(t)=x+\int_0^tb(Y^{\vare}(s))ds+\int_0^t\sigma(Y^{\vare}(s)) h^\vare(s)ds+\sqrt{\vare}\int_0^t\sigma(Y^{\vare}(s))dB(s).
\end{equation}

By virtue of the Proposition \ref{eqprop} and Theorem \ref{suffthm}, to prove Theorem \ref{thmULDP}, we need to verify the conditions (i) and (ii) in Theorem \ref{suffthm} for the measurable maps $\Ga^\vare$ and $\Ga^0$.

The verification of Conditions (i) and (ii) is similar to the proof of Proposition 3.1 in \cite{WYZZ}, we here only give a sketch.

{\bf Proof of condition (i):}   Let $x_n \rightarrow x$ and $\{h_n\}_{n \in \NN} \subset S^N$ converges in the weak topology to $h$ as $n \rightarrow \infty$.

Define $W_{x_n}^{h_n}(t)=e^{-\eta\int_0^{t}|h_n(s)|^2ds}V(X_{x_n}^{h_n}(t))$, where $\theta > 0, \eta > 0$ and $V$ are in Assumption \ref{ass2}. Apply the chain rule  and \eqref{22}, \eqref{23}, we get
\begin{eqnarray}\label{4-1}
&&W_{x_n}^{h_n}(t) \notag\\
&\leq& V(x_n) + \int_0^t e^{-\eta\int_0^{s}|h_n(r)|^2dr}\big[ \langle \nabla V(X_{x_n}^{h_n}(s)),b(X_{x_n}^{h_n}(s))\rangle +\frac{|\sigma^{*}(X_{x_n}^{h_n}(s)) \cdot \nabla V (X_{x_n}^{h_n}(s))|^2}{\eta V(X_{x_n}^{h_n}(s))}\big]ds \notag \\
&\leq& V(x_n) + \int_0^t e^{-\eta\int_0^{s}|h_n(r)|^2dr}[C(1+V(X_{x_n}^{h_n}(s))) +\frac{\theta}{2}(M+CV(X_{x_n}^{h_n}(s)))]ds \notag\\
&\leq& V(x_n) + C\int_0^t 1+W_{x_n}^{h_n}(s)\  ds.
\end{eqnarray}
Applying the Gronwall inequality, the above inequality \eqref{4-1} yields   $\sup_{n\in\mathbb{N}}\sup_{t\in[0,T]}V(X_{x_n}^{h_{n}}(t))<\infty$.
Thus,
\begin{equation}
 \sup_{n\in\mathbb{N}}\sup_{t\in[0,T]}|X_{x_n}^{h_n}(t)|\leq L
\end{equation}
for some constant $L>0$.

Next, using the Arzela-Ascoli theorem, we can show that $\{X_{x_n}^{h_n},n\in\mathbb{N}\}$ is pre-compact in the space $C([0,T],\mathbb R^d)$ and
\[
 X_{x_n}^{h_n} \rightarrow \tilde{X} \quad \text{in}~C([0,T],\mathbb{R}^d),
\]
for some $\tilde{X}$.
Then, the uniqueness of skeleton equation \eqref{30} implies $ \tilde{X} = X_x^h$. Therefore, condition (i) holds.

{\bf Proof of condition (ii):} Let $\{x^{\vare}, \vare>0\} \subset \{x:  |x| \leq N \}$ and a family $\{h^{\vare}, \vare > 0\} \subset \tilde{S}^N$, we need to prove that
$\rho(Y^{\vare,x^\vare},Z^{\vare,x^\vare}) \rightarrow 0$ in probability as $\vare \rightarrow 0$, where $Y^{\vare,x^\vare}= \Ga^{\vare}\big( x^\vare, B(\cdot) + \frac{1}{\sqrt{\vare}}\int_0^{\cdot}{h^{\vare}(s)ds}\big)$
and $Z^{\vare, x^\vare}=\Ga^0 \big(x^{\vare}, \int_0^{\cdot}{h^{\vare}(s)ds}\big)$.

Recall that $\vare_0$ is the constant appeared in Assumption \ref{ass1}.
For $R>0$, $0<p\leq \vare_0$, define
\begin{equation*}
\tau^{\vare}_R=\inf\{t\geq0:|Y^{\vare,x^\vare}(t)|\geq R\},\quad \tau^\vare_{p}=\inf\{t\geq0:|Y^{\vare,x^\vare}(t)-Z^{\vare,x^\vare}(t)|^2\geq p\}.
\end{equation*}
From the proof of \eqref{4-1} we also see that  there exists a constant $L>0$ such  that
\[
\sup_{\vare>0}\sup_{t\in[0,T]}|Z^{\vare,x^\vare}(t)|\leq L.
\]

Applying the It\^{o} formula to $e^{-\int_0^t (L_R +|h^\vare(s)|^2)ds}|Y^{\vare, x^\vare}(t)-Z^{\vare, x^\vare}(t)|^2$, replacing $t$ by the stopping time  $\hat{\tau}^\vare:=T\wedge \tau^\vare_R\wedge\tau^\vare_{p}$ and  taking expectation, we obtain
\begin{eqnarray}
\EE e^{-\int_0^{\hat{\tau}^\vare} (L_R +|h^\vare(s)|^2)ds}|Y^{\vare, x^\vare}(\hat{\tau}^\vare)-Z^{\vare,, x^\vare}(\hat{\tau}^\vare)|^2 \leq \vare \EE \int_0^{\hat{\tau}^\vare} \|\si(Y^{\vare, x^\vare}(s))\|^2 ds.
\end{eqnarray}
 Thus,
\begin{equation}\label{4-2}
\lim_{\vare \rightarrow 0} \PP(|Y^{\vare, x^\vare}(\hat{\tau}^\vare)-Z^{\vare, x^\vare}(\hat{\tau}^\vare)|^2  \geq p ) =0.
\end{equation}

Noting that
\[
\{\tau^\vare_{p}\leq T\wedge\tau^{\vare}_R\}\subset \{\sup_{s\leq T\wedge\tau^\vare_R\wedge\tau^\vare_{p}}|Y^{\vare, x^\vare}(s)-Z^{\vare, x^\vare}(s)|^2\geq p\}\subset \{|Y^{\vare, x^\vare}(\hat{\tau}^\vare)-Z^{\vare, x^\vare}(\hat{\tau}^\vare)|^2\geq p\},
\]
it follows from \eqref{4-2} that
\begin{equation}\label{4-3}
  \lim_{\vare\rightarrow 0}\PP\big(\tau^\vare_{p}\leq \tau^{\vare}_R\wedge T\big) \leq \lim_{\vare\rightarrow 0}\PP\big(\sup_{s\leq T\wedge\tau^\vare_R\wedge\tau^\vare_{p}}|Y^{\vare, x^\vare}(s)-Z^{\vare, x^\vare}(s)|^2\geq p\big)=0.
\end{equation}

Next, we prove
\begin{equation}\label{4-4}
\lim_{R\rightarrow\infty}\sup_{\vare\in(0,\frac{\theta}{2})}\PP\big(\tau^\vare_R \leq T\wedge\tau^\vare_{p}\big)=0.
\end{equation}
Applying the It\^{o} formula to $U^{\vare, h^\vare}(t) = e^{-\eta\int_0^{t}|h^\vare(s)|^2ds}V(Y^{\vare, x^\vare}(t))$ gives
\begin{eqnarray}\label{4-6}
&&\EE U^{\vare, h^\vare}(t\wedge \tau^\vare_{R}\wedge\tau^\vare_{p})\notag \\
&\leq& V(x^\vare) + \EE \int_0^{t\wedge \tau^\vare_{R}\wedge\tau^\vare_{p}} e^{-\eta\int_0^s |h^\vare(r)|^2dr} \big[\langle \nabla V(Y^{\vare, x^\vare}(s)),b(Y^{\vare, x^\vare}(s))\rangle \notag \\
&&+\frac{|\sigma^{*}(Y^{\vare, x^\vare}(s)) \cdot \nabla V(Y^{\vare, x^\vare}(s))|^2}{\eta V(Y^{\vare, x^\vare}(s))}+\vare\cdot\text{Trace}\big(\nabla^2 V(Y^{\vare, x^\vare}(s))\sigma(Y^{\vare, x^\vare}(s))\sigma^{*}(Y^{\vare, x^\vare}(s))\big)\big]ds \notag\\
&\leq& V(x^\vare) +\EE \int_0^{t\wedge \tau^\vare_{R}\wedge\tau^\vare_{p}} e^{-\eta\int_0^s |h^\vare(r)|^2dr}
\big[C(1+V(Y^{\vare, x^\vare}(s))) +(\frac{\theta}{2}-\vare)(CV(Y^{\vare, x^\vare}(s))+ M\big] ds \notag \\
&\leq& V(x^\vare) + C \EE \int_0^t  \big(1 +U^{\vare, h^\vare}(s\wedge \tau^\vare_{R}\wedge\tau^\vare_{p})\big) ds.
\end{eqnarray}
Applying the Gronwall inequality again, \eqref{4-6}  yields $\EE U^{\vare, h^\vare}(T\wedge \tau^\vare_{R}\wedge\tau^\vare_{p}) \leq (V(x^\vare)+C)e^{CT}$.
From the above inequality, we deduce that .
\[
\PP\big(\tau^\vare_R \leq T\wedge\tau^\vare_{p}\big) \leq \frac{(C+ \sup_{\vare \in (0, \frac{\theta}{2})}V(x^\vare))e^{CT+\et N}}{\inf_{|x| \geq R}V(x)}
\]
Let $R\rightarrow\infty$, we obtain \eqref{4-4}  since $\sup_{\vare >0} |x^\vare| \leq N$.

For arbitrary $\delta>0$, we have
\begin{eqnarray*}
&&\PP\big(\sup_{0\leq s\leq T}|Y^{\vare, x^\vare}(s)-Z^{\vare, x^\vare}(s)|\geq\de\big)\\
&=&\PP\big(\sup_{0\leq s\leq T}|Y^{\vare, x^\vare}(s)-Z^{\vare, x^\vare}(s)|\geq\de,\tau^\vare_R\wedge\tau^\vare_{\de^2}>T\big) \\
&&+\PP\big(\sup_{0\leq s\leq T}|Y^{\vare, x^\vare}(s)-Z^{\vare, x^\vare}(s)|\geq\de,\tau^\vare_R\leq T\wedge\tau^\vare_{\de^2}\big)\\
&&+\PP\big(\sup_{0\leq s\leq T}|Y^{\vare, x^\vare}(s)-Z^{\vare, x^\vare}(s)|\geq\de,\tau^\vare_{\de^2}\leq \tau^\vare_R\wedge T\big)\\
&\leq&\PP\big(\sup_{0\leq s\leq T\wedge\tau^\vare_R\wedge\tau^\vare_{\de^2}}|Y^{\vare, x^\vare}(s)-Z^{\vare, x^\vare}(s)|^2\geq \de^2\big)+\PP\big(\tau^\vare_R\leq T\wedge\tau^\vare_{\de^2}\big)+\PP\big(\tau^\vare_{\de^2}\leq \tau^\vare_R\wedge T\big).
\end{eqnarray*}
\eqref{4-3} (with $p=\de^2$) implies that
\begin{eqnarray*}
\lim_{\vare\rightarrow 0}\PP\big(\sup_{0\leq s\leq T}|Y^{\vare, x^\vare}(s)-Z^{\vare, x^\vare}(s)|\geq\de\big)
\leq
\sup_{\vare\in(0,\frac{\theta}{2})}\PP\big(\tau^\vare_R\leq T\wedge\tau^\vare_{\de^2}\big).
\end{eqnarray*}
Let $R\rightarrow \infty$ and \eqref{4-4} to get
\begin{equation*}
\lim_{\vare\rightarrow0}\PP\big(\sup_{0\leq s\leq T}|Y^{\vare, x^\vare}(s)-Z^{\vare, x^\vare}(s)|\geq\de\big)=0.
\end{equation*}
Therefore, condition (ii) holds.
\begin{rmk}\label{rmk2.2}
If the diffusion processes $X^{\vare}$ lives in a domain $D$ with boundary, to apply Theorem \ref{thmULDP} one only needs to slightly adjust the conditions on $V$ in Assumption \ref{ass2} to the condition that $V \in C^2(D; \RR_+)$ and $\lim_{\underset{|x| \rightarrow \infty}{ x \in D}} V(x) = +\infty$.
\end{rmk}

\section{The concentration of limiting measures and metastability}

In this section, we will apply the ULDP obtained  in Section 2 to describe the support of the limiting measures of the invariant measures of the system  (\ref{11}).

From now on we denote $C([0,T],\mathbb{R}^d)$ by ${\bf C}_{T}$. ${\bf C}_T^x$ (${\bf AC}_T^x$) will denote the space of (absolutely)  continuous functions on $[0,T]$ started from $x$ with values in $\mathbb{R}^d$.

\emph{Quasipotential}, introduced by Freidlin and Wentzell (see, e.g., \cite [p.90 or p.142]{FW}),
is a very useful notion and is defined by
\[  \mathbb{V}(x,y):=\inf\big\{I_{x}^T(\varphi):\varphi(0)=x,\ \varphi(T)=y,\ T> 0\big\},\: x,\ y\in\mathbb{R}^d.  \]
where $I_x^T(\cdot)$ is the rate function defined in (\ref{rate-0}).
We can also define $\mathbb{V}$ on pairs of subsets in $\mathbb{R}^d$ as follows:
\[
 \mathbb{V}(D_0,D):=\inf\big\{I_x^{T}(\varphi):\varphi(0)\in D_0,\varphi(T)\in D, T> 0\big\}, \: D_0,D\subset\mathbb{R}^d.
\]
We first assume that the noise is non-degenerate:
\begin{equation}\label{E}
\beta^{*}\sigma(x)\sigma^{*}(x)\beta >0\ {\rm for\ all}\ x\in \mathbb{R}^d\ {\rm and}\ \beta \in \mathbb{R}^d\backslash \{0\}.
\end{equation}
Under this assumption, the rate function has the following expression
\begin{equation}\label{actfalini}
  I_{x}^{T}(\varphi)
=\left\{
\begin{array}{ll}
\int\limits_{0}^{T}L(\varphi(t),\dot{\varphi}(t)) {\rm d}t, & \hbox{if $\varphi\in{\bf AC}_T^x$,} \\
+\infty, & \hbox{otherwise,}
\end{array}
\right.
\end{equation}
where $L(u,\beta)=\frac{1}{2}\big(\beta-b(u)\big)^{*} \left(\sigma(u)\sigma^{*}(u)\right)^{-1}\big(\beta-b(u)\big), u,\beta\in\mathbb{R}^d$. It follows from \cite [Lemma 2.3, p.94]{FW} or \cite [Lemma 2.5]{XCJ} that

\begin{lem}\label{Vcon}
Suppose that {\rm(\ref{E})} holds. Then for any $R>0$ there exists $L=L(R)>0$ such that for any $x,y\in \bar{B}_R(O)$, the smooth function $\varphi(t)=x+\frac{t}{|y-x|}(y-x),\ t\in [0,|y-x| ]$ satisfies
\begin{equation*}\label{VconE}
I_{x}^{|y-x|}(\varphi)\le L|y-x|
\end{equation*}
which implies that $\mathbb{V}(x,y)$ is continuous.
\end{lem}

Consider now the small noise perturbation
\begin{equation}\label{HEQ}
\begin{cases}
    X_{\varepsilon}^{(2)}(t) = b\big(X_{\varepsilon}(t), X_{\varepsilon}^{(1)}(t)\big) + \sqrt\eps \si\big(X_{\varepsilon}(t), X_{\varepsilon}^{(1)}(t)\big) \dot{B}(t), \\
    (X_{\varepsilon},X_{\varepsilon}^{(1)})(0) = (x_1, x_2)=:x \in \RR^2
\end{cases}
\end{equation}
of single degree of freedom second-order differential equation:
\begin{equation}\label{HEQU}
\begin{cases}
    X_{\varepsilon}^{(2)}(t) = b\big(X_{\varepsilon}(t), X_{\varepsilon}^{(1)}(t)\big), \\
    (X_{\varepsilon},X_{\varepsilon}^{(1)})(0) = (x_1, x_2)=:x \in \RR^2
\end{cases}
\end{equation}
where $X_{\varepsilon}^{(i)}$ stands for the derivative of order $i$,  the coefficients $b$ and $\si$ are real-valued, locally Lipschitz continuous functions on $\RR^2$ and $\si(x) \ne 0$ for all $x \in \RR^2$.

Let $Y_{\varepsilon}:=(Y_{\varepsilon}^1,Y_{\varepsilon}^2) := (X_{\varepsilon},X_{\varepsilon}^{(1)})$. Then \eqref{HEQ} is equivalent to the following system of SDEs:
\begin{equation}\label{HEQU}
\begin{cases}
    dY_{\varepsilon}^{1}(t) = Y_{\varepsilon}^{2}(t)dt, \\
    dY_{\varepsilon}^{2}(t) = b\big(Y_{\varepsilon}^{1}(t),  Y_{\varepsilon}^{2}(t)\big)dt + \sqrt\eps \si\big(Y_{\varepsilon}^{1}(t),  Y_{\varepsilon}^{2}(t)\big) dB(t),\\
    (Y_{\varepsilon}^{1},Y_{\varepsilon}^{2})(0) = (x_1, x_2).\\

\end{cases}
\end{equation}
In this case, the rate function for the large deviation  of $Y_{\varepsilon}$ can be written as
\begin{equation}\label{high-rate-0808}
  I_{x}^{T}(\psi)
=\left\{
\begin{array}{ll}
\frac{1}{2}\int_{0}^{T}\Big|\frac{\varphi^{(2)}(s) - b(\varphi(s),  \varphi^{(1)}(s))}{\si(\varphi(s), \varphi^{(1)}(s))}\Big|^2ds & \hbox{if $\psi=(\varphi, \varphi^{(1)})\in{\bf AC}_T^x$,} \\
+\infty, & \hbox{otherwise,}
\end{array}
\right.
\end{equation}
and the corresponding quasipotential can be expressed as
\[
 \mathbb{V}(x,y) =\inf\big\{I_{x}^T(\psi): \psi(0)=(\varphi,  \varphi^{(1)})(0)=x,\ \psi(T)=(\varphi, \varphi^{(1)})(T) =y,\ T> 0\big\},\: x,\ y\in\mathbb{R}^2.
\]

Lemma \ref{Vcon} implies that for any $x,y\in \bar{B}_R(O)$, there exists a sampling orbit $\varphi(t)$ connecting $x$ and $y$ and  spending the time $T(x,y)=|y-x|$ such that $I_{x}^{|y-x|}(\varphi)\le L|y-x|$. In particular, $T(x,y)\le 2R$. The continuity of $\mathbb{V}$ and the boundedness of spent time of the above constructed  sampling orbit play an important role in estimating rare probability of certain subsets of (\ref{11}) with non-degenerate noise. However,  the driving noise of (\ref{HEQU}) is degenerate. We cannot prove the version of Lemma \ref{Vcon} for system (\ref{HEQU}) because $\varphi(t)$ is not a sampling function as defined in (\ref{high-rate-0808}), instead prove the following weaker version, which is sufficient to study the limiting measures of stochastic van der Pol equation and stochastic van der Pol Duffing equation in Section 4.
\begin{lem}\label {PropQ 0808}
Suppose that the rate function $I_x^T$ is of the form {\rm(\ref{high-rate-0808})}.
 Then \\
{\rm (i)} for any $(x,y)\in \RR^2\times \RR^2$ and $T>0$, there exists  $\psi^x\in C_{T}$ with $\psi^x(0)=x$ and $\psi^x(T)=y$ such that $I^{T}_x(\psi^x)<\infty$.
Consequently, for any $(x,y)\in \RR^2\times \RR^2$, $\mathbb{V}(x,y)<\infty;$\\
 {\rm (ii)} the quasipotential $\mathbb{V}$ is upper semi-continuous on $\RR^2\times \RR^2$, that is, for any $(x,y)\in \RR^2\times \RR^2$,
\begin{eqnarray}\label{eq Part 1}
    \lim_{\delta\searrow0}\sup_{(\tilde{x},\tilde{y})\in B_\delta((x,y))\setminus (x,y)}\mathbb{V}(\tilde{x},\tilde{y})\leq \mathbb{V}(x,y),
\end{eqnarray}
in particular, $\mathbb{V}$ is continuous at any equilibrium of {\rm(\ref{HEQU})};\\
{\rm (iii)} for given $(x_0,y_0)\in \RR^2\times \RR^2$ and $\delta, \eta>0$, suppose that $\mathbb{V}(x,y_0)<\eta$ for all $x\in \overline{B}_{\delta}(x_0):=\{x\in\mathbb{R}^2:|x-x_0|\leq \delta\}$. Then there exists a constant $T_0\in(0,\infty)$ such that, for any $x\in \overline{B}_{\delta}(x_0)$, there exist $T^x\in(0,T_0]$ and $\psi^x\in C_{T^x}$ with $\psi^x(0)=x$ and $\psi^x(T^x)=y_0$ such that $I^{T^x}_x(\psi^x)<\eta$.
\end{lem}

\begin{proof}
Fix $x=(x_1,x_2),y=(y_1,y_2)\in\RR^2$ and $T>0$.
Let $f_1(t) = x_1 +x_2t$ and $f_2(t) = y_1 +y_2(t-T)$. We choose a cut-off function $\al \in C^{\infty}(\RR)$:
 \begin{eqnarray}
    \al(t)
       =
         \left\{
  \begin{array}{ll}
  1, & t \in [0,\frac{T}{3}] \\
   >0, & t\in (\frac{T}{3},\frac{2T}{3}) \\
    0, & t \in [\frac{2T}{3},T].
  \end{array}
\right.
\end{eqnarray}
Then $\varphi = \al f_1+ (1-\al)f_2$ satisfies $(\varphi, \varphi^{(1)})(0)= x, (\varphi, \varphi^{(1)})(T)= y$. By the assumptions that the coefficients $b$ and $\si$ are real-valued, locally Lipschitz continuous functions on $\RR^2$ and $\si(x) \ne 0$ for all $x \in \RR^2$, it is easy to see that $I_x^T((\varphi, \varphi^{(1)})) < \infty$.
The proof of (i) is complete.

We are now in the position to prove (ii). Notice that, for any $\tilde{x},\tilde{y}\in\mathbb{R}^2$, $\mathbb{V}(\tilde{x},\tilde{y})\geq0$, and that if $U$ is a equilibrium of \rm(\ref{HEQU}) and $x,y\in U$, then $\mathbb{V}(x,y)=0$. Once we prove  $\eqref{eq Part 1}$, the claim that $\mathbb{V}$ is continuous at any equilibrium of {\rm(\ref{HEQU})} holds. In the following, we prove \eqref{eq Part 1}.

To this end, we first introduce the following two functions, which will be used later.

For $j>0$, let
\begin{eqnarray}
    \theta^-_j(t)
       =
         \left\{
  \begin{array}{ll}
  0, & \hbox{$t<0$;} \\
    -t, & \hbox{$t\in[0,j]$;} \\
    t-2j, & \hbox{$t\in(j,2j+\sqrt{2}j]$;} \\
    0, & \hbox{$t>2j+\sqrt{2}j$,}
  \end{array}
\right.
\end{eqnarray}
and
\begin{eqnarray}
    \theta^+_j(t)
       =
         \left\{
  \begin{array}{ll}
  0, & \hbox{$t<0$;} \\
    t, & \hbox{$t\in[0,j]$;} \\
    2j-t, & \hbox{$t\in(j,2j+\sqrt{2}j]$;} \\
    0, & \hbox{$t>2j+\sqrt{2}j$.}
  \end{array}
\right.
\end{eqnarray}

We remark that
\begin{eqnarray}\label{eq theta jian j prop}
    \int_0^{2j+\sqrt{2}j}\theta^-_j(t)dt=0\text{ and }\int_0^{2j+\sqrt{2}j}\int_0^s\theta^-_j(t)dtds=-\frac{1}{3}(3+2\sqrt{2})j^3,
\end{eqnarray}
\begin{eqnarray}\label{eq theta jia j prop}
    \int_0^{2j+\sqrt{2}j}\theta^+_j(t)dt=0\text{ and }\int_0^{2j+\sqrt{2}j}\int_0^s\theta^+_j(t)dtds=\frac{1}{3}(3+2\sqrt{2})j^3,
\end{eqnarray}
and
\begin{eqnarray}\label{eq theta j sup}
    \sup_{s\in[0,2j+\sqrt{2}j]}\Big|\int_0^{s}\theta^-_j(t)dt\Big|
    =
    \sup_{s\in[0,2j+\sqrt{2}j]}\Big|\int_0^{s}\theta^+_j(t)dt\Big|
    =j^2.
\end{eqnarray}

Now we prove \eqref{eq Part 1}.

Since $\mathbb{V}(x,y)<\infty$, for any $\epsilon>0$, there exist $T_0>0$ and $\psi=(\varphi,\varphi^{(1)})=\{(\varphi(t),\varphi^{(1)}(t)),t\in[0,T_0]\}\in {\bf AC}_{T_0}^x$ satisfying
\begin{enumerate}
    \item $\psi(0)=x$, $\psi(T_0)=y$;
    \item letting $h^\sigma(s)=\frac{\varphi^{(2)}(s) - b(\varphi(s),  \varphi^{(1)}(s))}{\si(\varphi(s), \varphi^{(1)}(s))}, s\in[0,T_0]$, we have
    \begin{eqnarray}\label{eq sigma}
    \frac{1}{2}\int_{0}^{T_0}\Big|h^\sigma(s)\Big|^2ds-\epsilon\leq \mathbb{V}(x,y).
    \end{eqnarray}
\end{enumerate}
Let us remark that
\begin{eqnarray}
&&\varphi^{(1)}(t)=x_2+\int_0^t\varphi^{(2)}(s)ds,\ t\in[0,T_0];\label{eq Lemma 3.1 00}\\
&&\varphi(t)=x_1+\int_0^t\varphi^{(1)}(s)ds,\ t\in[0,T_0];\label{eq Lemma 3.1 01}\\
&&\varphi(0)=x_1;\ \varphi(T_0)=y_1;\ \varphi^{(1)}(0)=x_2;\ \varphi^{(1)}(T_0)=y_2.\label{eq Lemma 3.1 02}
\end{eqnarray}

For any $1\geq \delta>0$ and any $\tilde{x}=(\tilde{x}_1,\tilde{x}_2)\in B_\delta(x),\tilde{y}=(\tilde{y}_1,\tilde{y}_2)\in B_\delta(y)$. We define
\begin{eqnarray}\label{eq zeta 00 2}
   \zeta_2(t)
   =
   \left\{
  \begin{array}{ll}
    {\rm sgn}(x_2-\tilde{x}_2), & \hbox{$t\in[0,|x_2-\tilde{x}_2|)$;} \\
    \varphi^{(2)}(t-|x_2-\tilde{x}_2|), & \hbox{$t\in[|x_2-\tilde{x}_2|,T_0+|x_2-\tilde{x}_2|)$;}\\
    {\rm sgn}(\tilde{y}_2-y_2), & \hbox{$t\in[T_0+|x_2-\tilde{x}_2|,T_0+|x_2-\tilde{x}_2|+|\tilde{y}_2-y_2|]$.}
  \end{array}
\right.
\end{eqnarray}

The proof of \eqref{eq Part 1} is divided into three cases: Case 1: $y_2>0$; Case 2: $y_2<0$; Case 3: $y_2=0$.

Case 1: $y_2>0$.

Without loss of generality, we choose $0<\delta\leq 1$ small enough such that $\tilde{y}_2\geq \frac{y_2}{2}$. Set
$$y_3:=\tilde{x}_1+\frac{1}{2}|x_2-\tilde{x}_2|(x_2+\tilde{x}_2)+y_1-x_1+\frac{1}{2}|y_2-\tilde{y}_2|(y_2+\tilde{y}_2).$$

In the following, we divide the case into two subcases.

Subcase 1: $\tilde{y}_1\geq y_3$.

For any $t\in(T_0+|x_2-\tilde{x}_2|+|\tilde{y}_2-y_2|,T_0+|x_2-\tilde{x}_2|+|\tilde{y}_2-y_2|+\frac{\tilde{y}_1- y_3}{\tilde{y}_2}]$, set
\begin{eqnarray*}
    \zeta_2(t)=0.
\end{eqnarray*}
Let
\begin{eqnarray*}
    \zeta_1(t)=\tilde{x}_2+\int_0^t\zeta_2(s)ds;\
    \zeta(t)=\tilde{x}_1+\int_0^t\zeta_1(s)ds.
\end{eqnarray*}
We have
\begin{eqnarray}
    &&\zeta_1(0)=\tilde{x}_2;\ \zeta(0)=\tilde{x}_1;\label{eq lemma 3.1 1001}\\
    &&\zeta_1(|x_2-\tilde{x}_2|)=x_2;\ \ \ \ \
    \zeta(|x_2-\tilde{x}_2|)=\tilde{x}_1+\frac{1}{2}|x_2-\tilde{x}_2|(x_2+\tilde{x}_2);\nonumber\\
    &&\zeta_1(t)=\varphi^{(1)}(t-|x_2-\tilde{x}_2|),\ t\in [|x_2-\tilde{x}_2|,|x_2-\tilde{x}_2|+T_0];\nonumber\\
    &&\zeta(t)=\zeta(|x_2-\tilde{x}_2|)+\varphi(t-|x_2-\tilde{x}_2|)-x_1,\ t\in [|x_2-\tilde{x}_2|,|x_2-\tilde{x}_2|+T_0];\nonumber\\
    &&\zeta_1(T_0+|x_2-\tilde{x}_2|)=\varphi^{(1)}(T_0)=y_2;\nonumber\\
    &&\zeta(T_0+|x_2-\tilde{x}_2|)=\zeta(|x_2-\tilde{x}_2|)+\varphi(T_0)-x_1=\zeta(|x_2-\tilde{x}_2|)+y_1-x_1;\nonumber\\
    &&\zeta_1(T_0+|x_2-\tilde{x}_2|+|\tilde{y}_2-y_2|)=\tilde{y}_2;\nonumber\\
    &&\zeta(T_0+|x_2-\tilde{x}_2|+|\tilde{y}_2-y_2|)=y_3;\nonumber\\
    &&\zeta_1(T_0+|x_2-\tilde{x}_2|+|\tilde{y}_2-y_2|+\frac{\tilde{y}_1- y_3}{\tilde{y}_2})=\tilde{y}_2;\label{eq lemma 3.1 1002}\\
    &&\zeta(T_0+|x_2-\tilde{x}_2|+|\tilde{y}_2-y_2|+\frac{\tilde{y}_1- y_3}{\tilde{y}_2})=\tilde{y}_1.\label{eq lemma 3.1 1003}
\end{eqnarray}
Let
\begin{eqnarray}\label{eq lemma 3.1 1004}
    h(s)
    =
    \frac{\zeta_2(s) - b(\zeta(s),  \zeta_1(s))}{\si(\zeta(s), \zeta_1(s))}.
\end{eqnarray}
Then
\begin{eqnarray*}
    &&\left|\int_0^{T_0+|x_2-\tilde{x}_2|+|\tilde{y}_2-y_2|+\frac{\tilde{y}_1- y_3}{\tilde{y}_2}}|h(s)|^2ds-\int_0^{T_0}|h^\sigma(s)|^2ds\right|\\
    &\leq&
    \int_0^{|x_2-\tilde{x}_2|}|h(s)|^2ds
    +
    \Big|\int_{|x_2-\tilde{x}_2|}^{T_0+|x_2-\tilde{x}_2|}|h(s)|^2ds-\int_0^{T_0}|h^\sigma(s)|^2ds\Big|\\
    &&+
    \int_{T_0+|x_2-\tilde{x}_2|}^{T_0+|x_2-\tilde{x}_2|+|y_2-\tilde{y}_2|}|h(s)|^2ds
    +
    \int_{T_0+|x_2-\tilde{x}_2|+|y_2-\tilde{y}_2|}^{{T_0+|x_2-\tilde{x}_2|+|y_2-\tilde{y}_2|}+\frac{\tilde{y}_1- y_3}{\tilde{y}_2}}|h(s)|^2ds\\
    &\leq&
      \delta\Big(\frac{1+\max_{\{|\xi_1|+|\xi_2|\leq |x_1|+3|x_2|+3\}}|b(\xi_1,\xi_2)|}{\min_{\{|\xi_1|+|\xi_2|\leq |x_1|+3|x_2|+3\}}|\sigma(\xi_1,\xi_2)|}\Big)^2\\
      &&+
      \Big|\int_{|x_2-\tilde{x}_2|}^{T_0+|x_2-\tilde{x}_2|}|h(s)|^2ds-\int_0^{T_0}|h^\sigma(s)|^2ds\Big|\\
      &&+
      \delta\Big(\frac{1+\max_{\{|\xi_1|+|\xi_2|\leq 2|x_1|+|x_2|+|y_1|+2|y_2|+4\}}|b(\xi_1,\xi_2)|}{\min_{\{|\xi_1|+|\xi_2|\leq 2|x_1|+|x_2|+|y_1|+2|y_2|+4\}}|\sigma(\xi_1,\xi_2)|}\Big)^2\\
      &&+
      2\delta\frac{3+|x_2|+|y_2|}{|y_2|}\Big(\frac{\max_{\{|\xi_1|+|\xi_2|\leq 2|x_1|+|x_2|+|y_1|+|y_2|+3\}}|b(\xi_1,\xi_2)|}{\min_{\{|\xi_1|+|\xi_2|\leq 2|x_1|+|x_2|+|y_1|+|y_2|+3\}}|\sigma(\xi_1,\xi_2)|}\Big)^2.
\end{eqnarray*}
Notice that
\begin{eqnarray*}
    &&\int_{|x_2-\tilde{x}_2|}^{T_0+|x_2-\tilde{x}_2|}|h(s)|^2ds\\
    &=&
    \int_{0}^{T_0}\Big|\frac{\varphi^{(2)}(s) - b(\zeta(|x_2-\tilde{x}_2|)+\varphi(s)-x_1,  \varphi^{(1)}(s))}{\si(\zeta(|x_2-\tilde{x}_2|)+\varphi(s)-x_1,  \varphi^{(1)}(s))}\Big|^2ds.
\end{eqnarray*}
Keeping in mind the assumptions that
the coefficients $b$ and $\si$ are real-valued, locally Lipschitz continuous functions on $\RR^2$ and $\si(x) \ne 0$ for all $x \in \RR^2$, it is not difficult to prove that
there exists a positive constant $C_{\delta,T_0,\varphi,x,y}$, independent of
$\Tilde{x}$ and $\Tilde{y}$, such that
\begin{eqnarray*}
    \Big|\int_{|x_2-\tilde{x}_2|}^{T_0+|x_2-\tilde{x}_2|}|h(s)|^2ds-\int_0^{T_0}|h^\sigma(s)|^2ds\Big|
    \leq
    C_{\delta,T_0,\varphi,x,y}.
\end{eqnarray*}
Here $C_{\delta,T_0,\varphi,x,y}$ satisfies
\begin{eqnarray*}
    \lim_{\delta\rightarrow0}C_{\delta,T_0,\varphi,x,y}=0.
\end{eqnarray*}

We can conclude that there exists a positive constant $D_{\delta,T_0,\varphi,x,y}$, independent of
$\Tilde{x}$ and $\Tilde{y}$, such that
\begin{eqnarray}\label{eq h 1.1}
    \Big|\int_0^{T_0+|x_2-\tilde{x}_2|+|\tilde{y}_2-y_2|+\frac{\tilde{y}_1- y_3}{\tilde{y}_2}}|h(s)|^2ds-\int_0^{T_0}|h^\sigma(s)|^2ds\Big|
    \leq D_{\delta,T_0,\varphi,x,y},
    \end{eqnarray}
    here
\begin{eqnarray}\label{eq h 1.2}
    \lim_{\delta\rightarrow0}D_{\delta,T_0,\varphi,x,y}=0.
\end{eqnarray}

By \eqref{eq sigma}, (\ref{eq lemma 3.1 1001})--(\ref{eq h 1.1}), the definition of $\mathbb{V}$ implies that
\begin{eqnarray}\label{eq subcase 1}
\mathbb{V}(\tilde{x},\tilde{y})
&\leq&
I_{\tilde{x}}^{T_0+|x_2-\tilde{x}_2|+|\tilde{y}_2-y_2|+\frac{\tilde{y}_1- y_3}{\tilde{y}_2}}((\zeta,\zeta_1))\nonumber\\
&=&
\frac{1}{2}\int_0^{T_0+|x_2-\tilde{x}_2|+|\tilde{y}_2-y_2|+\frac{\tilde{y}_1- y_3}{\tilde{y}_2}}|h(s)|^2ds\nonumber\\
&\leq&
\frac{1}{2}\int_0^{T_0}|h^\sigma(s)|^2ds+D_{\delta,T_0,\varphi,x,y}\nonumber\\
&\leq&\mathbb{V}(x,y)+\epsilon+D_{\delta,T_0,\varphi,x,y}.
\end{eqnarray}

Subcase 2: $\tilde{y}_1< y_3$.
For $t\in[0,T_0+|x_2-\tilde{x}_2|+|\tilde{y}_2-y_2|]$, define
\begin{eqnarray}
    \widetilde{\zeta}_2(t)=\zeta_2(t)+\theta^-_{j_1}(t-(T_0+|x_2-\tilde{x}_2|)+2j_1+\sqrt{2}j_1).
\end{eqnarray}
Here $j_1$ satisfies
$$
\frac{1}{3}(3+2\sqrt{2})j_1^3
  =
y_3-\tilde{y}_1.
$$
Set
\begin{eqnarray*} \widetilde{\zeta}_1(t)=\tilde{x}_2+\int_0^t\widetilde{\zeta}_2(s)ds;\
    \widetilde{\zeta}(t)=\tilde{x}_1+\int_0^t\widetilde{\zeta}_1(s)ds.
\end{eqnarray*}
We have
\begin{eqnarray*}
    &&\widetilde{\zeta}_1(0)=\tilde{x}_2;\ \widetilde{\zeta}(0)=\tilde{x}_1;\\
    &&\widetilde{\zeta}_1(|x_2-\tilde{x}_2|)=x_2;\ \ \ \ \
    \widetilde{\zeta}(|x_2-\tilde{x}_2|)=\tilde{x}_1+\frac{1}{2}|x_2-\tilde{x}_2|(x_2+\tilde{x}_2);\\
    &&\widetilde{\zeta}_1(t)=\varphi^{(1)}(t-|x_2-\tilde{x}_2|),\ t\in [|x_2-\tilde{x}_2|,T_0+|x_2-\tilde{x}_2|-2j_1-\sqrt{2}j_1];\\
    &&\widetilde{\zeta}(t)=\widetilde{\zeta}(|x_2-\tilde{x}_2|)+\varphi(t-|x_2-\tilde{x}_2|)-x_1,\ t\in [|x_2-\tilde{x}_2|,T_0+|x_2-\tilde{x}_2|-2j_1-\sqrt{2}j_1].
    \end{eqnarray*}

    For $t\in[T_0+|x_2-\tilde{x}_2|-2j_1-\sqrt{2}j_1,T_0+|x_2-\tilde{x}_2|]$,
    \begin{eqnarray*}
    &&\widetilde{\zeta}_1(t)=\varphi^{(1)}(t-|x_2-\tilde{x}_2|)+\int_0^t\theta^-_{j_1}(s-(T_0+|x_2-\tilde{x}_2|)+2j_1+\sqrt{2}j_1)ds;\\
    &&\widetilde{\zeta}(t)=\widetilde{\zeta}(|x_2-\tilde{x}_2|)+\varphi(t-|x_2-\tilde{x}_2|)-x_1
    +\int_0^t\int_0^s\theta^-_{j_1}(l-(T_0+|x_2-\tilde{x}_2|)+2j_1+\sqrt{2}j_1)dlds.
    \end{eqnarray*}
By (\ref{eq theta jian j prop}), we have
    \begin{eqnarray*}
    &&\widetilde{\zeta}_1(T_0+|x_2-\tilde{x}_2|)=y_2;\\
    &&\widetilde{\zeta}(T_0+|x_2-\tilde{x}_2|)=\widetilde{\zeta}(|x_2-\tilde{x}_2|)+y_1-x_1
    +\tilde{y}_1-y_3.
    \end{eqnarray*}
    Hence
    \begin{eqnarray*}
    &&\widetilde{\zeta}_1(T_0+|x_2-\tilde{x}_2|+|\tilde{y}_2-y_2|)=\tilde{y}_2;\\
    &&\widetilde{\zeta}(T_0+|x_2-\tilde{x}_2|+|\tilde{y}_2-y_2|)=\tilde{y}_1.
\end{eqnarray*}

Let
\begin{eqnarray}\label{eq control 000}
    \widetilde{h}(s)
    =
    \frac{\widetilde{\zeta}_2(s) - b(\widetilde{\zeta}(s),  \widetilde{\zeta}_1(s))}{\si(\widetilde{\zeta}(s), \widetilde{\zeta}_1(s))}.
\end{eqnarray}
Using similar arguments as proving \eqref{eq h 1.1}, we can get
 that there exists a positive constant $\widetilde{D}_{\delta,T_0,\varphi,x,y}$, independent of
$\Tilde{x}$ and $\Tilde{y}$, such that
\begin{eqnarray}\label{eq h 2.1}
    \Big|\int_0^{T_0+|x_2-\tilde{x}_2|+|\tilde{y}_2-y_2|}|\widetilde{h}(s)|^2 ds-\int_0^{T_0}|h^\sigma(s)|^2 ds\Big|
    \leq \widetilde{D}_{\delta,T_0,\varphi,x,y},
    \end{eqnarray}
    here
\begin{eqnarray}\label{eq h 2.2} \lim_{\delta\rightarrow0}\widetilde{D}_{\delta,T_0,\varphi,x,y}=0.
\end{eqnarray}
Similar to the proof of  \eqref{eq subcase 1}, we have
\begin{eqnarray}\label{eq subcase 2}
\mathbb{V}(\tilde{x},\tilde{y})
&\leq&
I_{\tilde{x}}^{T_0+|x_2-\tilde{x}_2|+|\tilde{y}_2-y_2|}((\widetilde{\zeta},\widetilde{\zeta}_1))\nonumber\\
&=&
\frac{1}{2}\int_0^{T_0+|x_2-\tilde{x}_2|+|\tilde{y}_2-y_2|}|\widetilde{h}(s)|^2 ds\nonumber\\
&\leq&
\frac{1}{2}\int_0^{T_0}|h^\sigma(s)|^2ds+\widetilde{D}_{\delta,T_0,\varphi,x,y}\nonumber\\
&\leq&\mathbb{V}(x,y)+\epsilon+\widetilde{D}_{\delta,T_0,\varphi,x,y}.
\end{eqnarray}

The arbitrary of $\epsilon$, \eqref{eq h 1.2}, \eqref{eq subcase 1}, \eqref{eq h 2.2}, and \eqref{eq subcase 2} imply that if $y_2>0$ then \eqref{eq Part 1} holds.

Case 2: Using similar arguments as proving Case 1,  if $y_2<0$ then \eqref{eq Part 1} holds.

Now we consider Case 3: $y_2=0$.

 Set
\begin{eqnarray*}
    y_4:=
    \tilde{x}_1
    +
    \frac{1}{2}|x_2-\tilde{x}_2|(x_2+\tilde{x}_2)
    +
    y_1-x_1
    +
    \frac{1}{2}|\tilde{y}_2|^2{\rm sgn}(\tilde{y}_2),
\end{eqnarray*}
and ${\rm sgn}_1={\rm sgn}(\tilde{y}_1-y_4)$.

Define
{\small
\begin{eqnarray}
   \bar{\zeta}_2(t)
   =
   \left\{
  \begin{array}{ll}
    {\rm sgn}(x_2-\tilde{x}_2), & \hbox{$t\in[0,|x_2-\tilde{x}_2|)$;} \\
    \varphi^{(2)}(t-|x_2-\tilde{x}_2|), & \hbox{$t\in[|x_2-\tilde{x}_2|,T_0+|x_2-\tilde{x}_2|)$;}\\
    \theta^{{\rm sgn}_1}_{j_2}(t-|x_2-\tilde{x}_2|), & \hbox{$t\in[T_0+|x_2-\tilde{x}_2|,T_0+|x_2-\tilde{x}_2|+2j_2+\sqrt{2}j_2)$;}\\
    {\rm sgn}(\tilde{y}_2-y_2), & \hbox{$t\in[T_0+|x_2-\tilde{x}_2|+2j_2+\sqrt{2}j_2,T_0+|x_2-\tilde{x}_2|+2j_2+\sqrt{2}j_2+|\tilde{y}_2-y_2|]$.}
  \end{array}
\right.
\end{eqnarray}
}
Here $j_2$ satisfies
\begin{eqnarray}
   \frac{1}{3}(3+2\sqrt{2})j_2^3
=
|\tilde{y}_1-y_4|.
\end{eqnarray}

Set
\begin{eqnarray*}
    \bar{\zeta}_1(t)=\tilde{x}_2+\int_0^t\bar{\zeta}_2(s)ds;\
    \bar{\zeta}(t)=\tilde{x}_1+\int_0^t\bar{\zeta}_1(s)ds.
\end{eqnarray*}
Using similar arguments as proving Case 1, we have
\begin{eqnarray*}
    &&\bar{\zeta}_1(0)=\tilde{x}_2;\ \bar{\zeta}(0)=\tilde{x}_1;\\
    &&\bar{\zeta}_1(T_0+|x_2-\tilde{x}_2|+2j_2+\sqrt{2}j_2+|\tilde{y}_2-y_2|)=\tilde{y}_2;\\
    &&\bar{\zeta}(T_0+|x_2-\tilde{x}_2|+2j_2+\sqrt{2}j_2+|\tilde{y}_2-y_2|)=\tilde{y}_1.
\end{eqnarray*}
Let
\begin{eqnarray}\label{eq control 001}
    \bar{h}(s)
    =
    \frac{\bar{\zeta}_2(s) - b(\bar{\zeta}(s),  \bar{\zeta}_1(s))}{\si(\bar{\zeta}(s),\bar{\zeta}_1(s))}, \ t\in[0,T_0+|x_2-\tilde{x}_2|+2j_2+\sqrt{2}j_2+|\tilde{y}_2-y_2|].
\end{eqnarray}
Again using similar arguments as in the proof of Case 1, for Case 3,  we can prove \eqref{eq Part 1}, completing the proof of \eqref{eq Part 1}.

The proof of (ii) is complete.

Now we prove (iii). Fix $(x_0,y_0)\in \RR^2\times \RR^2$ and $\delta, \eta>0$, and suppose that $\mathbb{V}(x,y_0)<\eta$ for all $x\in \overline{B}_{\delta}(x_0)$.
 The proof of (ii) implies that for any $z\in \overline{B}_{\delta}(x_0)$, there exist constants $\delta_z, T_z\in(0,\infty)$ such that, for any $k\in B_{\delta_z}(z)$, there exists a $T^k\in(0,T_z]$ and $\psi^k\in C_{T^k}$ with $\psi^k(0)=k$ and $\psi^k(T^k)=y_0$ such that $I^{T^k}_k(\psi^k)<\eta$. Applying the Heine-Borel Covering theorem, it is not difficult to see that (iii) holds.


The proof of Lemma \ref{PropQ 0808} is complete.
\end{proof}
\vskip 3mm
For $x\in \mathbb{R}^d$, let $\Psi_t(x)=\Psi(t,x)$ denote the solution of the system $\dot{Y}=b(Y)$ with $Y(0)=x$. We now recall the definitions of $\alpha$, $\omega$ limit sets:
$$
\alpha(x)=\{p\in \mathbb{R}^d:  \mbox{there exists a sequence} \,\, t_n \,\,\mbox{ such that}\,\, t_n\rightarrow -\infty\,\, \mbox{ and that}\,\,  \Psi(t_n,x)\rightarrow p\};
$$
$$
\omega(x)=\{p\in \mathbb{R}^d:  \mbox{there exists a sequence} \,\, t_n \,\, \mbox{ such that}\,\, t_n\rightarrow +\infty \,\, \mbox{ and that}\,\,  \Psi(t_n,x)\rightarrow p\}.
$$
The Birkhoff center, $B(\Psi)$, of $\Psi$ is defined to be the closure of all periodic orbits of $\Psi$. By \cite[Theorem 3.1]{Chen2020}, ${\rm supp(}\mu)\subset B(\Psi)$ for any limiting measure $\mu$ of the invariant measures of the system (\ref{11}).
The following concepts are adopted from \cite [Chapter 6]{FW}. We define an equivalent relation $\sim$ between points of $\mathbb{R}^d$:
$$
x\sim y\Leftrightarrow \mathbb{V}(x,y)=0=\mathbb{V}(y,x).
$$
A set $K$ is called an {\it equivalent class} if $\mathbb{V}(x,y)=0$ for any $x,y\in K$.
We note that if the quasipotential $\mathbb{V}$ is continuous at any point of a limit set of the system $\dot{x}=b(x)$ then it is an equivalent class. The precise meaning  of the assumption (${\bf A}_4$) in Section 1 is the following:\\
 There exists a finite number of compact equivalent classes $K_1, K_2, \cdots, K_l$ such that\\
 (1) if $x\in K_i$ and $y\notin K_i$, then $x\nsim y$;\\
 (2) every $\omega-$ limit set(also $\alpha-$ limit set) of the dynamical  system $\dot{x}=b(x)$ is contained in one of the $K_i$.
 \vskip 0.3cm
 The condition (1) means that each equivalent class $K_i$ is maximal with respect to the  relation $\sim$.

 Let $L:=\{K_1, K_2, \cdots, K_l\}$. For simplicity, we denote  $L=\{1, 2, \cdots, l\}$. Let $W\subset L$. A graph consisting of arrows $m\rightarrow n (m\in L\backslash W, n\in L, n\neq m)$
 is called a {\it $W$-graph} if it satisfies the following
conditions:\\
(1) every point $m\in L\backslash  W$ is the initial point of exactly one arrow;\\
(2) there are no closed cycles in the graph.
\vskip 0.3cm
We denote by $G(W)$ the set of $W$-graphs. For each $i\in L$, define
$$W(K_i) := {\rm min}_{g\in G(i)}\sum_{(m\rightarrow n)\in g}\mathbb{V}(K_m,K_n).$$
Suppose that (${\bf A}_3$) and (\ref{E}) hold. Then  it follows from Khasminskii \cite [Chapter 4]{ KhasminskiiB} that the system (\ref{11}) admits a unique invariant measure $\mu^{\varepsilon}$ for $0<\vare \le \vare_0$. Moreover, the condition
\begin{equation}\label{DissiP}
\langle b(x),\nabla V\rangle \le -\gamma,\ {\rm for}\ |x|\gg 1,
\end{equation}
implies that the system $\dot{x}=b(x)$ is dissipative. Furthermore, if Assumptions {\rm \ref{ass1}}, {\rm\ref{ass2}} and (${\bf A}_4$) are satisfied, then ULDP,  the Lemmas on  estimating the
escape time from a domain in \cite [Chapter 6]{FW} and Khasminskii's formula (4.14) in \cite [Theorem 3.5, p.75]{KhasminskiiB} hold. Following the  proofs of \cite [Theorem 4.1, p.166]{FW}, \cite [Theorem 4.2, p.167]{FW}, we  deduce the following:
\begin{thm}\label{FW}
 Suppose that Assumptions {\rm \ref{ass1}},  {\rm\ref{ass2}}, {\rm(}${\bf A}_3${\rm)}, {\rm(}${\bf A}_4${\rm)} and \eqref{E} are satisfied. Then for any $\gamma>0$ there exist $0<\vare^*<\vare_0$ and $\rho>0$ {\rm(}which can be chosen arbitrarily small{\rm)} such that when $0<\vare<\vare_*$ and $i\in L$,
 $$
 \exp\{-\vare^{-1}\left(W(K_i)-{\rm min}_{j\in L}W(K_j)+\gamma\right)\}\le\mu^{\vare}(B(K_i,\rho))\le \exp\{-\vare^{-1}\left(W(K_i)-{\rm min}_{j\in L}W(K_j)-\gamma\right)\},
 $$
 where $B(K_i,\rho)$ denotes the $\rho$-neighborhood of $K_i$.
 Let
 $$\nu={\rm min}_{j\in L}\{W(K_j)\}\ {\rm and}\ L_0:=\{i_1,\ i_2,\cdots,\ i_k\}=\{i\in L:W(K_i)=\nu\}. $$
 If $\mu^{\vare_k}$ converges weakly to $\mu$ as $\vare_k\rightarrow 0$, then the support of $\mu$, ${\rm supp(}\mu)$,  is contained in  $\bigcup_{j=1}^kK_{i_j}$. In particular, if $L_0=\{i_0\}$ and there exists only one normalized invariant measure $\mu_0$ of the system $\dot{x}=b(x)$ whose support is contained in $K_{i_0}$, then $\mu^{\vare}$ converges weakly to $\mu_0$ as $\vare\rightarrow 0$.
 \end{thm}

In view of Theorem \ref{thmULDP}, following the proofs in the paper \cite {Hwang} by Hwang and Sheu  we obtain the following result which is applicable to unbounded coefficients.

\begin{thm}\label{HW}
 Suppose that Assumptions {\rm \ref{ass1}}, {\rm\ref{ass2}}, {\rm(}${\bf A}_3${\rm)}, {\rm(}${\bf A}_4${\rm)} and \eqref{E} are satisfied. Then for any $\alpha>0, \exists\ \delta>0$
 such that for any compact set $F$ there is $0<\vare^*_0<\vare^*$, the following holds
 $$\left |\EE_x[f(X^{\vare}(T))]-\int f(y)\mu^{\vare}(dy) \right |\le \|f\|_{\infty}\exp(-\vare^{-1}\delta)$$
 where
 $$x\in F,\quad \quad f\in C_b(\RR^d), \quad \quad 0<\vare \le \vare_0^*, \quad \quad T=\exp\left(\vare^{-1}(\Lambda+\alpha)\right)\ {\rm with}$$
 $$\Lambda=\nu- {\rm min}\{\mathbb{V}(g):g\in G(\{i,j\}), i, j\in L, i\neq j\}.$$
 \end{thm}
 A subset $\mathcal{R}\subset \mathbb{R}^d$ is said to be a
{\it repeller} (an {\it attractor}) for the solution flow $\Psi$ of $\dot{x}=b(x)$ provided:
(i) $\mathcal{R}$ is nonempty, compact and invariant; and
(ii) $\mathcal{R}$ has a neighborhood $N\subset \mathbb{R}^d$, called a {\it fundamental neighborhood} of $\mathcal{R}$,  such that
$\displaystyle\lim_{t\rightarrow-\infty}{\rm dist}\big(\Psi_{t}(x),\mathcal{R}\big)=0$ ($\displaystyle\lim_{t\rightarrow +\infty}
{\rm dist}\big(\Psi_{t}(x),\mathcal{R}\big)=0$) uniformly in $x\in N$. In the case of attractor, we call $\mathcal{R}$
{\it Lyapunov stable}.

 Substituting (${\bf A}_1$) by the Assumptions {\rm \ref{ass1}} and {\rm\ref{ass2}}, following the same proof of \cite[Theorem 3.1, Corollary 3.3, p.78]{XCJ} we obtain the following two results, which significantly extend the scope of applications.

\begin{thm}\label{saddle}
 Suppose that Assumptions {\rm \ref{ass1}}, {\rm\ref{ass2}},  {\rm(}${\bf A}_3${\rm)} and \eqref{E} hold. Suppose that there exist $m$ points $x_1, x_2, \cdots, x_m$ in $\mathbb{R}^d$, an attractor $\mathcal{A}=:\mathcal{S}_{m+1}$ and compact equivalent classes $\mathcal{S}_1, \mathcal{S}_2, \cdots, \mathcal{S}_m$  such that $\mathcal{S}_i\supset\alpha(x_i), i=1, 2,\cdots, m$, $\{\mathcal{A}, \mathcal{S}_i, i=1, 2,\cdots, m\}$ are pairwise  disjoint, $\omega(x_i)\subset \mathcal{S}_{i+1}, i=1, 2,\cdots, m$.
 Then there exist neighborhoods $U_i$ of $\mathcal{S}_i, i=1, 2, \cdots, m$, $\kappa>0$ and $0<\varepsilon_*\le \vare_0$
such that for any $\vare\in(0,\varepsilon_*)$, we have
\[ \mu^{\vare}(U_i)\leq \exp\{-\kappa/ \vare\}  .   \]
As a result, if $\mu^{\vare_j}$  converges weakly $\mu$ as $\vare_j\rightarrow 0$,
then $\mu(U_i)=0, i=1, 2, \cdots, m$.
\end{thm}
\begin{rmk}\label{rmk3.1}
Note that if $m=1$ and $\mathcal{S}_1$ is a repeller of $\dot{x}=b(x)$ then  its dissipation implies that there must be an attractor $\mathcal{S}_2$  such that $\dot{x}=b(x)$  admits a connecting orbit connecting  $\mathcal{S}_1$ and $\mathcal{S}_2$. Therefore, limiting measure always stays away from any repelling equivalent class.
\end{rmk}


\section{Applications}
In this part, we will provide several interesting models to give the precise description of the support of the limiting  measures of the invariant measures of corresponding SDEs (\ref{11}) using the results in the previous sections.
\subsection {The realization of Freidlin's two phase portraits of dynamical systems}
In the book \cite [Figure 12, p.151]{FW}, Freidlin and Wentzell depicted the phase portraits of some dynamical system (see Figure \ref{fig1} below) in which there exist exactly four equivalent classes. They designed the following transitive difficulty matrix
\begin{displaymath}
    \ \big(\mathbb{V}(K_i,K_j)\big)_{4\times 4}:= \left(
        \begin{array}{ccccc}
          & 0 & 0 & 9 & 9 \\
          & 1 & 0 & 9 & 9 \\
          & 7 & 6 & 0 & 6 \\
          & 1 & 0 & 0 & 0
        \end{array}
        \right )
\end{displaymath}
and used it to illustrate their estimates of the transition probabilities between the equivalent classes, the concentration of the limiting measures of invariant measures (see \cite [ p.170]{FW}), the exit position on the boundary of a domain (see \cite [ p.175]{FW}) and the asymptotics of the mathematical expectation of the exit time from a domain (see \cite [ p.177]{FW}). However, they did not provide examples of stochastic planar dynamical systems that have the phase portraits in Figure 1, which is a nontrivial matter. They did not  carry out the computation of the matrix $\big(\mathbb{V}(K_i,K_j)\big)_{4\times 4}$ for concrete systems either,  which is actually a difficult task for a given system of SDEs.

In the following, we first present a planar polynomial system to realize the global phase portraits of  Figure \ref{fig1} and then precisely compute the transition difficulty matrix $\big(\mathbb{V}(K_i,K_j)\big)_{4\times 4}$. As a consequence, we are able to identify the supports of the limiting measures of the invariant measures of the perturbed systems.
\begin{example}
Let $x:=(x_1,x_2)$ and consider the following deterministic planar polynomial system
\begin {equation}\label{DS}
\frac{dx}{dt}=b_1(x):=\left(x_2-F'(x_1)\big(H(x_1,x_2)+1\big),-F'(x_1)-x_2\big(H(x_1,x_2)+1\big)\right)
\end{equation}
 where
$$F(x_1):= \frac{x_1^4}{4}+\frac{x_1^3}{3}-x_1^2,\ F'(x_1)=x_1(x_1-1)(x_1+2),\ H(x_1,x_2):= \frac{x_2^2}{2}+F(x_1).$$
By the chain rule,
$$\frac{dH(x_1,x_2)}{dt}=-\big(H(x_1,x_2)+1\big)|\nabla H|^2.$$
Therefore, the equilibria of the system (\ref{DS}) are
$$K_1(-2,0),\ K_3(1,0),\ K_4(0,0)$$
and (\ref{DS}) admits the unique limit cycle
$$K_2: H(x_1,x_2)+1=0.$$
\end{example}

By linearized technique, we can prove that $K_1$ is an unstable node, $K_4$ is a saddle, $K_3$ is a stable focus and $K_2$ is a stable limit cycle.  Applying the LaSalle invariant principle to the interior and exterior of $K_2$, respectively, we get the global portraits depicted in Figure 1,
\begin{figure}[ht]
  \centering
  \includegraphics[width=0.6\textwidth]{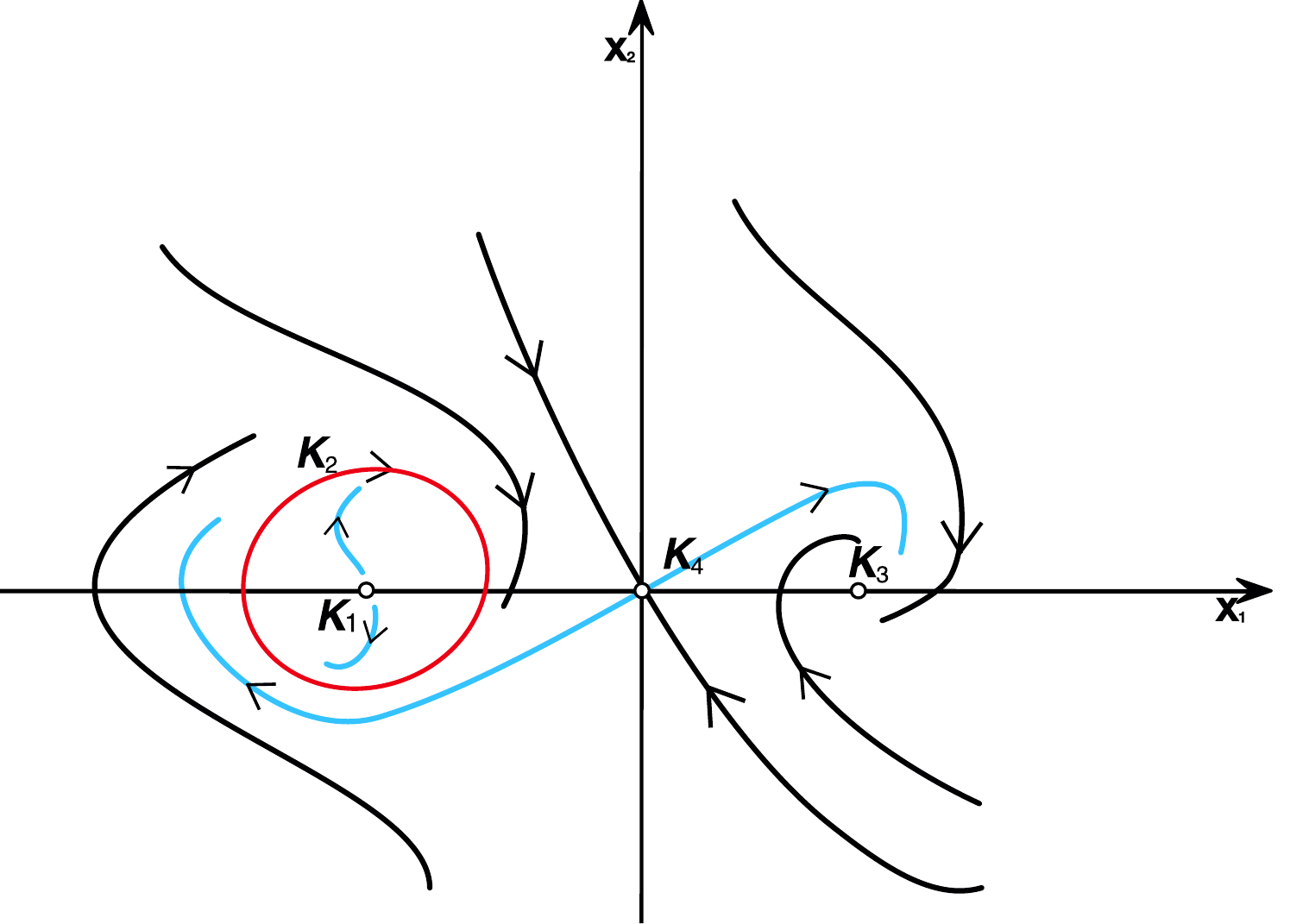}\\
  \caption{The phase portrait of system (\ref{DS}).}\label{fig1}
\end{figure}
which is the realization of \cite[Figure12, p.151]{FW}. It is easy to see that
\[
H(K_1)= -\frac{8}{3}= {\rm min}\{H(x_1,x_2)|(x_1,x_2)\in \RR^2\},\ H(K_2)=-1,\  H(K_3)=-\frac{5}{12},\ H(K_4)= 0.
\]

Now we consider the random  perturbation of (\ref{DS}) by an additive noise:
\begin {equation}\label{SS}
dx^\vare(t)=b_1(x^\vare(t))dt +\sqrt{\vare}dB(t),
\end{equation}
where $\{B(t)\}_{t \geq 0}$ is a $2$-dimensional Brownian motion. To characterize the support of the limiting measures of the invariant measures of the system (\ref{SS}), we verify the various conditions listed in the theorems in Section 3. The Assumptions  \ref{ass1} and the condition (${\bf A}_4$) are obviously satisfied. Thus, we only need to verify the Assumptions  \ref{ass2} and the condition (${\bf A}_3$).  Define a Lyapunov function $V(x_1,x_2)=H(x_1,x_2)+3>0$. Clearly, $\lim_{|(x_1,x_2)|\rightarrow \infty} V(x_1,x_2)=\infty$   and
$$\text{Trace}\big(\sigma^{*}(x_1,x_2)\nabla^2 V(x_1,x_2)\sigma(x_1,x_2)\big)=3x_1^2+2x_1-1>-2$$
which means that (\ref{23}) holds. Besides, the left-hand of (\ref{22}) with $\theta=2$ and $\eta=1$ is
$$
-(V(x_1,x_2)-2)|\nabla V|^2+(3x_1^2+2x_1-1)+\frac{|\nabla V|^2}{V}=:J_1+J_2+J_3.
$$
We can verify that
$$
\lim_{|(x_1,x_2)|\rightarrow \infty}J_1=-\infty,\ \lim_{|(x_1,x_2)|\rightarrow \infty}\frac{J_i}{J_1}=0\,\, {\rm for}\,\, i=2,3.
$$
Thus there exists a positive constant $C$ such that
$$J_1+J_2+J_3\le \big(\frac{1}{2}J_1+C\big) \rightarrow -\infty\ {\rm as}\ |(x_1,x_2)|\rightarrow \infty$$
which implies that both (\ref{22}) and (${\bf A}_3$) hold. Thus we have proved that Assumptions  \ref{ass1} and \ref{ass2}, (${\bf A}_3$) and (${\bf A}_4$) are satisfied.

From the phase portraits for (\ref{DS}) in Figure \ref{fig1}, we see that
\begin{displaymath}
    \ \big(\mathbb{V}(K_i,K_j)\big)_{4\times 4}:= \left(
        \begin{array}{ccccc}
          & 0 & 0 & a & a \\
          & b & 0 & a & a \\
          & b+c & c & 0 & c \\
          & b & 0 & 0 & 0
        \end{array}
        \right )
\end{displaymath}
where
\begin{equation*}\label{TP}
a:= \mathbb{V}(K_2,K_4),\  b:= \mathbb{V}(K_2,K_1),\ c:= \mathbb{V}(K_3,K_4).
\end{equation*}
We shall prove that
\begin{equation}\label{TP}
 \mathbb{V}(K_2,K_4)=1,\   \mathbb{V}(K_2,K_1)=\frac{25}{9},\  \mathbb{V}(K_3,K_4)=\frac{95}{144}.
\end{equation}
Let
\[
U(x_1,x_2):=\frac{H^2(x_1,x_2)}{2} + H(x_1,x_2),\ l(x_1,x_2):=(x_2, -x_1(x_1-1)(x_1+2)).
\]

Then $\langle \nabla U(x_1,x_2),l(x_1,x_2)\rangle\equiv 0$ and (\ref{DS}) can be rewritten as a quasipotential system
\begin {equation}\label{DS1}
\frac{dx}{dt}=b_1(x)= - \nabla U(x_1,x_2) + l(x_1,x_2).
\end{equation}
Note that
$$U(K_1)=\frac{8}{9},\ U(K_2)=-\frac{1}{2},\ U(K_3)=-\frac{95}{288},\ U(K_4)=0.$$
Following the arguments from Section 3 of \cite [Chapter 4]{FW}, we proceed as follows. For any $T>0$ and $\varphi\in {\bf AC}_X^T$ with $\varphi_T=Y$, we have
\begin{align*}
I_{X}^T(\varphi)
  =&\frac{1}{2}\int_{0}^{T}|\dot{\varphi}_s+\nabla U(\varphi_s)-l(\varphi_s)|^2ds\\
=&\frac{1}{2}\int_{0}^{T}|\dot{\varphi}_s-\nabla U(\varphi_s)-l(\varphi_s)|^2ds+2\int_{0}^{T}\langle\dot{\varphi}_s, \nabla U(\varphi_s) \rangle ds\\
=& \frac{1}{2}\int_{0}^{T}|\dot{\varphi}_s-\nabla U(\varphi_s)-l(\varphi_s)|^2ds+2(U(Y)-U(X))\\
\ge &  2(U(Y)-U(X)).
\end{align*}
This shows that for any $X, Y\in \RR^2$
\begin{equation}\label{RFE}
I_{X}^T(\varphi)=\frac{1}{2}\int_{0}^{T}|\dot{\varphi}_s-\nabla U(\varphi_s)-l(\varphi_s)|^2+2(U(Y)-U(X))
\end{equation}
and
\begin{equation}\label{LB}
\mathbb{V}(X, Y)\ge 2(U(Y)-U(X)).
\end{equation}
In particular, letting $Y=K_1$ and $X$ lie  in the interior of $K_2$, we have
$$\mathbb{V}(X, K_1)\ge 2\Big(\frac{8}{9}-U(X)\Big)$$
from which and the continuity of $\mathbb{V}(X,K_1)$ it follows that
$$\mathbb{V}(K_2, K_1)\ge 2\Big(\frac{8}{9}-U(K_2)\Big)=\frac{25}{9}.$$

In order to study the extremals of $I(\varphi)$ from $K_2$ to $K_1$, we have to consider the following equations for the extremals
\begin {equation}\label{EE}
\frac{dX}{dt}= \nabla U(x_1,x_2) + l(x_1,x_2)
\end{equation}
whose equilibria still are
$$K_1(-2,0),\ K_3(1,0),\ K_4(0,0)$$
and $K_2$ is the limit cycle of (\ref{EE}). Its phase portraits are drawn in Figure \ref{fig2} below.
\begin{figure}[h]
  \centering
  \includegraphics[width=0.6\textwidth]{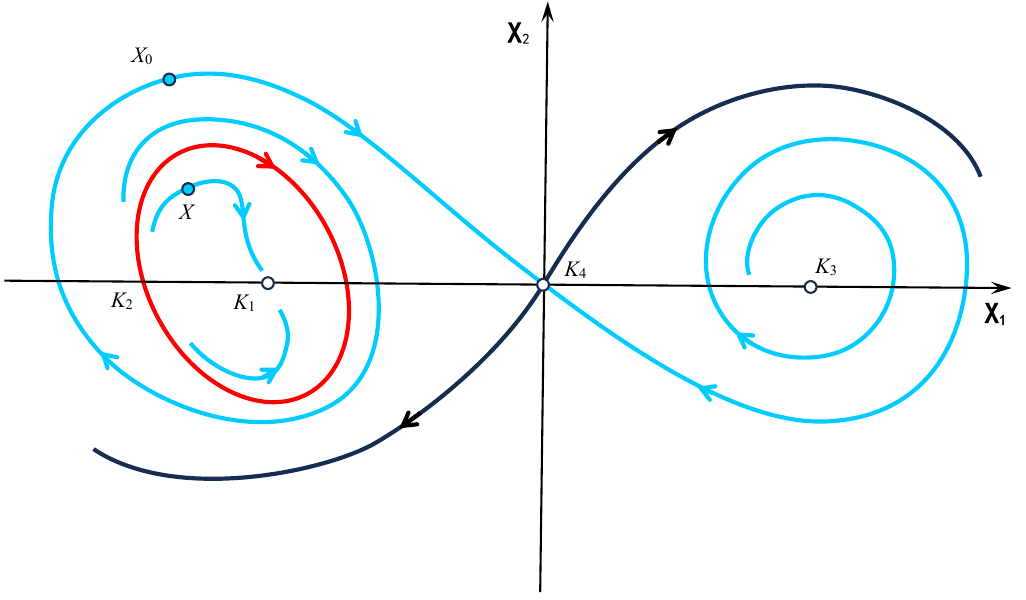}\\
  \caption{The phase portrait of system (\ref{EE}).}\label{fig2}
\end{figure}

Let $X\neq K_1$ be a point lying inside $K_2$ and denote by $\varphi_t, t\in \RR$ the solution of (\ref{EE}) passing through $X$. From Figure \ref{fig2}, $\alpha(X)=K_2$ and $\omega(X)=K_1$. For any $T>0$, $$2(U(\varphi_T)-U(\varphi_{-T}))=I(\varphi)\ge \mathbb{V}(\varphi_{-T},\varphi_T).$$
Let $T\rightarrow \infty$ to obtain that
$$\frac{25}{9}=2(U(K_1)-U(K_2))\ge \mathbb{V}(K_2, K_1).$$
This proves that $b=\mathbb{V}(K_2, K_1)=\frac{25}{9}$.

Fix $Y=K_4$ in (\ref{LB}), we get
$$\mathbb{V}(X, K_4)\ge 2(U(K_4)-U(X))=-2U(X).$$
From the continuity of $\mathbb{V}(X, K_4)$ it follows that
$$\mathbb{V}(K_2, K_4)\ge -2U(K_2)=1.$$
Choose $X_0\neq K_4$ in the stable manifold of $K_4$ illustrated in Figure \ref{fig2} and let $\varphi_t, t\in \RR$ be the solution of (\ref{EE}) passing through $X_0$. From Figure \ref{fig2}, $\alpha(X_0)=K_2$ and $\omega(X_0)=K_4$. For any $T>0$, $$2(U(\varphi_T)-U(\varphi_{-T}))=I(\varphi)\ge \mathbb{V}(\varphi_{-T},\varphi_T).$$
Letting $T \rightarrow \infty$, we have $1\ge \mathbb{V}(K_2, K_4)$, which means that $\mathbb{V}(K_2, K_4)=1$. Analogously, we can prove that $\mathbb{V}(K_3,K_4)=\frac{95}{144}$.

Finally, from (\ref{TP}) we conclude that
\begin{displaymath}
    \ \big(\mathbb{V}(K_i,K_j)\big)_{4\times 4}:= \left(
        \begin{array}{ccccc}
          & 0 & 0 & 1 & 1 \\
          & \frac{25}{9} & 0 & 1 & 1 \\
          & \frac{55}{16}& \frac{95}{144} & 0 & \frac{95}{144} \\
          & \frac{25}{9} & 0 & 0 & 0
        \end{array}
        \right ).
\end{displaymath}
Let $\mu^{\varepsilon}$ be the invariant measure of the solution of the SDEs (\ref{SS}) for $0<\varepsilon\ll 1$. Then for any $\varepsilon_i\rightarrow 0$ such that $\mu^{\varepsilon_i}$  converges weakly to $\mu$, Theorem \ref{saddle} and Remark \ref{rmk3.1} imply that $\mu(K_1)= \mu(K_4)=0$. Because $\mathbb{V}(K_3,K_2)=\frac{95}{144}<1=\mathbb{V}(K_2,K_3)$, it follows from Theorem \ref{FW} that $\mu^{\varepsilon}$  converges weakly to the unique ergodic measure $\mu_{K_2}(\cdot):=\frac{1}{T}\int_0^T\delta_{\varphi_t(X^*)}(\cdot)dt$ supported on $K_2$ as $\varepsilon\rightarrow 0$, where $X^*\in K_2$ and $T$ is the period of the limit cycle $K_2$. The density for $\mu_{K_2}(\cdot)$ is

$$m_{K_2}(X)=\frac{1}{T}\frac{1}{|\nabla H(X)|},\ X:=(x_1,x_2)\in K_2.$$

Consider the following
Cauchy problem $\mathbb{R}_+\times \mathbb{R}^2$:
\begin{equation}\label{AL}
   \begin{cases}
    \frac{\partial u^{\epsilon}(t,X)}{\partial t}=\frac{\epsilon}{2}\Delta u^{\epsilon}(t,X)+\langle b(X), \nabla u^{\epsilon}(t,X)\rangle,\\
     u^{\epsilon}(0,X)=g(X)\in C_b(\mathbb{R}^2).
\end{cases}
\end{equation}
Denote by $\mathcal{B}(K_2)$ and $\mathcal{B}(K_3)$ the attracting domains of $K_2$ and $K_3$, respectively. Define
$$g_{K_2}:=\frac{1}{T}\oint_{K_2}g(X)\frac{1}{|\nabla H(X)|}dl.$$
Here the integral is respect to arc measure. Then from the discussion above and \cite [Theorem 2.2]{FW2}, we derive the following.
\begin{prop}\label{ALRDE}
$\mathbb{V}(K_3,K_2)=\frac{95}{144}, \mathbb{V}(K_2,K_3)=1$ and $\mu^{\varepsilon}$  converges weakly to the unique ergodic measure $\mu_{K_2}(\cdot)$ as $\varepsilon\rightarrow 0$.

Let $\lim_{\epsilon\rightarrow 0} T(\epsilon)=+\infty$. Then we have \\
{\rm (1)} If $X\in \mathcal{B}(K_3)$, then
\begin{equation*}
   \lim_{\epsilon\rightarrow 0}u^{\epsilon}(T(\epsilon),X)=\begin{cases}
    g(K_3), & \text{if ~$\limsup_{\epsilon \rightarrow 0}\epsilon \log T(\epsilon)<\frac{95}{144}$},\\
       g_{K_2}, & \text{if ~$\liminf_{\epsilon \rightarrow 0}\epsilon \log T(\epsilon)>\frac{95}{144}$}.
\end{cases}
\end{equation*}\\
{\rm (2)} If $X\in \mathcal{B}(K_2)$, then\\
$$\lim_{\epsilon\rightarrow 0}u^{\epsilon}(T(\epsilon),X)=g_{K_2}.$$
\end{prop}

\begin{example} \label{Ex4.2}
Let $X:=(x_1,x_2)$ and consider the following deterministic planar polynomial system
\begin {equation}\label{DS2}
\frac{dX}{dt}=b_2(X):=\left(x_2-F'(x_1)\big(H(x_1,x_2)+1\big),-F'(x_1)-x_2\big(H(x_1,x_2)+1\big)\right)
\end{equation}
 where
$$F(x_1):= \frac{x_1^4}{4}-\frac{x_1^3}{3}-x_1^2,\ F'(x_1)=x_1(x_1+1)(x_1-2),\ H(x_1,x_2):= \frac{x_2^2}{2}+F(x_1)$$
and its random perturbed system
\begin {equation}\label{SS2}
dX^\vare(t)=b_2(X^\vare(t))dt +\sqrt{\vare}dB(t).
\end{equation}
\end{example}
\noindent The system (\ref{DS2}) admits three equilibria
$$K_1(-1,0),\ K_3(0,0),\ K_4(2,0)$$
and the stable limit cycle
$$K_2: H(x_1,x_2)+1=0.$$
The global phase portraits of (\ref{DS2}) are depicted in Figure \ref{fig3} below,
\begin{figure}[h]
  \centering
  \includegraphics[width=0.5\textwidth]{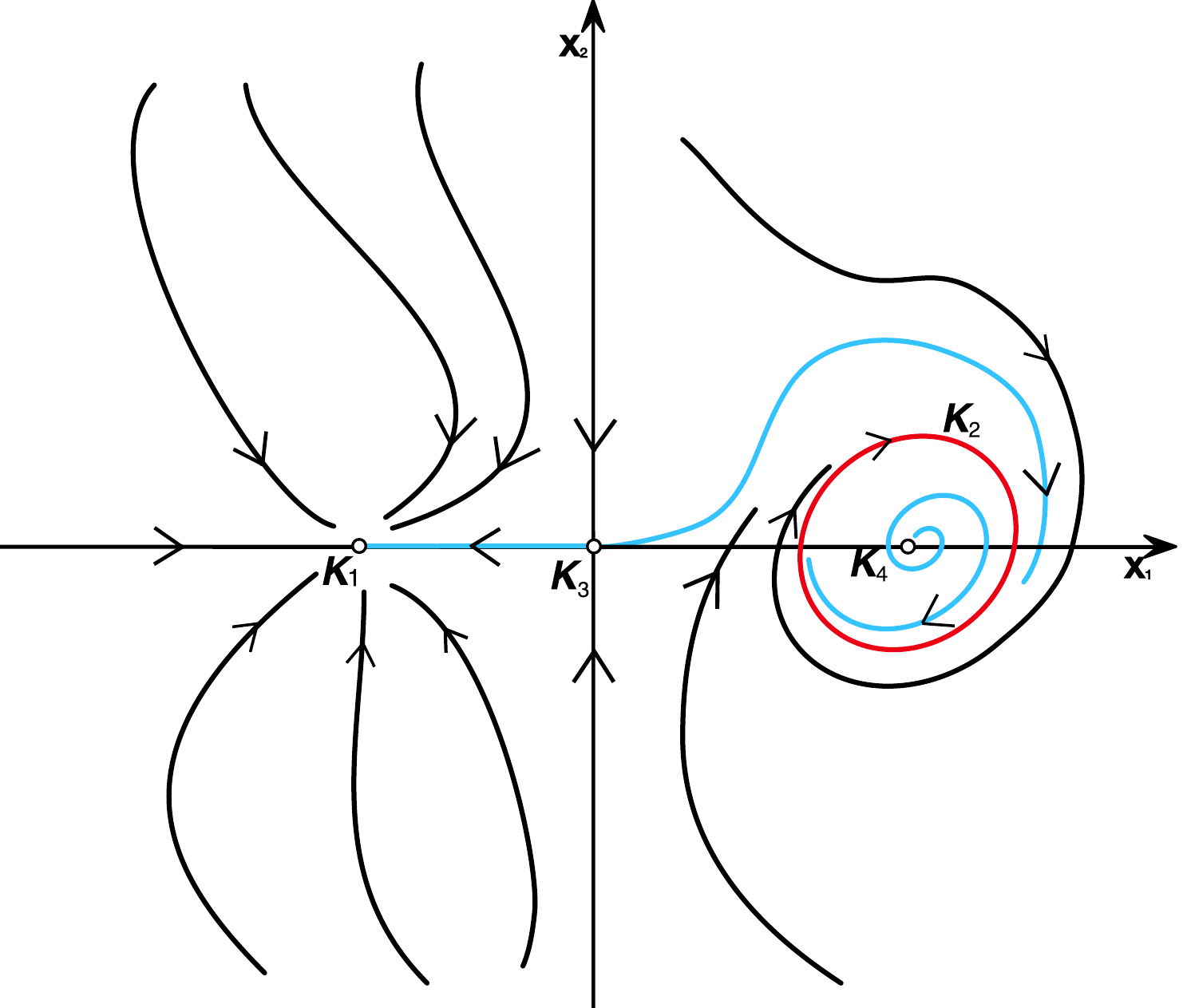}\\
  \caption{The phase portrait of system (\ref{DS2}).}\label{fig3}
\end{figure}
which is the realization of \cite [Figure 1(a), p.557]{FW2}. In the same way  as in Example 4.1, we can show that  the transition difficulty matrix is given by
\begin{displaymath}
    \ \big(\mathbb{V}(K_i,K_j)\big)_{4\times 4}:= \left(
        \begin{array}{ccccc}
          & 0 & \frac{95}{144} & \frac{95}{144} & \frac{55}{16} \\
          & 1 & 0 & 1 & \frac{25}{9} \\
          & 0 & 0 & 0 & \frac{25}{9} \\
          & 1 & 0 & 1 & 0
        \end{array}
        \right ).
\end{displaymath}
Let $\mu^{\varepsilon}$ be the invariant measure of the SDEs (\ref{SS2}) for $0<\varepsilon\ll 1$. Then  $\mu^{\varepsilon}$  converges weakly to the unique ergodic measure $\mu_{K_2}(\cdot):=\frac{1}{T}\int_0^T\delta_{\varphi_t(X^*)}(\cdot)dt$ supported on $K_2$ as $\varepsilon\rightarrow 0$, where $X^*\in K_2$ and $T$ is the period of the limit cycle $K_2$.
\vskip 0.3cm
Furthermore, the same result as Proposition \ref{ALRDE} holds  with $K_3$ replaced by $K_1$.
%
%
%

\subsection {Stochastic van der Pol equation and single diode circuit}\label{Section 4.2}
\begin{example}
van der Pol (1927) established the following well known triode circuit equation
$$X^{(2)} + (X^2-1)X^{(1)} +X=0$$
which is equivalent to
\begin{equation}\label{DvdPS}
\begin{cases}
    X^{(1)}= Y, \\
    Y^{(1)}=-[(X^2-1)Y +X]
\end{cases}
\end{equation}
and now called the van der Pol equation.  The Birkhoff center  is composed of the origin $O$ and the limit cycle $\Gamma$, where $O$ is a repeller and $\Gamma$ attracts all points except the origin $O$. The global portraits of (\ref{DvdPS}) are shown in Figure \ref{fig5}, see \cite{ZhangZhifen}.
\begin{figure}[h]
  \centering
  \includegraphics[width=0.4\textwidth]{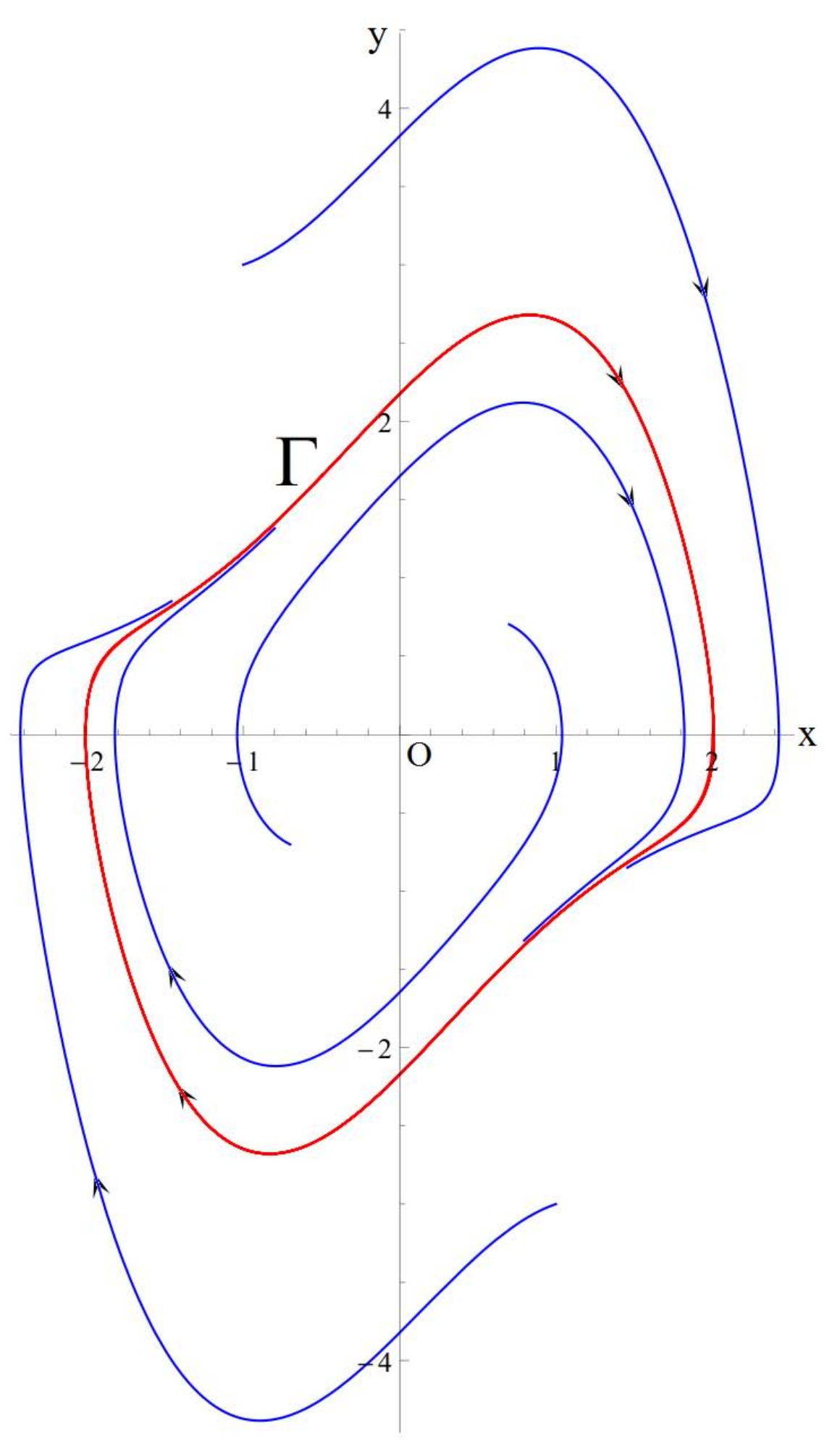}\\
  \caption{The unique limit cycle.}\label{fig5}
\end{figure}

Consider the random perturbation of the van der Pol equation:
\begin{equation}\label{SVDP}
X_{\varepsilon}^{(2)} + (X_{\varepsilon}^2-1)X_{\varepsilon}^{(1)} +X_{\varepsilon}=\sqrt{\varepsilon}\sigma(X_{\varepsilon},X_{\varepsilon}^{(1)})\dot{B}_t,\ \sigma(X_{\varepsilon},X_{\varepsilon}^{(1)})>0.
\end{equation}
Let $Y_{\varepsilon}:=(Y_{\varepsilon}^1,Y_{\varepsilon}^2) := (X_{\varepsilon},X_{\varepsilon}^{(1)})$. Then (\ref{SVDP}) is equivalent to the following system of SDEs:
\begin{equation}\label{SvdPS}
\begin{cases}
    dY_{\varepsilon}^{1}(t) = Y_{\varepsilon}^{2}(t)dt, \\
    dY_{\varepsilon}^{2}(t)=-[(Y_{\varepsilon}^{1}(t)^2-1)Y_{\varepsilon}^{2}(t) +Y_{\varepsilon}^{1}(t)]dt+\sqrt\eps \si\big(Y_{\varepsilon}^{1}(t),  Y_{\varepsilon}^{2}(t)\big) dB(t),\\
    (Y_{\varepsilon}^{1},  Y_{\varepsilon}^{2})(0) = (x_1,  x_2)=:x \in \RR^2.
\end{cases}
\end{equation}
If the first equation of (\ref{SvdPS}) were also perturbed by a Brownian motion, independent of $B(t)$, then one could prove that the limiting measure of invariant measures for corresponding system (\ref{SvdPS}) is the Haar measure of $\Gamma$ by Theorem \ref{saddle} and Remark \ref{rmk3.1}, the readers are referred to the discussion below. However, for system (\ref{SvdPS}), which has interesting physical background, we are not able to prove that its quasipotential $\mathbb{V}(x,y)$ is continuous on $\mathbb{R}^2\times \mathbb{R}^2$, so  it is not straightforward to get the above conclusion. But,  with the help of Lemma \ref{PropQ 0808}, we still can prove the same result.

Let $b(Y^1,Y^2)$ and $(0,\sigma(Y^1,Y^2))^*$ denote the drift  and diffusion coefficient of (\ref{SvdPS}), respectively.
According to Nevelson's, as in \cite [p.82]{KhasminskiiB}, we construct the following Lyapunov function for (\ref{SvdPS})
\begin{equation}\label{LyaF}
V(Y^1,Y^2)=\frac{\sqrt{5}-1}{24}(Y^1)^4+\frac{1}{2}\left(Y^2+\frac{(Y^1)^3}{3}-\frac{\sqrt{5}+1}{2}Y^1\right)^2.
\end{equation}
We have
$$V_{Y^1}=\frac{(Y^1)^5}{3}-\frac{3\sqrt{5}+5}{6}(Y^1)^3+(Y^1)^2Y^2+\frac{\sqrt{5}+3}{2}Y^1-\frac{\sqrt{5}+1}{2}Y^2;$$
$$V_{Y^2}=\frac{(Y^1)^3}{3}-\frac{\sqrt{5}+1}{2}Y^1+Y^2;$$
$$\langle b(Y^1,Y^2),\nabla V(Y^1,Y^2) \rangle = -\frac{(Y^1)^4}{3}+\frac{\sqrt{5}+1}{2}(Y^1)^2-\frac{\sqrt{5}-1}{2}(Y^2)^2.$$
We have the following result
\begin{prop}\label{SvdP}
Suppose that $\si(Y^1,Y^2)$ is a positive,  locally Lipschitz continuous function  satisfying $\si^2(Y^1,Y^2) \leq c_1((Y^1)^4 + (Y^2)^2 +1)$, $c_1>0$. Then there exists $\varepsilon_0>0$ such that for any given $\varepsilon\in (0,\varepsilon_0]$, the set $\mathcal{J^{\varepsilon}}$ of invariant measures of the system {\rm(\ref{SvdPS})} is nonempty and $\mathcal{J}:=\bigcup_{0<\varepsilon \le \varepsilon_0}\mathcal{J^{\varepsilon}}$ is tight. Furthermore, let $\mu^{\varepsilon_i}\in \mathcal{J}^{\varepsilon_i}$ such that $\mu^{\varepsilon_i}$  converges weakly to $\mu$ as $\varepsilon_i\rightarrow 0$, then $\mu=\mu_{\Gamma}$, where $\Gamma$ is the unique limit cycle  of the deterministic van der Pol equation  {\rm(\ref{DvdPS})} illustrated in Figure {\rm\ref {fig5}} and $\mu_{\Gamma}$ is its corresponding Haar measure supported on $\Gamma$.
\end{prop}
\begin{proof}
Let $ 0<\varepsilon \le \varepsilon_0:=\frac{1}{6c_1}$. By the assumption of $\sigma$, there exists a positive constant $R$ such that for $  (Y^1)^2+(Y^2)^2\ge R,$
\DEQS
\mathcal{L}^{\varepsilon}V(Y^1,Y^2)
&:=& \langle b(Y^1,Y^2),\nabla V(Y^1,Y^2) \rangle+\frac{\varepsilon}{2 } \si^2(Y^1,Y^2)\\
&\leq&  -\frac{(Y^1)^4}{3}+\frac{\sqrt{5}+1}{2}(Y^1)^2-\frac{\sqrt{5}-1}{2}(Y^2)^2+\frac{\varepsilon c_1}{2 }\left((Y^1)^4 + (Y^2)^2 +1\right)\\
&\leq&  -\frac{(Y^1)^4}{4}+\frac{\sqrt{5}+1}{2}(Y^1)^2-\frac{(Y^2)^2}{4}+\frac{1}{12 }\\
&\leq&  -\frac{1}{5}\left((Y^1)^4 +(Y^2)^2\right),
\EEQS
where $\mathcal{L}^{\varepsilon}$ denotes the generator of the system (\ref{SvdPS}).
Denote by $\mathcal{J^{\varepsilon}}$ the set of all
 invariant measures of (\ref{SvdPS}) for a given $0<\varepsilon \le \varepsilon_0$. By \cite [Theorem 3.1]{Chen2020}, $\mathcal{J^{\varepsilon}}$ is nonempty and $\mathcal{J}:=\bigcup_{0<\varepsilon \le \varepsilon_0}\mathcal{J^{\varepsilon}}$ is tight. Furthermore, let $\mu^{\varepsilon_i}\in \mathcal{J}^{\varepsilon_i}$ such that $\mu^{\varepsilon_i}$  converges weakly to $\mu$ as $\varepsilon_i\rightarrow 0$. Then $\mu$ is an invariant measure of the  deterministic van der Pol equation {\rm(\ref{DvdPS})} and its support is contained in the Birkhoff center $\Gamma\cup\{O\}$.  The  origin of the  deterministic van der Pol equation is a repeller and the unique limit cycle $\Gamma$ is stable, which is illustrated in Figure \ref{fig5} (see \cite{ZhangZhifen}).

  In order to see that the system (\ref{SvdPS}) admits ULDP, we are  going to check that the Lyapunov function $V(Y^1,Y^2)$ constructed above satisfies the Assumption \ref{ass2}. \eqref{As b V} and \eqref{23} are obvious because
  $$\text{Trace}\big((0,\sigma(Y^1,Y^2))\nabla^2 V(Y^1,Y^2)(0,\sigma(Y^1,Y^2))^*\big) =\sigma^2(Y^1,Y^2)>0.$$
 For any given $ 0<\varepsilon \le \varepsilon_0$, let $\theta=\frac{\varepsilon}{2}$ and $\eta=\frac{8}{\varepsilon}$ in \eqref{22}, then we have the following estimate for the left-hand of \eqref{22}:
 $$J\le \mathcal{L}^{\varepsilon}V(Y^1,Y^2)\le -\frac{1}{5}\left((Y^1)^4 +(Y^2)^2\right),\ {\rm for}\  (Y^1)^2+(Y^2)^2\ge R$$
which implies that \eqref{22} holds. Applying Theorem \ref{thmULDP}, we conclude that the system (\ref{SvdPS}) admits ULDP.

Finally, we shall use Lemma \ref{PropQ 0808} and the idea to prove \cite[Theorem \ref{saddle}]{XCJ} to get that there is no concentration for limiting measures on the repeller $O(0,0)$. This will prove that when $\varepsilon\rightarrow 0$ the invariant measure $\mu^{\varepsilon}$ of {\rm (\ref{SvdPS})}  converges weakly to  the Haar measure $\mu_{\Gamma}$.

In order to keep the same notation as the proof of  \cite[Theorem 3.1]{XCJ}, we let $\mathcal{R}=O(0,0)$.

 From the proof of \cite [Lemma 4.1]{XCJ},  we know that \cite [Lemma 4.1]{XCJ} still holds as soon as ULDP  is satisfied. Thus, applying this lemma, we get that there exist $\delta_1,\ \delta^*_2\in(0,\delta),\ \delta_1>\delta^*_2$
and $s_0>0$ such that $\mathbb{V}(\partial \mathcal{R}_{\delta_1},\partial \mathcal{R}_{\delta^*_2})\geq s_0$. Set $\delta_2=\delta^*_2/2$. Applying \cite [Proposition 2.1]{XCJ}, we have
$$T_0=\inf\big\{u\geq 0:\Psi_{t}\big(\overline{\mathcal{R}_{\delta_1}}\backslash\mathcal{R}_{\delta_2}\big)
\subset(\mathcal{R}_{\delta})^c,t\geq u\big\}<+\infty.$$ Let $$F_0=\{\psi\in {\bf C}_{T_0}:\psi(0)\in \overline{\mathcal{R}_{\delta_1}}\backslash\mathcal{R}_{\delta_2}, \psi(T_0)\in \overline{\mathcal{R}_{\delta_1}}\}.$$
Then it is a closed subset of ${\bf C}_{T_0}$, $F_0^0$($:= \{ \varphi(0) :  \varphi \in F_0\}$)  is bounded and $F_0$ does not contain any solution of system {\rm(\ref{DvdPS})} by the definition of $T_0$.
By  \cite [Proposition 2.4]{XCJ}, we have
\[ s_1:= \inf\{I^{T_0}_{x}(\psi): \psi\in F_0,\ x=\psi(0)\in \overline{\mathcal{R}_{\delta_1}}\backslash\mathcal{R}_{\delta_2}\ \}=:I^{T_0}(F_0)>0.\]

%
Since $O$ is an equilibrium of {\rm(\ref{DvdPS})}, $\mathbb{V}(O,O)=0$ by its definition. Applying (ii) of Lemma \ref{PropQ 0808}, we know that $\mathbb{V}(x,y)$ is continuous at $(O,O)$. Thus, there exists $0<\delta_3<\delta_2$ such that
\begin{equation}\label{QC}
\mathbb{V}(O,y),\ \mathbb{V}(x,O)<\frac{1}{20}(s_0\wedge s_1)\ {\rm for\ all}\ x,y\in \overline{\mathcal{R}_{\delta_3}}.
\end{equation}
Using \cite [Proposition 2.1]{XCJ}, we get that
\begin{equation}\label{T3}
T_3=\inf\big\{u\geq 0:\Psi_{t}\big(\overline{\mathcal{R}_{\delta_1}}\backslash\mathcal{R}_{\delta_3}\big)
\subset(\mathcal{R}_{\delta})^c,\forall t\geq u\big\}<+\infty.
\end{equation}
 Then obviously $T_0\leq T_3$.

Fix a point $y_0\in \partial{\mathcal{R}_{\delta_{3}}}$. Then by (\ref{QC}), there exist $T_2$ and $\tilde{\psi}^0\in {\bf C}_{T_2}$ such that $\tilde{\psi}^0(0)=O,\ \tilde{\psi}^0(T_2)=y_0$ and
\begin{equation}\label{Q1}
I^{T_2}_O(\tilde{\psi}^0)<\frac{1}{20}(s_0\wedge s_1).
\end{equation}
Besides, from (iii) of Lemma \ref{PropQ 0808} it follows that there exists $T_1$ such that for any $x\in \overline{\mathcal{R}_{\delta_3}}$, there exist $T^x\le T_1$ and $\tilde{\psi}^x\in {\bf C}_{T^x}$ with $\tilde{\psi}^x(0)=x,\ \tilde{\psi}^x(T^x)=O$ satisfying
\begin{equation}\label{Q2}
I^{T^x}_x(\tilde{\psi}^x)<\frac{1}{20}(s_0\wedge s_1).
\end{equation}
Let $T=T_1+T_2+T_3$. Then for each $x\in\overline{\mathcal{R}_{\delta_2}}$,
if $x\in \overline{\mathcal{R}_{\delta_2}}\backslash\mathcal{R}_{\delta_{3}}$, $\psi^x:=\Psi_{\cdot}(x)|_{[0,T]}$; for $x\in \mathcal{R}_{\delta_{3}}$,
\[
\psi^x:=\left\{
         \begin{array}{ll}
           \tilde{\psi}^x(t), & t\in [0, T^x], \\
           \tilde{\psi}^0(t-T^x), & t\in [T^x, T^x+T_2],\\
           \Psi_{t-T^x-T_2}(y_0), & t\in [T^x+T_2, T].
         \end{array}
       \right.
\]
From (\ref{Q1}), (\ref{Q2}) and (\ref{T3}) it follows  immediately that
\begin{equation}\label{repup}
  I^T_x(\psi^x)\leq \frac{1}{10}(s_0\wedge s_1),\ {\rm and}
\end{equation}
\begin{equation}\label{reploc}
  \psi^x(T)\in(\mathcal{R}_{\delta})^c.
\end{equation}
The remaining proof is entirely the same as that of \cite [Theorem 3.1]{XCJ} if the notation $I^T_x$ here is replaced by $S_T^x$ there. Thus, we conclude that there exists a neighborhood $U_0$ of $\mathcal{R}$, $\kappa>0$ and $\varepsilon_*>0$
such that for any $\vare\in(0,\varepsilon_*)$, we have
\[ \mu^{\vare}(U_0)\leq \exp\{-\kappa/ \vare\}  .   \]
This completes the proof.

\end{proof}
\end{example}

\begin{rmk}\label{rmk4.1}
The Lyapunov function (\ref{LyaF}) can serve us to prove both  stochastically and periodically forced van der Pol equation
\begin{equation*}
X_{\varepsilon}^{(2)} + (X_{\varepsilon}^2-1)X_{\varepsilon}^{(1)} +X_{\varepsilon}= \sin \omega t+\sqrt{\varepsilon}\sigma(X_{\varepsilon},X_{\varepsilon}^{(1)})\dot{B}_t,\ \sigma(X_{\varepsilon},X_{\varepsilon}^{(1)})>0
\end{equation*}
admits a periodic stationary distribution under the same condition on the diffusion coefficient, which is again new as far as we know.
\end{rmk}

\begin{example}
It is important to investigate the system behavior of random perturbation of voltage in electric circuits for circuit performance optimization, fault diagnosis, robustness analysis, etc., see for example \cite{Xu2021}.  The following is the  single diode circuit taken from Moser \cite{Moser}, see Figure \ref{fig4}. Here the rectangle is the symbol for the nonlinear characteristic, given by $i=f(v)$. During the operation of the circuit, the voltage $E$ is usually randomly perturbed by external noise, component variations and temperature changes. Let $L$ and $C$ be positive constants. Then the stochastic diode circuit equation reads as

\begin{figure}[h]
  \centering
  \includegraphics[width=0.5\textwidth]{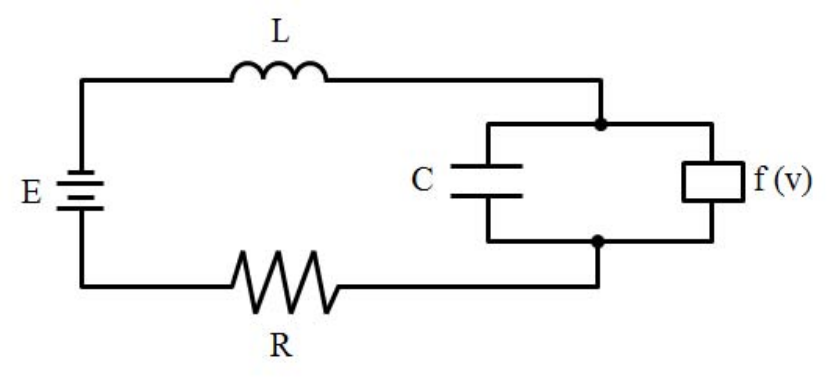}\\
  \caption{Single Diode Circuit.}\label{fig4}
\end{figure}

\begin{equation}\label{1DELEs}
\begin{cases}
    Ldi_t= \left(E-Ri_t-v_t\right)dt+ \sqrt{\varepsilon} \sigma(v_t, i_t)dB_t,\ i, v\in \mathbb{R}, \\
     Cdv_t=(i_t-f(v_t))dt,\ \sigma(v_t, i_t)>0
\end{cases}
\end{equation}
where the noise term in the first equation represents the random perturbation to the voltage $E$, $i_t$ and $v_t$ are the current and the voltage in the circuit at the time $t$, respectively, $f$ is continuously differentiable.
\end{example}

Using the homeomorphic transformation $H: \mathbb{R}^2\rightarrow \mathbb{R}^2$ defined by
\begin{equation}\label{HT1}
H(v,i):\begin{cases}
    v=v, \\
    w=C^{-1}(i-f(v)),
\end{cases}
\end{equation}
the system (\ref{1DELEs}) is transformed into
  the following stochastic second-order equation
  \begin{equation}\label{1DVanDuffing}
\begin{cases}
    dv=wdt, \\
    dw=-C^{-1}[\left(f'(v)+\frac{CR}{L}\right)w +\frac{R}{L}f(v)+\frac{v-E}{L}]dt+\frac{\sqrt{\varepsilon}}{LC}\sigma\big(v,Cw+f(v)\big)d{B}_t.
\end{cases}
\end{equation}
 For $x,y,p,q\in \mathbb{R}^2$, let
 $$\mathbf{A^{{\rm 20}}}\mathbf{C}_T:=\{(v(\cdot),i(\cdot))\in \mathbf{A}\mathbf{C}_T |\ \ C\dot{v}(\cdot) = i(\cdot)-f(v(\cdot))\},$$
  $$\mathbf{A^{{\rm 20}}}\mathbf{C}_T^{x}:=\{(v(\cdot),i(\cdot))\in \mathbf{A}\mathbf{C}_T |\ (v(0),i(0))=x,\ C\dot{v}(\cdot) = i(\cdot)-f(v(\cdot))\},$$
  $$\mathbf{A^{{\rm 20}}}\mathbf{C}_T^{x,y}:=\{(v(\cdot),i(\cdot))\in \mathbf{A}\mathbf{C}_T |\ (v(0),i(0))=x,\ (v(T),i(T))=y,\ C\dot{v}(\cdot) = i(\cdot)-f(v(\cdot))\},$$
  $$\mathbf{A^{{\rm 22}}}\mathbf{C}_T:=\{(v(\cdot),w(\cdot))\in \mathbf{A}\mathbf{C}_T |\ \ w(\cdot)=\dot{v}(\cdot) \},$$
  $$\mathbf{A^{{\rm 22}}}\mathbf{C}_T^{p}:=\{(v(\cdot),w(\cdot))\in \mathbf{A}\mathbf{C}_T |\ (v(0),w(0))=p,\ w(\cdot)=\dot{v}(\cdot)\},$$
  $$\mathbf{A^{{\rm 22}}}\mathbf{C}_T^{p,q}:=\{(v(\cdot),w(\cdot))\in \mathbf{A}\mathbf{C}_T |\ (v(0),w(0))=p,\ (v(T),w(T))=q,\ w(\cdot)=\dot{v}(\cdot)\}.$$
  The homeomorphic transformation (\ref{HT1}) induces the mappings $\mathcal{H}_T:\mathbf{A^{{\rm 20}}}\mathbf{C}_T\rightarrow \mathbf{A^{{\rm 22}}}\mathbf{C}_T$, $\mathcal{H}_T:\mathbf{A^{{\rm 20}}}\mathbf{C}_T^x\rightarrow \mathbf{A^{{\rm 22}}}\mathbf{C}_T^{H(x)}$ and $\mathcal{H}_T:\mathbf{A^{{\rm 20}}}\mathbf{C}_T^{x,y}\rightarrow \mathbf{A^{{\rm 22}}}\mathbf{C}_T^{H(x),H(y)}$.  It is easy to check that $\mathcal{H}_T:\mathbf{A^{{\rm 20}}}\mathbf{C}_T\rightarrow \mathbf{A^{{\rm 22}}}\mathbf{C}_T$ is a one-to-one mapping. Therefore, both $\mathcal{H}_T:\mathbf{A^{{\rm 20}}}\mathbf{C}_T^x\rightarrow \mathbf{A^{{\rm 22}}}\mathbf{C}_T^{H(x)}$ and $\mathcal{H}_T:\mathbf{A^{{\rm 20}}}\mathbf{C}_T^{x,y}\rightarrow \mathbf{A^{{\rm 22}}}\mathbf{C}_T^{H(x),H(y)}$ are also one-to-one. Now let us define the norm on the space ${\bf AC}_T$ by
$$\|f\|:=\|f\|_{C^0}+\|f'\|_{L^1}\ {\rm for}\ f\in {\bf AC}_T$$
where $\|f\|_{C^0}:={\rm max}_{0\le t\le T}|f(t)|$ and $\|f'\|_{L^1} :=\int_0^T|f'(t)|dt$. Thus we have
\begin{lem}\label{BS}
The space $({\bf AC}_T, \|\cdot\|)$ is a Banach space. The metric spaces $\mathbf{A^{{\rm 20}}}\mathbf{C}_T^x$, $\mathbf{A^{{\rm 20}}}\mathbf{C}_T^{x,y}$, $\mathbf{A^{{\rm 22}}}\mathbf{C}_T^p$ and $\mathbf{A^{{\rm 22}}}\mathbf{C}_T^{p,q}$ are all complete under the metric induced by $\|\cdot\|$.
\end{lem}
\begin{proof}
Let $\{f_n\}\subset {\bf AC}_T$ be a Cauchy sequence. Then both $\{f_n\}\subset {\bf C}_T$ and $\{f_n'\}\subset L^1([0,T])$  are Cauchy sequences. Therefore, from the completeness of the spaces $({\bf C}_T,\|\cdot\|_{C^0} )$ and $(L^1([0,T]),\|\cdot\|_{L^1} )$ it follows that there exist $f\in {\bf C}_T$ and $g\in L^1([0,T])$ such that $\|f_n-f\|_{C^0}\rightarrow 0$ and $\|f_n'-g\|_{L^1}\rightarrow 0$ as $n\rightarrow \infty$. Applying the Newton-Leibnitz formula, we get that
$$f_n(t)=f_n(0)+\int_0^tf_n'(s)ds,\ {\rm for\ all}\ t\in[0,T].$$
Letting $n\rightarrow \infty$ in the above equality, we have that
$$f(t)=f(0)+\int_0^tg(s)ds,\ {\rm for\ all}\ t\in[0,T],$$
which implies that $f\in {\bf AC}_T$. This proves that the space $({\bf AC}_T, \|\cdot\|)$ is a Banach space.

Because the spaces $\mathbf{A^{{\rm 20}}}\mathbf{C}_T^x$, $\mathbf{A^{{\rm 20}}}\mathbf{C}_T^{x,y}$, $\mathbf{A^{{\rm 22}}}\mathbf{C}_T^p$ and $\mathbf{A^{{\rm 22}}}\mathbf{C}_T^{p,q}$ are closed in ${\bf AC}_T$, they are all complete.
\end{proof}

Using the definition  given in (\ref{rate-0}),
the rate function of system (\ref{1DELEs}) is given by
\begin{align}
  \mathcal{I}^{T}_{(v,i)}(v(\cdot),i(\cdot))
& =
\left\{
\begin{array}{ll}
\frac{1}{2}\int\limits_{0}^{T}\Big|\frac{L\dot{i}(s)-(E-v(s)-Ri(s))}{L^2C^2\sigma( v(s),i(s))}\Big|^2ds,   & \hbox{if} \begin{array}{c}
(v(\cdot),i(\cdot))\in \mathbf{A}^{\rm 20}\mathbf{C}_T^{(v,i)},	\\
	\end{array} \\
+\infty, & \hbox{otherwise.}
\end{array}
\right. \label {actfalini0}\\
&=\left\{
\begin{array}{ll}
\frac{1}{2}\int\limits_{0}^{T}\Big|\frac{LC\dot{w}(s) +L f'(v(s))w(s)-(E-v(s)-RCw(s)-Rf(v(s)))}{L^2C^2\sigma(v(s),Cw(s)+f(v(s)))}\Big|^2ds, & \hbox{if}
\ (v(\cdot),w(\cdot))\in \mathbf{A}^{\rm 22}\mathbf{C}_T^{(v,w)},	 \\
+\infty, & \hbox{otherwise.}
\end{array}
\right. \label{actfalini}
\end{align}
Let $\mathcal{J}^{T}_{(v,w)}(v(\cdot),w(\cdot))$ denote  the rate function of the second-order equation (\ref{1DVanDuffing}). Denote by $\mathbb{V}^{20}$ and $\mathbb{V}^{22}$ the quasipotentials of (\ref{1DELEs})  and (\ref{1DVanDuffing}), respectively. Then we summarize the relations between the solutions, the rate functions, the quasipotentials  and invariant measures between (\ref{1DELEs})  and (\ref{1DVanDuffing}) as follows.

\begin{lem}\label{Relation2022}
{\rm (i)} Let $(v(t,\omega,v_0,i_0), i(t,\omega,v_0,i_0))$ and $(v(t,\omega,v_0,w_0), w(t,\omega,v_0,w_0))$ be the solutions of {\rm(\ref{1DELEs})} and {\rm(\ref{1DVanDuffing})} passing through $(v_0,i_0)$ and $(v_0,w_0)$, respectively. Then
\begin{equation}\label{SoluRe}
\begin{cases}
   (v(t,\omega,H(v_0,i_0)), w(t,\omega,H(v_0,i_0)))= H(v(t,\omega,v_0,i_0), i(t,\omega,v_0,i_0)), \\
   {\rm for}\ {\rm all}\ (t,\omega,(v_0,i_0))\in \mathbb{R}_+\times \Omega\times \mathbb{R}^2.
\end{cases}
\end{equation}

{\rm (ii)}  The mapping $\mathcal{H}_T:\mathbf{A^{{\rm 20}}}\mathbf{C}_T\rightarrow \mathbf{A^{{\rm 22}}}\mathbf{C}_T$ is a homeomorphism. Besides, $\mathcal{H}_T:\mathbf{A^{{\rm 20}}}\mathbf{C}_T^x\rightarrow \mathbf{A^{{\rm 22}}}\mathbf{C}_T^{H(x)}$ and $\mathcal{H}_T:\mathbf{A^{{\rm 20}}}\mathbf{C}_T^{x,y}\rightarrow \mathbf{A^{{\rm 22}}}\mathbf{C}_T^{H(x),H(y)}$ are homeomorphisms for all $x,y\in \mathbb{R}^2$ and $T>0$.

{\rm (iii)} $\mathcal{J}^{T}_{H(v,i)}(\mathcal{H}_T(v(\cdot),i(\cdot)))=\mathcal{I}^{T}_{(v,i)}(v(\cdot),i(\cdot))$ for all $(v,i)\in \mathbb{R}^2$ and  $(v(\cdot),i(\cdot))\in \mathbf{A}^{\rm 20}\mathbf{C}_T^{(v,i)}.$

{\rm (iv)} $\mathbb{V}^{22}(H(x),H(y))= \mathbb{V}^{20}(x,y)$ for all $x,y\in\mathbb{R}^2$.

{\rm (v)} If $\mu^{\varepsilon}$ is an invariant measure of {\rm(\ref{1DELEs})} if and only if  $\mu^{\varepsilon}\circ H^{-1}$ is the  invariant measure of {\rm(\ref{1DVanDuffing})}.

\end{lem}
\begin{proof}
(i) follows immediately from the transformation $H$ and the It\^{o} formula. (iii) follows from (\ref{actfalini0}) and (\ref{actfalini}).

In order to prove (ii), let $\{(v_n(\cdot),i_n(\cdot))\}_{n=0}^{\infty}\subset \mathbf{A^{{\rm 20}}}\mathbf{C}_T$ such that $\|v_n(\cdot)-v_0(\cdot)\|\rightarrow 0$ and $\|i_n(\cdot)-i_0(\cdot)\|\rightarrow 0$ as $n\rightarrow \infty$, then we need to prove that $\|w_n(\cdot)-w_0(\cdot)\|\rightarrow 0$ as $\rightarrow \infty$ with $w_n(\cdot)=C^{-1}(i_n(\cdot)-f(v_n(\cdot)))$ and $w_0(\cdot)=C^{-1}(i_0(\cdot)-f(v_0(\cdot)))$. In fact, it follows from $\lim_{n\rightarrow \infty}\|v_n(\cdot)-v_0(\cdot)\|= 0$ that there exist constants $K>0, L>0$ such that  $|v_n(t)|\le K$ for all $t\in [0,T]$ and $n=0, 1, 2, \cdots$ and $|f'(s)|\le L$ for all $|s|\le K$.
\begin{align*}
\|w_n(\cdot)-w_0(\cdot)\|&\le C^{-1}\|i_n(\cdot)-i_0(\cdot)\|+ C^{-1}\|f(v_n(\cdot))-f(v_0(\cdot))\|_{C^0} \\
& +C^{-1}\|f'(v_n(\cdot))v_n'(\cdot)-f'(v_0(\cdot))v_0'(\cdot)\|_{L^1}\\
&\le C^{-1}\|i_n(\cdot)-i_0(\cdot)\|+ C^{-1}L\|v_n(\cdot)-v_0(\cdot)\|_{C^0} \\
& +C^{-1}\|f'(v_n(\cdot))(v_n'(\cdot)-v_0'(\cdot))+(f'(v_n(\cdot))-f'(v_0(\cdot)))v_0'(\cdot)\|_{L^1}\\
&\le C^{-1}\|i_n(\cdot)-i_0(\cdot)\|+ C^{-1}L\|v_n(\cdot)-v_0(\cdot)\| \\
&+\|f'(v_n(\cdot))-f'(v_0(\cdot))\|_{C^0}\|v_0'(\cdot)\|_{L^1}\rightarrow 0\ {\rm as}\ n\rightarrow \infty.\\
\end{align*}
This proves that $\lim_{n\rightarrow \infty}\mathcal{H}_T(v_n(\cdot),i_n(\cdot))=\mathcal{H}_T(v_0(\cdot),i_0(\cdot))$. Similarly, we can show that $$\lim_{n\rightarrow \infty}\mathcal{H}_T^{-1}(v_n(\cdot),w_n(\cdot))=\mathcal{H}_T^{-1}(v_0(\cdot),w_0(\cdot)).$$

Now we prove (iv). For any fixed points $x,y\in \mathbb{R}^2$, by the definition of quasipotential,
\begin{align*}
\mathbb{V}^{22}(H(x),H(y))&= {\rm inf}\{\mathcal{J}^T_{H(x)}(v(\cdot), w(\cdot)) |\ (v(\cdot), w(\cdot))\in \mathbf{A^{{\rm 22}}}\mathbf{C}_T^{H(x),H(y)},\ T>0\} \\
& = {\rm inf}\{\mathcal{J}^T_{H(x)}(\mathcal{H}_T(v(\cdot), i(\cdot))) |\ (v(\cdot), i(\cdot))\in \mathbf{A^{{\rm 20}}}\mathbf{C}_T^{x,y},\ T>0\} \\
&  = {\rm inf}\{\mathcal{I}^T_{x}((v(\cdot), i(\cdot))) |\ (v(\cdot), i(\cdot))\in \mathbf{A^{{\rm 20}}}\mathbf{C}_T^{x,y},\ T>0\} \\
& =\mathbb{V}^{20}(x,y).
\end{align*}
The second and third equalities have used (ii) and (iii) respectively. (iv) is proved.

(v) follows from (\ref{SoluRe}) and \cite[Theorem 2.6]{DJNZ}.
\end{proof}
\vskip 0.2cm

Take $\ f(v)=(v-E)^3-(v-E)$ and replace $v - E$ by $v$. Then the unperturbed equation of (\ref{1DVanDuffing}) is the following Li\'{e}nard equation
\begin{equation}\label{2DVanDuffing}
\begin{cases}
    \dot{v}=w, \\
    \dot{w}=-C^{-1}[\left(3v^2-1+\frac{CR}{L}\right)w +\frac{R}{L}(v^3-v)+\frac{v}{L}].
\end{cases}
\end{equation}

Let $V(v,i)= Li^2 +Cv^2$ be a Lyapunov function for (\ref{1DELEs}) and  assume that $\si$ is a locally  Lipschitz continuous function satisfying $0<\si^2(v,i) \leq c_1(v^4 + i^2+1)$, $c_1 >0$. Then,  the left-hand side of the condition \eqref{22} with $\theta=\frac{L}{4c_1}$ and $\eta=\frac{16c_1}{L}$ is
\begin{align}\label{4.2-J}
J&:=-2Ri^2+2Ei-2v(v-E)^3+2v(v-E)+\frac{\theta}{L} \si^2(v,i) + \frac{4i^2 \si^2(v,i)}{\et V(v,i)} \notag\\
&\leq -2Ri^2+2Ei-2v(v-E)^3+2v(v-E) +\frac{1}{2}(i^2 +v^4+1),
\end{align}
for all $(v,i) \in \RR^2$.
Since  $\lim_{|v| \rightarrow \infty} \frac{-2v(v-E)^3+2v(v-E)}{v^4} = -2$, \eqref{4.2-J} implies that $J\le C^*(1+V(v,i)$ for some large constant $C^*$, that is, \eqref{22} holds. On the other hand, the conditions \eqref{As b V} and \eqref{23} are obviously satisfied. It follows from Theorem \ref{thmULDP} that (\ref{1DELEs}) admits ULDP.

Let $\varepsilon_0 =\frac{L}{4c_1}(R \wedge 1) $. It can also be seen from \eqref{4.2-J} that
 there exist positive constants $k,M$ such that
\begin{align*}
\mathcal{L}^{\varepsilon}V(v,i)&:=-2Ri^2+2Ei-2v(v-E)^3+2v(v-E)+\frac{\varepsilon}{L} \si^2(v,i)\\
&\le -k(i^2+v^4) \ {\rm for}\ 0<\varepsilon \le \varepsilon_0,   i^2+v^2\ge M,
\end{align*}
where $\mathcal{L}^{\varepsilon}$ is the generator of the system (\ref{1DELEs}).
Denote by $\mathcal{J^{\varepsilon}}$ the set of all
 invariant measures of (\ref{1DELEs}) for a given $0<\varepsilon \le \varepsilon_0$. By \cite [Theorem 3.1]{Chen2020}, $\mathcal{J^{\varepsilon}}$ is nonempty and $\mathcal{J}:=\bigcup_{0<\varepsilon \le \varepsilon_0}\mathcal{J^{\varepsilon}}$ is tight. Furthermore, let $\mu^{\varepsilon_i}\in \mathcal{J}^{\varepsilon_i}$ such that $\mu^{\varepsilon_i}$  converges weakly to $\mu$ as $\varepsilon_i\rightarrow 0$. Then $\mu$ is an invariant measure of the  unperturbed equations of (\ref{1DELEs}):
 \begin{equation}\label{Unp1DVanDuffing}
\begin{cases}
   L\frac{di}{dt}= E-Ri-v, \\
     C\frac{dv}{dt}=i-(v-E)^3+(v-E)
\end{cases}
\end{equation}
and its support is contained in the Birkhoff center of (\ref{Unp1DVanDuffing}) (see \cite{Chen2020}).



The details are contained in the following result.
\begin{prop}\label{SDC}
Suppose that $\si(v,i)$ is a locally Lipschitz, positive  continuous function with $\si^2(v,i) \leq c_1(v^4 + i^2 +1)$, $c_1 >0$. Then there exists $\varepsilon_0>0$ such that for any given $\varepsilon\in (0,\varepsilon_0]$, the set $\mathcal{J^{\varepsilon}}$ of  invariant measures of the system {\rm(\ref{1DELEs})} is nonempty and $\mathcal{J}:=\bigcup_{0<\varepsilon \le \varepsilon_0}\mathcal{J^{\varepsilon}}$ is tight. Furthermore, let $\mu^{\varepsilon_i}\in \mathcal{J}^{\varepsilon_i}$ such that $\mu^{\varepsilon_i}$  converges weakly to $\mu$ as $\varepsilon_i\rightarrow 0$. Then $\mu$ is an invariant measure of  {\rm(\ref{Unp1DVanDuffing})}
and its support is contained in the Birkhoff center of the system {\rm(\ref{Unp1DVanDuffing})}. More precisely,\\
{\rm(i)} if $\frac{CR}{L}\ge 1$ and $0<R\le 1$, then when $\varepsilon\rightarrow 0$ the invariant measure $\mu^{\varepsilon}$ of {\rm (\ref{1DELEs})}  converges weakly to  $\delta_{(E,0)}(\cdot)$;\\
{\rm(ii)} if $\frac{CR}{L}\ge 1$ and $R>1$, then when $\varepsilon\rightarrow 0$ the invariant measure $\mu^{\varepsilon}$ of {\rm (\ref{1DELEs})}  converges weakly to  $\lambda\delta_{(E+ \sqrt{1-R^{-1}} ,- R^{-1}\sqrt{1-R^{-1}})}(\cdot)+(1-\lambda)\delta_{(E- \sqrt{1-R^{-1}} , R^{-1}\sqrt{1-R^{-1}})}(\cdot)$ with $0\le \lambda \le 1$;\\
{\rm(iii)} if $\frac{CR}{L}< 1$ and $0<R\le 1$, then when $\varepsilon\rightarrow 0$ the invariant measure $\mu^{\varepsilon}$ of {\rm (\ref{1DELEs})}  converges weakly to $\mu_{\Gamma}$, where $\mu_{\Gamma}$ is the Haar measure supported on the unique limit cycle $\Gamma$, see Figure {\rm \ref{fig5}}.\\

\end{prop}
\begin{proof}
As discussed above, (\ref{1DELEs}) admits ULDP under the condition on $\sigma(v,i)$. It is easy to see from (\ref{HT1}), (\ref{Unp1DVanDuffing}) and (\ref{2DVanDuffing}) that $E^*(v^*,i^*)$ is an equilibrium of (\ref{Unp1DVanDuffing}) if and only if $F^*(v^*,0)$  is an equilibrium of (\ref{2DVanDuffing}).
By (iv) of Lemma \ref{Relation2022},  the continuity of $H$ and (ii) of Lemma \ref {PropQ 0808},
$$\lim_{x\rightarrow E^*} \mathbb{V}^{20}(x,E^*)= \lim_{x\rightarrow E^*} \mathbb{V}^{22}(H(x),F^*)=\lim_{p\rightarrow F^*} \mathbb{V}^{22}(p,F^*)=\mathbb{V}^{22}(F^*,F^*)=0.$$
Similarly, $\lim_{y\rightarrow E^*} \mathbb{V}^{20}(E^*,y)=0$. Therefore, for any $\eta>0$, there exists $\delta_1>0$ such that $\mathbb{V}^{22}(p,F^*)<\eta$ for all $p\in \overline{B}_{\delta_1}(F^*)$. By (iii) of Lemma \ref {PropQ 0808}, there exists $T^*>0$ such that for any $p\in \overline{B}_{\delta_1}(F^*)$, there exist $T^p\le T^*$ and $\psi^p\in {\bf C}_{T^p}$ with $\psi^p(0)=p$ and $\psi^p(T^p)=F^*$ satisfying $\mathcal{J}^{T^p}_p(\psi^p)<\eta$.  Since $H$ is a homeomorphism on the plan $\mathbb{R}^2$, $H^{-1}(B_{\delta_1}(F^*))$ is a neighborhood of $E^*$. Choose $\delta>0$ sufficiently small so that $\overline{B}_{\delta}(E^*)\subset H^{-1}(B_{\delta_1}(F^*))$. For any $x\in \overline{B}_{\delta}(E^*)$, $p:=H(x)\in B_{\delta_1}(F^*)$. Thus, let $T^x=T^p$ and $\phi^x=\mathcal{H}_{T^x}^{-1}(\psi^{H(x)})$. Then by (iii) of Lemma \ref{Relation2022}, $\mathcal{I}_x^{T^x}(\phi^x)= \mathcal{J}^{T^p}_p(\psi^p)<\eta$ and $T^x\le T^*$ for any $x\in \overline{B}_{\delta}(E^*)$.

Suppose that $\frac{CR}{L}\ge 1$.  If $0<R\le 1$, then $(E,0)$ is the unique equilibrium of (\ref{Unp1DVanDuffing}). Assume that $R>1$. Then the equilibria of (\ref{Unp1DVanDuffing}) are $\mathcal{S}(E,0)$, $\mathcal{A}(E+ \sqrt{1-R^{-1}},- R^{-1}\sqrt{1-R^{-1}})$ and $\mathcal{B}(E- \sqrt{1-R^{-1}}), R^{-1}\sqrt{1-R^{-1}})$, $\mathcal{S}$ is a saddle, $\mathcal{A}$ and $\mathcal{B}$  are asymptotically stable. This shows that the equilibria of {\rm(\ref{Unp1DVanDuffing})} is finite in the both cases. It follows from \cite[Theorem 3.3]{Moser} that all solutions of {\rm(\ref{Unp1DVanDuffing})} are convergent to equilibria. Therefore the Birkhoff center of (\ref{Unp1DVanDuffing}) consists of equilibria. The global portraits of (\ref{Unp1DVanDuffing}) in this case are drawn in Figure \ref{portraits28}.
\begin{figure}[h]
  \centering
  \includegraphics[width=0.5\textwidth]{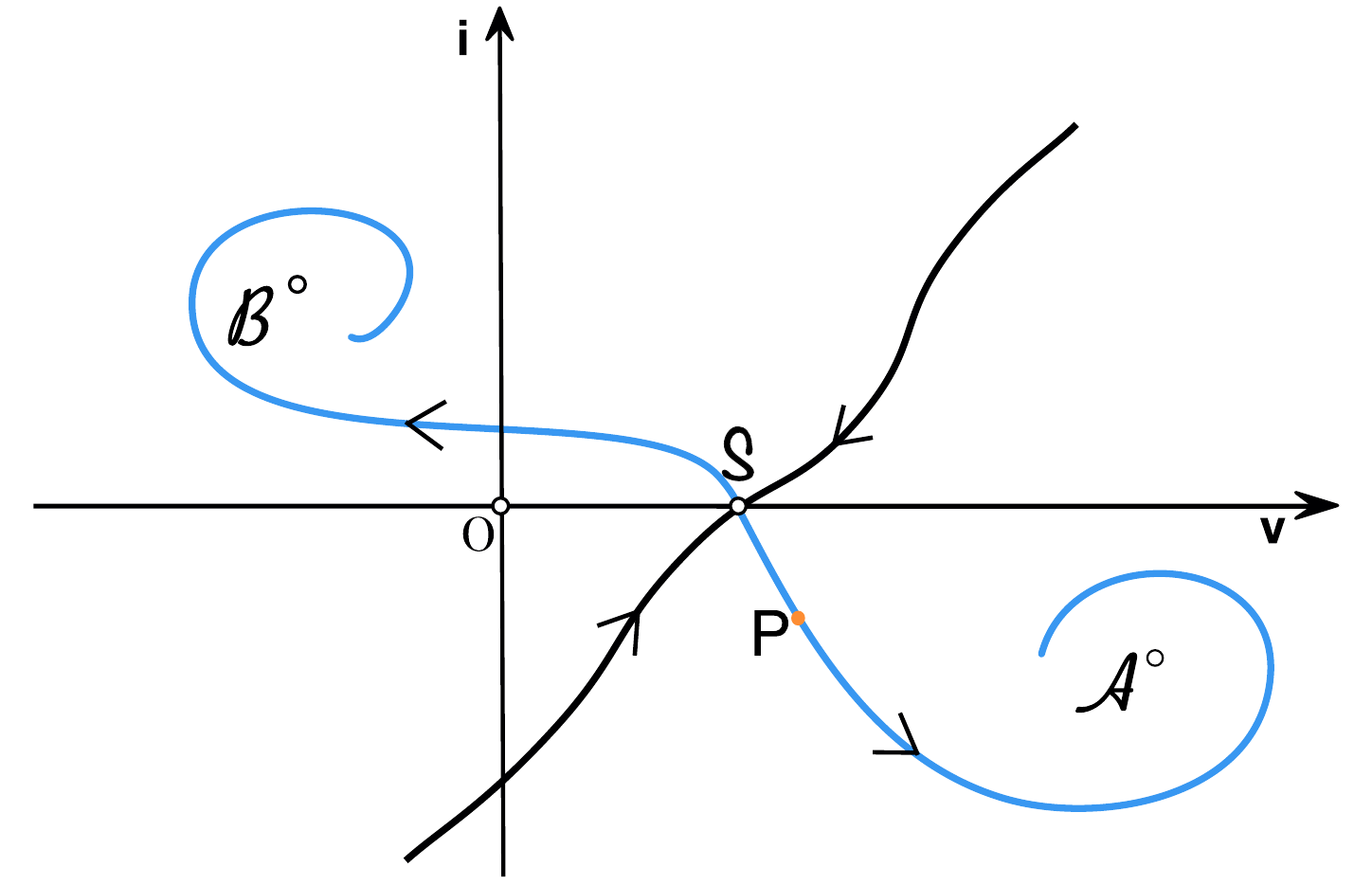}\\
  \caption{The phase portraits of system (\ref{Unp1DVanDuffing}).}\label{portraits28}
\end{figure}

(i) In this case, $(E,0)$ is the Birkhoff center of
 {\rm(\ref{Unp1DVanDuffing})}. Then the conclusion follows immediately.

(iii) Suppose that $\frac{CR}{L}< 1$ and $0<R\le 1$. Then $(E,0)$ is the unique equilibrium of {\rm(\ref{Unp1DVanDuffing})} which is a repeller. Since {\rm(\ref{Unp1DVanDuffing})} is equivalent to the Li\'{e}nard system (\ref{2DVanDuffing}) which has  the unique limit cycle $\Gamma$ illustrated in Figure \ref{fig5}, see \cite{ZhangZhifen}. Therefore, the system {\rm(\ref{Unp1DVanDuffing})} admits a unique limit cycle, still denoted by $\Gamma$. The Birkhoff center of {\rm(\ref{Unp1DVanDuffing})} is composed of the repeller $(E,0)$ and the unique stable limit cycle  $\Gamma$. Because {\rm(\ref{1DELEs})} admits ULDP, based on the results obtained in the first paragraph and preceded in the same manner as in the proof of Proposition \ref {SvdP}, we can conclude that there is no concentration for limiting measures on the repeller $(E,0)$. This proves that when $\varepsilon\rightarrow 0$ the invariant measure $\mu^{\varepsilon}$ of {\rm (\ref{1DELEs})} converges weakly to  the Haar measure $\mu_{\Gamma}$.

(ii) Suppose $\frac{CR}{L}\ge 1$ and $R>1$.  Then {\rm (\ref{1DELEs})} admits ULDP. We shall show that the limiting measures will not concentrate on the saddle $\mathcal{S}(E,0)$.

Let $\delta'< \frac{1}{2} \sqrt{(1-R^{-1})(1+R^{-2})}$ such that
$\mathcal{A}_{\delta'}$ is a fundamental neighborhood of $\mathcal{A}$. Note that
\cite [Lemma 4.2]{XCJ} still holds as soon as ULDP is satisfied. Thus by  Theorem \ref{thmULDP} and \cite [Lemma 4.2]{XCJ}, there exist $0<\delta_3<\delta_2<\delta^*_1<\delta_1
<\delta'$ and $s_0>0$ such that $\mathbb{V}\big(\partial \mathcal{A}_{\delta_2},
\partial \mathcal{A}_{\delta^*_1}\big) \geq s_0$.
Since $\mathcal{A}$ is an attractor, we obtain
$$T_0:=\inf\big\{u\geq 0:\Psi_{t}\big(\overline{\mathcal{A}_{\delta_1}}\big)\subset\mathcal{A}_{\delta_3}, t\geq u\big\}<+\infty.$$
The set $$F_0=\{\varphi\in {\bf C}_{T_0}:\varphi(0)\in \overline{\mathcal{A}_{\delta_1}},\varphi(T_0)\in (\mathcal{A}_{\delta_2})^c\}$$
 is a closed subset of ${\bf C}_{T_0}$. $F_0^0=\overline{\mathcal{A}_{\delta_1}}$ is compact and $F_0$ does not contain any solution of system (\ref{Unp1DVanDuffing}).
Thus, by Theorem \ref{thmULDP} and \cite [Proposition, 2.4]{XCJ}
we have
\[ s_1:= \inf\{\mathcal{I}^{T_0}_{x}(\varphi): \varphi\in F_0,\ x=\varphi(0)\in \overline{\mathcal{A}_{\delta_1}}\ \}>0.\]

  We claim that there exist $T_1>0$ and $ \tilde{\psi}\in {\bf C}_{T_1}$ with $
\tilde{\psi}(0)=\mathcal{S}$ and $\tilde{\psi}(T_1)= \mathcal{A}$ such that
\begin{equation}\label{R1}
\mathcal{I}_{\mathcal{S}}^{T_1}(\tilde{\psi})< \frac{1}{5}(s_0\wedge s_1).
\end{equation}

Since $\mathcal{S}$ is an equilibrium of {\rm(\ref{1DELEs})}, by the definition of quasipotential, $\mathbb{V}^{20}(\mathcal{S},\mathcal{S})=0$. The result in the first paragraph with $E^*=\mathcal{S}$ and $\eta=\frac{1}{10}(s_0\wedge s_1)$ shows that there is $0<\delta<\delta_3$ such that
\begin{equation}\label{R2}
\mathbb{V}^{20}(x,\mathcal{S}),\ \mathbb{V}^{20}(\mathcal{S},y)< \frac{1}{10}(s_0\wedge s_1)\ {\rm for\ all}\ x, y\in \overline{B_{\delta}(\mathcal{S})}.
\end{equation}
Let
$$W^u(\mathcal{S}):=\{x\in \mathbb{R}^2: \lim_{t\rightarrow -\infty}\Psi_t(x)=\mathcal{S}\}$$
be the unstable manifold of $\mathcal{S}$ and
$$W^s(\mathcal{S}):=\{x\in \mathbb{R}^2: \lim_{t\rightarrow \infty}\Psi_t(x)=\mathcal{S}\}$$
denote  the stable manifold of $\mathcal{S}$. Then $W^s(\mathcal{S})$ separates $\mathbb{R}^2$ into two parts which are the attracting domains of $\mathcal{A}$ and $\mathcal{B}$ respectively, see Figure \ref{portraits28}.

Choose $P\in B_{\delta}(\mathcal{S})\cap W^u(\mathcal{S})$ with $P \neq \mathcal{S}$. Then  $\mathbb{V}^{20}(\mathcal{S},P)< \frac{1}{10}(s_0\wedge s_1)$ by (\ref{R2}). By the definition of quasipotential, there exists $T_1^0$ and $\tilde{\psi}_1\in {\bf C}_{T_1^0}$ with $\tilde{\psi}_1(0)=\mathcal{S}$ and $\tilde{\psi}_1(T_1^0)=P$ such that
\begin{equation}\label{R3}
\mathcal{I}_{\mathcal{S}}^{T_1^0}(\tilde{\psi_1})< \frac{1}{10}(s_0\wedge s_1).
\end{equation}

Similarly, since $\mathcal{A}$ is an equilibrium of {\rm(\ref{1DELEs})}  and $\mathbb{V}^{20}(\mathcal{A},\mathcal{A})=0$. The first paragraph has shown that $\mathbb{V}^{20}(x,\mathcal{A})$ is continuous at $x=\mathcal{A}$. Thus, $\lim_{t\rightarrow \infty}\mathbb{V}^{20}(\Psi_t(P),\mathcal{A})=0$. Therefore, there exists $T_2^0$ such that $\mathbb{V}^{20}(\Psi_{T_2^0}(P),\mathcal{A})< \frac{1}{10}(s_0\wedge s_1)$. This implies that there exists $T_3^0$ and $\tilde{\psi}_3\in {\bf C}_{T_3^0}$ with $\tilde{\psi}_3(0)=\Psi_{T_2^0}(P)$ and $\tilde{\psi}_3(T_3^0)=\mathcal{A}$ such that
\begin{equation}\label{R4}
\mathcal{I}_{\Psi_{T_2^0}(P)}^{T_3^0}(\tilde{\psi_3})< \frac{1}{10}(s_0\wedge s_1).
\end{equation}
Let $T_1=T_1^0+T_2^0+T_3^0$ and define $\tilde{\psi}\in {\bf C}_{T_1}$ by
\[
\tilde{\psi}(t)=\left\{
         \begin{array}{ll}
           \tilde{\psi}_1(t), & t\in [0, T_1^0], \\
           \Psi_{t-T_1^0}(P), & t\in [T_1^0, T_1^0+T_2^0],\\
           \tilde{\psi}_3(t-T_1^0-T_2^0), & t\in [T_1^0+T_2^0, T_1].
         \end{array}
       \right.
\]
Obviously, $\tilde{\psi}(0)=\mathcal{S}$ and $\tilde{\psi}(T_1)=\mathcal{A}$.  (\ref{R1}) follows immediately from (\ref{R3}) and (\ref{R4}), hence the claim is true.

Choose $E^*= \mathcal{S}$ and $\eta=\frac{1}{10}(s_0\wedge s_1)$ in the first paragraph. Then  we know that there exists $T^*>0$ such that for any $x\in \overline{B}_{\delta}(\mathcal{S})$, there exist $T^x\le T^*$ and $\tilde{\psi}^x\in {\bf C}_{T^x}$ with $\tilde{\psi}^x(0)=x,\ \tilde{\psi}^x(T^x)=\mathcal{S}$ such that
\begin{equation}\label{R5}
I^{T^x}_x(\tilde{\psi}^x)< \frac{1}{10}(s_0\wedge s_1).
\end{equation}
  Let $T=T^*+T_1$. Define $\psi^x\in {\bf C}_T$ by
\[
\psi^x:=\left\{
         \begin{array}{ll}
           \tilde{\psi}^x(t), & t\in [0, T^x], \\
           \tilde{\psi}(t-T^x), & t\in [T^x, T^x+T_1],\\
           \mathcal{A}, & t\in [T^x+T_1, T].
         \end{array}
       \right.
\]
Thus, by (\ref{R1}) and (\ref{R5}), for every $x\in\bar{B}_{\delta}(\mathcal{S})$,
\begin{equation}\label{attractup}
 \mathcal{I}_{x}^{T}(\psi^x)\leq \frac{3}{10}(s_0\wedge s_1)
\end{equation}
and
\begin{equation}\label{attraloc}
\psi^x(T)= \mathcal{A}.
\end{equation}
The remaining proof is entirely the same as that of \cite [Theorem 3.2]{XCJ} if the notation $\mathcal{I}_{x}^{T}$ here is replaced by $S_T^x$ there. Thus, we conclude that there exists  $\varepsilon_*>0$
such that for any $\vare\in(0,\varepsilon_*)$,
\[ \mu^{\vare}(B_{\delta}(\mathcal{S}))\leq \exp\{-\frac{2}{5}(s_0\wedge s_1)/ \vare\}  .   \]
This completes the proof.

\end{proof}

\begin{rmk}\label{rmk4.2}
By Proposition \ref{SDC}, it remains that the case

$${\rm (iv)}\ 1< R <\frac{L}{C}$$
 not have been considered yet. In this case, $\mathcal{S}$ is still a saddle and  we cannot give the global dynamical behavior of system (\ref{Unp1DVanDuffing}). Let $M$ be the linearized matrix of (\ref{Unp1DVanDuffing}) at both equilibria $\mathcal{A}$  and $\mathcal{B}$. Then the domain of (iv) in plane $(\frac{L}{C},R)$ is divided into two parts by ${\rm Trace}(M)=0$ , (iv)$_+$: ${\rm Trace}(M)>0$ and (iv)$_-$: ${\rm Trace}(M)<0$. By calculation,
 $$ {\rm det} M=2(LC)^{-1}(R-1)>0;\ {\rm Trace}(M)=-(LR)^{-1}\left(R^2+2\frac{L}{C}R-3\frac{L}{C}\right)=: =-(LR)^{-1}g\left(\frac{L}{C},R\right).$$
And $g=0$ is equivalent to $R=h\left(\frac{L}{C}\right):= -\frac{L}{C}+\sqrt{\left(\frac{L}{C}\right)^2+3\frac{L}{C}},$ which is the Hopf bifurcation curve and strictly increasing to $\frac{3}{2}$ as $\frac{L}{C}\rightarrow \infty$. Both $\mathcal{A}$  and $\mathcal{B}$ are locally asymptotically stable between $R=\frac{L}{C}$ and $R= h\left(\frac{L}{C}\right)$, and they are repelling between $R=1$ and $R= h\left(\frac{L}{C}\right)$, see the classification Figure \ref{fig7} in the parameter $(\frac{L}{C},R)$ plane for detail. By the Poincar\'{e} and Bendixson theorem, there exists at least a stable limit cycle on the subdomain (iv)$_+$. By the similar way as above, we can prove that  the limit measures stay away from  the saddle $\mathcal{S}$ and both repellers $\mathcal{A}$  and $\mathcal{B}$ on the subdomain (iv)$_+$, therefore, its support  must be contained in limit cycles. On the subdomain (iv)$_-$, the limiting measures stay away from the saddle $\mathcal{S}$. Thus, its support is contained in the union of limit cycles and stable equilibria $\mathcal{A}$  and $\mathcal{B}$. At this stage, we cannot give the exact limiting measure or its support due to lacking of the global dynamics for  (\ref{Unp1DVanDuffing}).

 We conjecture that system (\ref{Unp1DVanDuffing}) admits a unique stable limit cycle containing $\mathcal{S}$, $\mathcal{A}$  and $\mathcal{B}$ in its interior on the subdomain (iv)$_+$. If it does hold, then the limiting measure is the arc-length measure of the unique limit cycle. And we also conjecture that when $R$ passes through the Hopf bifurcating line from bottom to top, the system has exactly three limit cycles, two of which are unstable and surround $\mathcal{A}$  and $\mathcal{B}$, respectively; the biggest limit cycle of which is stable and surrounds three equilibria $\mathcal{S}$, $\mathcal{A}$  and $\mathcal{B}$. In this case, the support of limiting measure is the union of $\mathcal{A}$, $\mathcal{B}$ and the biggest limit cycle. We continue to conjecture that as $R$ continues to increase, there must be another bifurcating line, when $R$ passes through it from bottom to top, the phase plane is divided into two parts by the stable manifold of the saddle point $\mathcal{S}$, which are  attracting basins of $\mathcal{A}$  and $\mathcal{B}$, respectively. In this case, the limiting measure is the convex combination of Delta measures of $\mathcal{A}$ and $\mathcal{B}$.
\begin{figure}[h]
  \centering
  \includegraphics[width=0.6\textwidth]{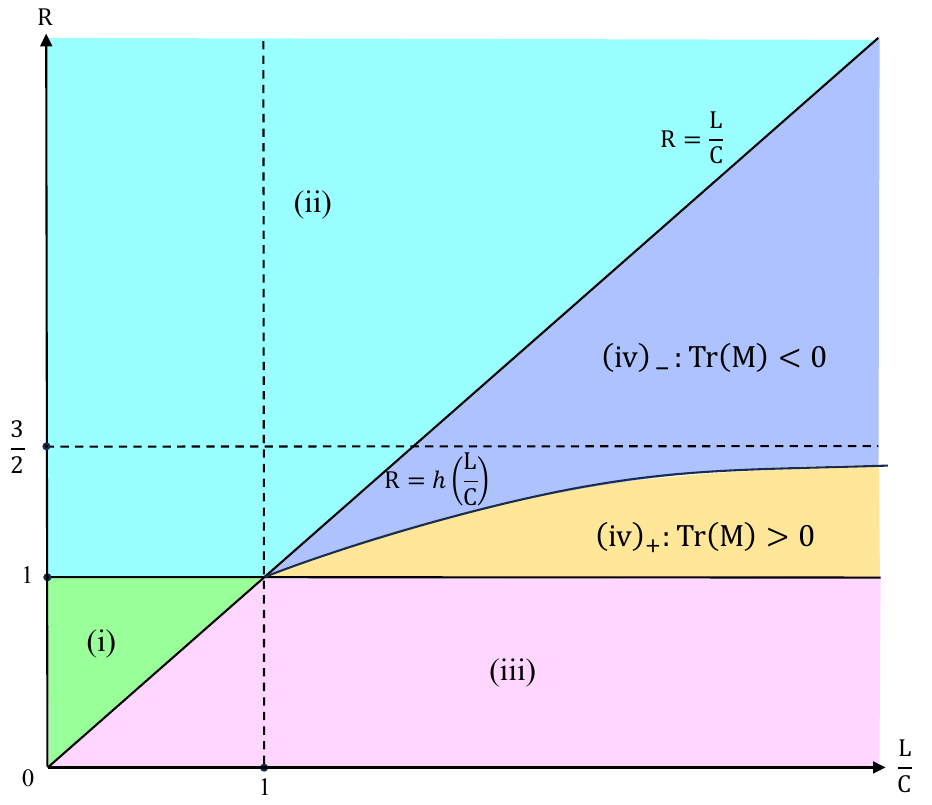}\\
  \caption{The classification of dynamical behavior of (\ref{Unp1DVanDuffing}) in $R, L, C$.}\label{fig7}
\end{figure}
\end{rmk}


\subsection {Random perturbation of May-Leonard model}
May and Leonard \cite{May} studied the following system in population biology
\begin{equation}\label{sys2}
    \begin{array}{l}
        \displaystyle \frac{d x_1}{dt}=x_1(1- x_1-\alpha x_{2}-\beta x_{3}),\\
        \noalign{\medskip}
        \displaystyle \frac{d x_2}{dt}=x_2(1-\beta x_1-x_{2}-\alpha x_{3}),\\
        \noalign{\medskip}
        \displaystyle \frac{d x_3}{dt}=x_3(1-\alpha x_1-\beta x_{2}-x_{3})
    \end{array}
\end{equation}
on the space $\mathbb{R}_+^3:=\{x:=(x_1,x_2,x_3)\in \mathbb{R}^3: x_i\ge 0, i=1,2,3\}$ and classified the long-term behavior of the system according to the sign of $\alpha+\beta-2$. We assume that the parameters satisfy $-1<\alpha+\beta<2$. The system (\ref{sys2}) has a positive equilibrium $E:=\frac{1}{1+\alpha+\beta}(1,1,1)=:(e_1,e_2,e_3)$.
Consider random perturbation of (\ref{sys2})
  \begin{equation}\label{SML}
    \begin{array}{l}
        \displaystyle d x_1=x_1\left(1- x_1-\alpha x_{2}-\beta x_{3}\right)dt +\sqrt{\varepsilon}x_1\sigma_1(x)dB_t^1,\\
        \noalign{\medskip}
        \displaystyle  \displaystyle d x_2=x_2\left(1-\beta x_1-x_{2}-\alpha x_{3}\right)dt +\sqrt{\varepsilon}x_2\sigma_2(x)dB_t^2,\\
        \noalign{\medskip}
        \displaystyle d x_3=x_3\left(1-\alpha x_1-\beta x_{2}-x_{3}\right)dt +\sqrt{\varepsilon}x_3\sigma_3(x)dB_t^3
    \end{array}
\end{equation}
where  $\sigma(x)={\rm Diag}(x_1\sigma_1(x),x_2\sigma_2(x),x_3\sigma_3(x))$ is defined on  $\mathbb{R}_+^3$ and $\{B(t)=(B_t^1,B_t^2,B_t^3)\}_{t \geq 0}$ is a $3$-dimensional Brownian motion.

\begin{prop}\label{MLeo}
Suppose that $\si(x)$ is a locally Lipschitz continuous function satisfying $0<\sum_{i=1}^3|\si_i(x)|^2 \leq c_1$ for all $x\in {\rm Int}\mathbb{R}_+^3$. If $-1<\alpha+\beta<2$, then the system {\rm(\ref{SML})} satisfies a uniform large deviation principle, and moreover  there exists $\varepsilon_0>0$  such that for any given $\varepsilon\in (0,\varepsilon_0]$, {\rm(\ref{SML})} admits a unique invariant measure $\mu^{\varepsilon}$ supported on ${\rm Int}\mathbb{R}_+^3$ and  $\mu^{\varepsilon}$  converges weakly to $\delta_E(\cdot)$ as $\varepsilon\rightarrow 0$.
\end{prop}
\begin{proof}
Let $b(x)$ be the drift vector field of the system (\ref{sys2}) and define the Lyapunov function
$$V(x)= \left(x_1-e_1-e_1\ln\frac{x_1}{e_1}\right)+\left(x_2-e_2-e_2\ln\frac{x_2}{e_2}\right)+\left(x_3-e_3-e_3\ln\frac{x_3}{e_3}\right).$$
Then
\begin{equation}\label{1}
\lim_{x\rightarrow \partial \mathbb{R}_+^3}V(x)=+\infty,
\end{equation}
$\nabla V(x)= \left(\frac{x_1-e_1}{x_1},\frac{x_2-e_2}{x_2},\frac{x_3-e_3}{x_3}\right)^*$, $\nabla^2V(x)={\rm Diag}\left(\frac{e_1}{x_1^2},\frac{e_2}{x_2^2},\frac{e_3}{x_3^2}\right)$ and
\begin{equation}\label{2}
\langle b(x), \nabla V(x)\rangle=-\sum_{i=1}^3(x_i-e_i)^2-(\alpha+\beta)\sum_{1\le i<j\le 3}(x_i-e_i)(x_j-e_j)\le -k|x-E|^2
\end{equation}
for some $k>0$ because the symmetric matrix
\begin{displaymath}
    \left(
        \begin{array}{cccc}
          & 1 & \frac{\alpha+\beta}{2} & \frac{\alpha+\beta}{2}\\
          & \frac{\alpha+\beta}{2} & 1 & \frac{\alpha+\beta}{2}  \\
          & \frac{\alpha+\beta}{2} & \frac{\alpha+\beta}{2} & 1 \\
        \end{array}
        \right )
\end{displaymath}
is positive definite following the assumption $-1<\alpha+\beta<2$. Note that the limit in (\ref{1}) is understood that there is at least $i$ such that either $x_i\rightarrow +\infty$ or $x_i\rightarrow 0$. (\ref{2}) implies that $E$ is a globally asymptotically stable equilibrium of (\ref{sys2}) on ${\rm Int}\mathbb{R}_+^3$. In addition,
\begin{equation}\label{3}
0<\text{Trace}\left(\sigma^*(x)\nabla^2V(x)\sigma(x)\right)=(1+\alpha+\beta)^{-1}\sum_{i=1}^3|\si_i(x)|^2\le c_1(1+\alpha+\beta)^{-1}
\end{equation}
and
$$|\sigma^*(x)\nabla V(x)|^2\le|x-E|^2\sum_{i=1}^3|\si_i(x)|^2\le c_1|x-E|^2.$$
Since $\lim_{x_i\rightarrow 0}\frac{x_i-e_i}{x_i-e_i-e_i\ln\frac{x_i}{e_i}}=0$ and $\lim_{x_i\rightarrow +\infty}\frac{x_i-e_i}{x_i-e_i-e_i\ln\frac{x_i}{e_i}}=1$, there exists a positive constant $c_2>0$ such that
\begin{equation}\label{4}
\frac{|x-E|^2}{V^2(x)}\le c_2\ {\rm for\ all}\ x\in {\rm Int}\mathbb{R}_+^3.
\end{equation}
\eqref{As b V} and \eqref{23} in the Assumption 2.2  follow from (\ref{1}) and (\ref{3}), respectively. Combining with (\ref{2}), (\ref{3}) and (\ref{4}) together, we see that \eqref{22} is satisfied. According to Remark \ref{rmk2.2} and Theorem 	\ref{thmULDP},  (\ref{SML}) admits ULDP.

Let
$$\mathcal{L}^{\varepsilon}V(x):=\langle b(x), \nabla V(x)\rangle+\frac{\varepsilon}{2}\text{Trace}\left(\sigma^*(x)\nabla^2V(x)\sigma(x)\right).$$
Then it follows from (\ref{2}) and (\ref{3}) that
\begin{equation}\label{5}
\mathcal{L}^{\varepsilon}V(x)\le -k|x-E|^2+\frac{\varepsilon c_1}{2(1+\alpha+\beta)}.
\end{equation}
Fix $\varepsilon\in (0,\frac{k}{c_1(1+\alpha+\beta)})$ and let $r(\varepsilon):=\sqrt{\frac{\varepsilon c_1}{2k(1+\alpha+\beta)}}$ and $B_r(E)$ denote the open ball with the center $E$ and the radius $r$. Then $B_{r(\varepsilon)}(E)\subset B_{(\sqrt{2}(1+\alpha+\beta))^{-1}}(E)\subset {\rm Int}\mathbb{R}_+^3$. Define $\gamma:=\frac{k}{4(1+\alpha+\beta)^2}$. Then it follows from (\ref{5}) that
\begin{equation}\label{6}
\mathcal{L}^{\varepsilon}V(x)\le -\gamma,\ x\in B^c_{(\sqrt{2}(1+\alpha+\beta))^{-1}}(E),\ \varepsilon\in \left(0,\frac{k}{2c_1(1+\alpha+\beta)}\right).
\end{equation}
By \cite[Theorem A]{Huang1} or \cite[Theorem 4.1]{KhasminskiiB}, (\ref{SML}) admits a unique invariant measure $\mu^{\varepsilon}$ for $0<\varepsilon<\frac{k}{2c_1(1+\alpha+\beta)}$. Applying \cite [Theorem 3.1]{Chen2020} to ${\rm Int}\mathbb{R}_+^3$, we obtain that $\{\mu^{\varepsilon}|0<\varepsilon<\frac{k}{2c_1(1+\alpha+\beta)}\}$ is tight and $\mu^{\varepsilon}$  converges weakly to $\delta_E(\cdot)$ as $\varepsilon\rightarrow 0$ because the equivalent class of (\ref{sys2}) on ${\rm Int}\mathbb{R}_+^3$ is the unique equilibrium $E$.
\end{proof}


\subsection {An unbounded example with infinite equivalent classes}
In this subsection, we will provide an example whose coefficients are unbounded and equivalent classes are infinite. Nevertheless,  we can determine its limiting measure by applying  Theorems \ref{thmULDP} and \ref{saddle}.

Let $H(x_1,x_2):= \frac{x_2^2}{2}+\frac{x_1^4}{4}-\frac{x_1^2}{2}\in [-\frac{1}{4}, +\infty)$. Consider the two dimensional deterministic system
\begin{equation}\label{DFigureEight}
\frac{dx}{dt}=\mathcal{H}(H)-F(H)\nabla H
\end{equation}
and its random perturbation
\begin{equation}\label{FigureEight}
dx=[\mathcal{H}(H)-F(H)\nabla H] dt+\sqrt{\varepsilon} \sigma(x) dB(t)
\end{equation}
where $\mathcal{H}(H):=(\frac{\partial H}{\partial x_2},-\frac{\partial H}{\partial x_1})^*$ and the $C^1$ function $F$ is taken to be one of the functions $F_i, i=1,...,4$ where
\[F_i(s)=\begin{cases}
(-1)^i|s|^3, \quad & s\in[-\frac{1}{4},0),\\
s^5\sin^2\frac{\pi}{s},\quad & s\in [0,1],\\
\rm positive\ smooth\ function, \quad & s\in (1,2),\\
1,\quad & s\in [2,+\infty),\\
\end{cases} \]
for $i=1,2$, or
\[F_j(s)=\begin{cases}
(-1)^j|s|^3, \quad & s\in[-\frac{1}{4},0),\\
G(s),\quad & s \in [0, +\infty),\\
\end{cases} \]
for $j=3,4$. Here $G:[0,+\infty)\rightarrow [0,1]$ is a $C^{\infty}$ function satisfying that $G(0)=0$ and there exist a sequence of intervals $I_n=[a_n,b_n]\subset (0,1]$ with $b_{n+1}<a_n$,  $n=1,2,3,\cdots$, $\lim_{n\rightarrow \infty}b_n=0$ and $b_1=1$ such that $G(s)=0$ for any $s\in I=\bigcup_{n=1}^{\infty}I_n$, $G(s)>0$ for any $s\in (0,+\infty)\backslash I$ and $G(s)\equiv 1$ for $s\gg 1$. Note that such a $G$ can be constructed using cut-off functions.

For these systems, we have the following result.
\begin{prop}
Suppose that $\sigma(x)=(\sigma_{i,j}(x_1,x_2))_{2\times 2}$ is a locally Lipschitz continuous function such that
\begin{equation}\label{UBC-1}
\beta^{*}\sigma(x)\sigma^{*}(x)\beta >0\ {\rm for\ all}\ x\in \mathbb{R}^2\ {\rm and}\ \beta \in \mathbb{R}^2\backslash \{0\}.
\end{equation}
and that there exists a positive constant $c_1$ such that for all $x\in \mathbb{R}^2$,
\begin{equation}\label{UBC}
\|\sigma(x)\|^2:=\sum_{i,j=1}^2\sigma_{i,j}^2(x_1,x_2)\le c_1(x_1^4+x_2^{\frac{4}{3}}+1).
\end{equation}
Then, for $i=1,3$,   $\mu^{\varepsilon}$  converges weakly to $\delta_O(\cdot)$ as $\varepsilon\rightarrow 0$, and for  $i=2,4$, $\mu^{\varepsilon}$  converges weakly to $\lambda_1\delta_{(1,0)}(\cdot)+\lambda_2\delta_{(-1,0)}(\cdot)$ with $\lambda_1+\lambda_2=1$ as $\varepsilon\rightarrow 0$.
\end{prop}
\begin{proof}
First, we prove that the system (\ref{FigureEight}) admits ULDP. It suffices to check that $V(x_1,x_2):= H(x_1,x_2)+\frac{1}{4}$ satisfies \eqref{As b V}, \eqref{22}  and \eqref{23} in the Assumption 2.2. \eqref{As b V} is obvious. We have
$$
\nabla V(x_1,x_2)=(x_1^3-x_1,x_2)^*,\ \nabla^2V(x_1,x_2)={\rm Diag}(3x_1^2-1,1);
$$
\begin{equation}\label{e1}
\langle b(x), \nabla V(x)\rangle=-F(H)|\nabla H|^2\ {\rm and}\ \langle b(x), \nabla V(x)\rangle=-|\nabla H(x)|^2\ {\rm if}\ |x|\gg 1 ;
\end{equation}
\begin{equation}\label{Epr2}
\text{Trace}\left(\sigma^*(x)\nabla^2V(x)\sigma(x)\right)=(3x_1^2-1)(\sigma_{1,1}^2(x)+\sigma_{1,2}^2(x))+\sigma_{2,1}^2(x)+\sigma_{2,2}^2(x);
\end{equation}
$$
|\sigma^*(x)\nabla V(x)|^2\le |\nabla H(x)|^2\|\sigma(x)\|^2.
$$

\DEQS
&&\limsup_{|x|\rightarrow +\infty}\frac{\text{Trace}\left(\sigma^*(x)\nabla^2V(x)\sigma(x)\right)}{|\nabla H(x)|^2} \\
&=& \limsup_{|x|\rightarrow +\infty}\frac{\text{Trace}\left(\sigma^*(x)\nabla^2V(x)\sigma(x)\right)}{x_1^6+x_2^2}\\
&\leq&  \limsup_{|x|\rightarrow +\infty}\frac{3x_1^2+1}{(x_1^6+x_2^2)^{\frac{1}{3}}}\times \frac{\sigma_{1,1}^2(x)+\sigma_{1,2}^2(x)}{(x_1^6+x_2^2)^{\frac{2}{3}}}+\limsup_{|x|\rightarrow +\infty} \frac{\sigma_{2,1}^2(x)+\sigma_{2,2}^2(x)}{x_1^6+x_2^2}\\
&\leq&  3c_1\limsup_{|x|\rightarrow +\infty} \frac{x_1^4+x_2^{\frac{4}{3}}+1}{(x_1^6+x_2^2)^{\frac{2}{3}}}\\
&\leq&  6c_1.
\EEQS
In the second inequality, we have used the condition (\ref{UBC}). This implies that there exists a positive constant $d_1$ such that
\begin{equation}\label{e2}
\text{Trace}\left(\sigma^*(x)\nabla^2V(x)\sigma(x)\right)\le d_1(|\nabla H(x)|^2+1)\ {\rm for\ all}\ x\in \mathbb{R}^2.
\end{equation}
Besides, the inequality
$$\limsup_{|x|\rightarrow +\infty}\frac{|\sigma^*(x)\nabla V(x)|^2}{|\nabla H(x)|^2V(x)}\le \limsup_{|x|\rightarrow +\infty}\frac{\|\sigma(x)\|^2}{V(x)}\le c_1\limsup_{|x|\rightarrow +\infty}\frac{x_1^4+x_2^{\frac{4}{3}}+1}{V(x)}\le 4c_1$$
 implies that there is a positive constant $d_2$ such that
\begin{equation}\label{e3}
\frac{|\sigma^*(x)\nabla V(x)|^2}{V(x)}\le d_2(|\nabla H(x)|^2+1)\ {\rm for\ all}\ x\in \mathbb{R}^2.
\end{equation}
Take $\theta=\frac{1}{2d_1}$ and $\eta= 4d_2$. Then it follows from (\ref{e1}), (\ref{e2}) and (\ref{e3}) that the left-hand side of \eqref{22} is negative if $|x|\gg 1$, which means that \eqref{22} holds. Finally, by (\ref{Epr2}) and (\ref{UBC}), we have
\DEQS
\text{Trace}\left(\sigma^*(x)\nabla^2V(x)\sigma(x)\right)
&\ge & -\|\sigma^*(x)\|^2\\
&\ge&  -c_1(x_1^4+x_2^{\frac{4}{3}}+1)\\
&\ge&  -c_1(x_1^4+x_2^2+2)\\
&\ge&  -4c_1-8c_1V(x),
\EEQS
that is, \eqref{23} holds. Thus, the system (\ref{FigureEight}) admits ULDP according to Theorem \ref{thmULDP}.

From (\ref{e1}) and (\ref{e2}) we can see that when $0<\varepsilon\le d_1^{-1}$,
$$\mathcal{L}^{\varepsilon}V(x)=\langle b(x), \nabla V(x)\rangle+\frac{\varepsilon}{2}\text{Trace}\left(\sigma^*(x)\nabla^2V(x)\sigma(x)\right)\le -\frac{1}{3}|\nabla H(x)|^2\,\, {\rm if}\ |x|\gg 1.$$
Therefore, there is a positive constant $k>0$ such that
$$\mathcal{L}^{\varepsilon}V(x)\le -k(x_1^6+x_2^2)\ {\rm if}\ |x|\gg 1.$$
\cite[Theorem A]{Huang1} or \cite[Theorem 4.1]{KhasminskiiB} implies that (\ref{FigureEight}) admits a unique invariant measure $\mu^{\varepsilon}$ for $0<\varepsilon\le\frac{1}{d_1}$ and \cite [Theorem 3.1]{Chen2020} implies that $\{\mu^{\varepsilon}|0<\varepsilon\le\frac{1}{d_1}\}$ is tight.

For $i=1$, by the definition of solution, $H^{-1}\big(n^{-1}\big)$ is a limit cycle of (\ref{DFigureEight}) for $n=1,2,3,\cdots$, which is asymptotically orbitally stable from the exterior and asymptotically orbitally unstable from the interior.  These limit cycles accumulate on the homoclinic cycle, a figure-eight curve. Using $H$ as a Lyapunov  function, we can prove that  limit set of $\Psi_t(x)$  lies on the zero set of $F(H)|\nabla H|^2$ (see Figure \ref{eg58i1}).
\begin{figure}[h]
  \centering
  \includegraphics[width=0.5\textwidth]{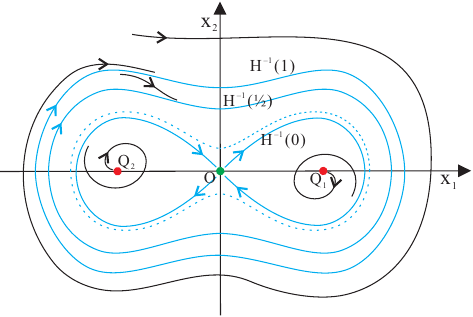}\\
  \caption{The phase portrait of system {\rm(\ref{DFigureEight})} for $i=1$.}\label{eg58i1}
\end{figure}
Thus $B(\Psi)=\bigcup_{n=1}^{+\infty}H^{-1}\big(n^{-1}\big)\cup H^{-1}(0)\cup\{(\pm 1,0)\}$. This shows that (\ref{DFigureEight}) possesses infinitely many equivalent classes because each $H^{-1}\big(n^{-1}\big)$ is a periodic orbit.

Let $\mu^{\varepsilon_k}$  converge weakly to $\mu$ as $\varepsilon_k\rightarrow 0$. Then $\mu$ is supported on $B(\Psi)$ by \cite [Theorem 3.1]{Chen2020}. Using $H+\frac{1}{4}$ as a Lyapunov function, we can verify that both $(1,0)$ and $(-1,0)$ are repellers. Theorem \ref{saddle} implies that $\mu$ is not supported on the equilibria $(\pm 1,0)$. Again using the Lyapunov method, we can prove that $H^{-1}\big([-4^{-1},n^{-1}]\big)$ is an attractor for each $n$. Utilizing Theorem \ref{saddle} to semistable limit cycles $\{H^{-1}\big(m^{-1}\big): \  m = 1, 2, 3, \cdots, n-1\}$ and the attractor $H^{-1}\big([-4^{-1},n^{-1}]\big)$, we conclude that $\mu$ is not supported on semistable limit cycles $\{H^{-1}\big(m^{-1}\big): \  m = 1, 2, 3, \cdots, n-1\}$. Since $n$ is arbitrary, we obtain that ${\rm supp}(\mu)\subset H^{-1}(0)$. Since $H^{-1}(0)$ consists of two homoclinic orbits connecting the origin $O$, it follows from the invariance of $\mu$ with respect to $\Psi$ that $\mu$ must be $\delta_O(\cdot)$. This proves that $\mu^{\varepsilon}$  converges weakly to $\delta_O(\cdot)$ as $\varepsilon\rightarrow 0$.

For $i=3$, the annular region $H^{-1}\big([a_n,b_n]\big)$ is full of nontrivial periodic orbits of {\rm(\ref{DFigureEight})} for $n=1,2,3,\cdots$, which is asymptotically orbitally stable from the exterior and asymptotically orbitally unstable from the interior.  These annular regions accumulate on the homoclinic cycle. Similarly, $B(\Psi)=\bigcup_{n=1}^{+\infty}H^{-1}\big([a_n,b_n]\big)\bigcup H^{-1}(0)\bigcup\{(\pm 1,0)\}$.

Let $\mu^{\varepsilon_k}$  converge weakly to $\mu$ as $\varepsilon_k\rightarrow 0$. Then $\mu$ is supported on $B(\Psi)$ by \cite [Theorem 3.1]{Chen2020}. In the same manner we can prove that there is no concentration of $\mu$ on the equilibria $(\pm 1,0)$ and that $H^{-1}\big([-4^{-1},b_n]\big)$ is an attractor for each $n$. By \cite [Proposition 5.5]{XCJ},   the annular region $H^{-1}\big([a_n,b_n]\big)$ is an equivalent class for each $n$. Applying Theorem \ref{saddle} to the annular regions
$\{H^{-1}\big([a_m,b_m]\big): \  m = 1, 2, 3, \cdots, n-1\}$ and the attractor $H^{-1}\big([-4^{-1},b_n]\big)$, we conclude that there is no concentration of $\mu$ on the annular regions $\{H^{-1}\big([a_m,b_m]\big): \  m = 1, 2, 3, \cdots, n-1\}$. Since $n$ is arbitrary, we obtain that ${\rm supp}(\mu)\subset H^{-1}(0)$ and $\mu^{\varepsilon}$  converges weakly to $\delta_O(\cdot)$ as $\varepsilon\rightarrow 0$.

Suppose that $i=2,4$. Then, using the same procedure we can show that there is no concentration of $\mu$ on  both $H^{-1}\big(n^{-1}\big)$ and $H^{-1}\big([a_n,b_n]\big)$ for each $n$. Again using $H+\frac{1}{4}$ as a Lyapunov function, we can verify that $(\pm 1,0)$ are attractors for both cases. Since every trajectory  of $\Psi$ in the interior of $H^{-1}(0)$ other than $(\pm 1,0)$ converges to $H^{-1}(0)$ as $t\rightarrow -\infty$ and to $(\pm 1,0)$ as $t\rightarrow +\infty$ (see Figure \ref{eg58i2}).
\begin{figure}[h]
  \centering
  \includegraphics[width=0.5\textwidth]{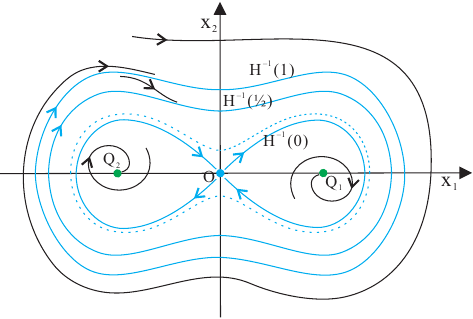}\\
  \caption{The phase portraits of system {\rm(\ref{DFigureEight})} for $i=2$.}\label{eg58i2}
\end{figure} Applying Theorem \ref{saddle} to $H^{-1}(0)$ and $(\pm 1,0)$, we get that $\mu\big(H^{-1}(0)\big)=0$. Finally, we must have ${\rm supp}(\mu)=\{(\pm 1,0)\}$, which proves that $\mu^{\varepsilon}$  converges weakly to $\lambda_1\delta_{(1,0)}(\cdot)+\lambda_2\delta_{(-1,0)}(\cdot)$ with $\lambda_1+\lambda_2=1$ as $\varepsilon\rightarrow 0$.
\end{proof}

\newpage

{\bf Acknowledgements} \\
 This work is partially supported by National Key R\&D
Program of China(No.2022YFA1006001), and by  NSFC (No. 12171321, 12371151, 12131019,
11721101), School Start-up Fund (USTC) KY0010000036, the Fundamental Research Funds for the Central Universities(USTC) (No. WK3470000031).

\end{document}